\documentclass[12pt,leqno,fleqn]{amsart} 
\usepackage[utf8]{inputenc}
\usepackage{amsthm}
\usepackage{amsmath,amstext,amsthm,amsxtra,amssymb}
\usepackage{hyperref}
\usepackage{mathtools}
\usepackage{cite}
\usepackage{amsfonts}
\usepackage{pgfplots}

\newtheorem{theorem}{Theorem}
\newtheorem{definition}{Definition}
\newtheorem{lemma}[theorem]{Lemma}
 \newtheorem{prop}[theorem]{Proposition}
 \newtheorem{corollary}[theorem]{Corollary}
 \newtheorem{conjecture}[theorem]{Conjecture}
 
 \usepackage{thmtools}
 \usepackage{thm-restate}

 \usepackage[margin=1in]{geometry}
 \usepackage{lipsum}

\usepackage{tikz}
\usetikzlibrary{arrows}

\newcommand\blfootnote[1]{%
  \begingroup
  \renewcommand\thefootnote{}\footnote{#1}%
  \addtocounter{footnote}{-1}%
  \endgroup
}

\newcommand{\Addresses}{
  \bigskip
  \footnotesize

  R.~Kesler, \textsc{School of Mathematics, Georgia Tech,
 Atlanta, Georgia 30313}\par\nopagebreak
   \textit{E-mail address}: rkesler6@math.gatech.edu

  \medskip

}

\begin{document}
\title[Estimates for Semi-Degenerate Multipliers]{$L^p$ Estimates for Semi-Degenerate Simplex Multipliers}
\author{Robert Kesler}

 \maketitle

\begin{abstract} 
Muscalu, Tao, and Thiele prove $L^p$ estimates for the ``Biest" operator defined on Schwartz functions by the map

\begin{align*}
\hspace{5mm} C^{1,1,1}:& (f_1, f_2, f_3) \mapsto \int_{\xi_1 < \xi_2< \xi_3} \left[ \prod_{j=1}^3 \hat{f}_j (\xi_j)   e^{2 \pi i x \xi_j } \right] d \vec{\xi}
\end{align*}
via a time-frequency argument that produces bounds for all multipliers with non-degenerate trilinear simplex symbols.  In this article we prove $L^p$ estimates for a pair of simplex multipliers for which the non-degeneracy condition fails and which are defined on Schwartz functions by the maps 
\begin{align*}
C^{1,1,-2}:& (f_1, f_2, f_3) \mapsto \int_{\xi_1 <\xi_2 < -\frac{\xi_3}{2}}\left[ \prod_{j=1}^3 \hat{f}_j (\xi_j)   e^{2 \pi i x \xi_j } \right] d \vec{\xi} \\ 
 C^{1,1,1,-2}:& (f_1, f_2, f_3, f_4) \mapsto \int_{\xi_1 <\xi_2 < \xi_3< -\frac{\xi_4}{2}} \left[\prod_{j=1}^4 \hat{f}_j (\xi_j) e^{2 \pi i x \xi_j} \right] d \vec{\xi}.
\end{align*}
 Our argument combines the standard $\ell^2$-based energy with an $\ell^1$-based energy in order to enable summability over various size parameters. As a consequence, we obtain that $C^{1,1,-2}$ maps into $L^p$ for all $1/2< p < \infty$ and $C^{1,1,1,-2}$ maps into $L^p$ for all $1/3 < p < \infty$. Both target $L^p$ ranges are shown to be sharp.  
\end{abstract}

 \section{Introduction}
Several recent articles have examined singular integral operators associated to simplexes from a time-frequency perspective. See, for example,  \cite{MR3596720,MR1491450,MR2221256,MR1887641,MR2127984, MR2127985}.  Such objects arise naturally in the asymptotic expansions of solutions to AKNS systems, where estimates of the form $\prod_{i=1}^n L^{p^\prime_i} (\mathbb{R}) \rightarrow L^{\frac{1}{\sum_{i=1}^n \frac{1}{p_i}}}(\mathbb{R})$ are sought for 
\begin{eqnarray*}
C_n: (f_1, ..., f_n) \mapsto  \sup_t \left| \int_{\xi_1 < ... <\xi_n<t}\left[ \prod_{j=1}^n f_j(\xi_j)e^{2 \pi i  x(-1)^j\xi_j}  \right] d \vec{\xi} \right|.
\end{eqnarray*} 
For details on the connection between the family of multisublinear operators $\left\{ C_n \right\}_{n \geq 1}$ and AKNS, see \cite{MR1809116}. 
It has also been of interest in the dynamics of particle systems to study the closely related family of fourier multipliers given for any $\vec{\alpha} \in \mathbb{R}^n$ and $\vec{f} \in \mathcal{S}^n(\mathbb{R})$ by the formula

\begin{eqnarray*}
C^{\vec{\alpha}}: (f_1, ..., f_n)  \mapsto \int_{\frac{\xi_1}{\alpha_1} < ... < \frac{\xi_n}{\alpha_n}} \left[ \prod_{j=1} ^n \hat{f}_j (\xi_j) e^{2 \pi i x \xi_j} \right] d\vec{\xi}. 
\end{eqnarray*}  A non-trivial example from the above simplex multiplier family is the ``Biest" operator $C^{1,1,1}$, which has been shown to satisfy a wide range of $L^p$ estimates via a robust time-frequency argument. More precisely, Muscalu, Tao, and Thiele prove the following result in \cite{MR2127985}:

\begin{theorem}\label{OT}
$C^{1,1,1} : L^{p_1}(\mathbb{R}) \times L^{p_2} (\mathbb{R}) \times L^{p_3}(\mathbb{R}) \rightarrow L^{p_4^\prime}(\mathbb{R})$ for all $(1/p_1, 1/p_2, 1/p_3, 1/p_4) \in \mathbb{D} \cap \mathbb{D}^\prime$, $1 < p_1, p_2, p_3 \leq \infty$ and $0 < p_4^\prime < \infty$, where $\mathbb{D}$ is the interior convex hull of $\{D_j\}_{j=1}^{12}$ given by 

\begin{align*}
D_1 =& \left(1, \frac{1}{2}, 1, - \frac{3}{2}\right) ,  D_2=\left(\frac{1}{2}, 1, 1, -\frac{3}{2}\right), D_3 = \left(\frac{1}{2}, 1, -\frac{3}{2}, 1\right) \\ 
D_4=& \left(1, \frac{1}{2} , -\frac{3}{2}, 1\right) 
, D_5 =\left(1 ,-\frac{1}{2}, 0, \frac{1}{2}\right), D_6 = \left(1, - \frac{1}{2}, \frac{1}{2}, 0\right) \\  D_7=&\left(\frac{1}{2}, - \frac{1}{2}, 0 , 1\right), D_8=\left(\frac{1}{2}, - \frac{1}{2}, 1, 0\right), 
D_9 = \left(-\frac{1}{2}, 1, 0, \frac{1}{2}\right)\\ D_{10}=&\left(-\frac{1}{2}, 1, \frac{1}{2}, 0\right), D_{11}=\left(-\frac{1}{2}, \frac{1}{2}, 1, 0\right), D_{12}=\left(-\frac{1}{2}, \frac{1}{2}, 0 , 1\right)
\end{align*}
and $\mathbb{D}^\prime$ is the interior convex hull of the collection $(D^\prime_1,..., D^\prime_{12})$ where each $D^\prime_j$ is gotten from the corresponding $D_j$ by swapping the $1$st and $3$rd positions. For instance, $D_2=(1, 1, \frac{1}{2}, - \frac{3}{2})$. For the dual index in positions $3$ or $4$, $C^{1,1,1}$ maps near $L^{2/5}(\mathbb{R})$, while in positions $1$ and $2$ it only map near $L^{2/3}(\mathbb{R})$.
\end{theorem}
For future use, we make the following definitions:

\begin{definition}
Let $m : \mathbb{R}^n \rightarrow \mathbb{C}$. Then define the multilinear multiplier $T_m$ on $\mathcal{S}^n(\mathbb{R})$ by

\begin{eqnarray*}
T_m : (f_1, ..., f_n) \mapsto \int_{\mathbb{R}^n} m(\vec{\xi}) \prod_{j=1}^n \left[ \hat{f}_j (\xi_j) e^{2 \pi i x \xi_j } \right]d \vec{\xi}.
\end{eqnarray*}

\end{definition} 
\begin{definition}
For every $\vec{\alpha} \in \mathbb{R}^n$, let $\tilde{C}^{\vec{\alpha}}$ denote the $n$-linear operator defined on $ \mathcal{S}^n(\mathbb{R})$ by

\begin{eqnarray*}
\tilde{C}^{\vec{\alpha}}(f_1,...,f_n)(x)= \int_{\xi_1 <  ... < \xi_n} \left[ \prod_{j=1}^n \hat{f}_j(\xi_j) e^{2 \pi i x \alpha_j \xi_j}\right] d\vec{\xi}.
\end{eqnarray*}

\end{definition}

\begin{definition}
For every $\vec{\alpha} \in \mathbb{R}^n$ with only non-zero entries, let $C^{\vec{\alpha}}$ denote the $n$-linear operator defined on $ \mathcal{S}^n(\mathbb{R})$ by 

\begin{eqnarray*}
C^{\vec{\alpha}}(f_1,...,f_n)(x)= \int_{\frac{\xi_1}{\alpha_1} < .... < \frac{\xi_n}{\alpha_n}} \left[ \prod_{j=1}^n \hat{f}_j(\xi_j) e^{2 \pi i x\xi_j}\right] d\vec{\xi}.
\end{eqnarray*}

\end{definition}
For every $\vec{\alpha} \in \mathbb{R}^n$ with only non-zero entries, we have by a simple change of variables
\begin{eqnarray*}
\tilde{C}^{\vec{\alpha}}(f_1, .., f_n)(x) = C^{\vec{\alpha}} (f_1(\alpha_1 \cdot), ..., f_n(\alpha_n \cdot))(x) \qquad \forall (f_1, .., f_n)\in \mathcal{S}^n(\mathbb{R})
\end{eqnarray*}
so that $C^{\vec{\alpha}}$ and $\tilde{C}^{\vec{\alpha}}$ satisfy the same $L^p$ estimates. 
\begin{definition}\label{Def:Degen}
Let $\vec{\alpha} \in \mathbb{R}^n$ satisfy the property that there exists a pair $(i,j)$ such that $1 \leq i \leq j \leq n, j-i \in \{0,1\}$, and $\sum_{k=i}^j \alpha_k =0 $. Then $\vec{\alpha}$ is a degenerate tuple and $C^{\vec{\alpha}}$ is a degenerate simplex multiplier. 
\end{definition}
\begin{definition}
Let $\vec{\alpha} \in \mathbb{R}^n$ satisfy the property that there exists no pair $1 \leq i \leq j \leq n$ such that $\sum_{k=i}^j \alpha_k =0$. Then $\vec{\alpha}$ is a non-degenerate tuple and $C^{\vec{\alpha}}$ is a non-degenerate simplex multiplier. 
\end{definition}
\begin{definition}
Let $\vec{\alpha} \in \mathbb{R}^n$ not be a degenerate tuple in the sense of definition \ref{Def:Degen}. Moreover, assume there exists a pair $(i,j)$ such that $1 \leq i ,  j \leq n$, $i < j-1$, and $\sum_{k=i}^j \alpha_k =0$. Then $\vec{\alpha}$ is a semi-degenerate tuple and $C^{\vec{\alpha}}$ is a semi-degenerate simplex multiplier. 
\end{definition}
The proof of Theorem \ref{OT} from \cite{MR2127985} involves splitting the symbol

\begin{align*}
1_{\{\xi_1 < \xi_2 < \xi_3\}} = \Psi_{\mathcal{R}_1} + \Psi_{\mathcal{R}_2} + \Psi_{\mathcal{R}_3} 
\end{align*}
where the three functions on the right side are localized to the respective regions 

\begin{align*}
\mathcal{R}_1 =& \left\{ |\xi_1- \xi_2| \gg |\xi_2 - \xi_3| \right\} \\  \mathcal{R}_2 =& \left\{ |\xi_1 - \xi_2| \simeq |\xi_2 - \xi_3| \right\}\\ \mathcal{R}_3 =& \left\{ |\xi_1 - \xi_2| \ll |\xi_2 - \xi_3| \right\}.
\end{align*}
 More precisely,  $\Psi_{\mathcal{R}_1}$ is supported on $\{\xi_1< \xi_2< \xi_3 : |\xi_1 - \xi_2| \geq C_1 |\xi_2 - \xi_3| \}$, identically equal to $1$ on $\{\xi_1 < \xi_2 < \xi_3: |\xi_1 - \xi_2| \geq C_2 |\xi_2 - \xi_3| \}$ for some constants $C_1 \ll C_2$, and has a special nested structure.  Similar statements hold for $\Psi_{\mathcal{R}_2}$ and $\Psi_{\mathcal{R}_3}$ on $\mathcal{R}_2$ and $\mathcal{R}_3$.

As the $L^p$ estimates in $\mathbb{D} \cap \mathbb{D}^\prime$ hold for $T_{\Psi_{\mathcal{R}_2}}$ by earlier work in \cite{MR1887641}, it suffices for Muscalu, Tao, and Thiele to estimate $T_{\Psi_{\mathcal{R}_1}}$ and $T_{\Psi_{\mathcal{R}_3}}$. Furthermore, by symmetry, it suffices for them to prove the estimates only for $T_{\Psi_{\mathcal{R}_1}}$. However, these estimates follow from the fact that $T_{\Psi_{\mathcal{R}_1}}$ can be written as an average of special non-degenerate model sums, which are described in greater detail in \S{4}, each non-degenerate model sum satisfies generalized restricted type estimates near the extremal points in $\mathbb{D}$ uniformly in the averaging parameters, and the generalized restricted type interpolation as shown in Chapter $3$ of \cite{MR2199086} finishes the argument. 

It is important to realize that for any non-degenerate $\vec{\alpha} \in \mathbb{R}^3 $ one can easily adapt the argument from \cite{MR2127985} to show that $C^{\vec{\alpha}}$ satisfies same $L^p$ estimates as $C^{1,1,1}$ does in Theorem \ref{OT}. In fact, we have from \cite{MR3329857} the following generalization. 

\begin{theorem}\label{CamT}
Fix $n \geq 1$ and let $\vec{\alpha} \in \mathbb{R}^n$ be non-degenerate. Then $C^{\vec{\alpha}}$ satisfies a wide range of $L^p$ estimates. 
\end{theorem}
While $C^{1,1,1}$ satisfies many $L^p$ estimates, Muscalu, Tao, and Thiele construct counterexamples in \cite{MR1981900} using Gaussian chirps that show no $L^p$ estimates hold for the degenerate simplex multiplier $C^{1,-1,1}$.  As the $4$-form associated to $C^{1,1,-2}$ cannot be written as an average of models of type $\Lambda_1$ and Gaussian chirps do not provide $C^{1,1,-2}$ counterexamples beyond those appearing in \S{3}, it is a natural question to ask which, if any, $L^p$ estimates hold in the semi-degenerate setting. One attractive feature of such simplex symbols is that they can be broken into simpler pieces, as illustrated by 

\begin{align*}
 &\left\{ \xi_1 < \xi_2 < -  \xi_3/2\right\} \\=& \left\{ \xi_1 +  \xi_2 < 2 \xi_2 < - \xi_3 \right\} \\=& \left\{ \xi_1 < \xi_2 \right\} \bigcap \left[ ( \left\{ -\xi_3 < \xi_1 + \xi_2 \right\} \cap \left\{ \xi_1 < \xi_2 \right\})\bigcup \{\xi_1 + \xi_2 \leq - \xi_3 \leq2 \xi_2\} \right]^c \\=& \left\{ \xi_1 < \xi_2 \right\} \bigcap  (\left\{  \xi_1 + \xi_2+\xi_3 >0 \right\}  \cap \{ \xi_1 < \xi_2\} )^c \\ &\bigcap \left(\{\xi_1 + \xi_2 +\xi_3 \leq 0 \} \cap \{ - \xi_3 \leq2 \xi_2\}  \right)^c.
\end{align*}
This elegant observation is due to Camil Muscalu. 
Using $H^+ = T_{\{ \xi >0\}}$ and $H^- = T_{\{\xi \leq 0\}}$, the above decomposition yields the identity

\begin{align}\label{Est:Iden}
 C^{1,1,-2}(f_1, f_2, f_3)(x)  =& C^{1,1}(f_1, f_2)(x) \cdot  f_3(x) \\-& H^+(C^{1,1}(f_1 ,f_2)\cdot  f_3)(x) \nonumber \\ -& H^-(f_1 \cdot C^{-2,1}(f_3, f_2))(x) \nonumber.
\end{align}
Because each term on the right side of the above display satisfies all interior Banach estimates, the same must be true for $\tilde{C}^{1,1,-1/2}$ and therefore $C^{1,1,-2}$. Given that $C^{1,1,-2}$ maps into $L^r(\mathbb{R})$ for all $1 < r < \infty$, it is tempting to ask whether such an object can map below $L^1(\mathbb{R})$, and if so, how low can the target exponent $r \geq \frac{1}{3}$ go. Our first result shows $r > 1/2$ is necessary for $C^{1,1,-2}$ to map into $L^r(\mathbb{R})$. Similarly, we have the identity

\begin{align}\label{Est:Iden2}
C^{1,1,1,-2}(f_1, f_2, f_3, f_4)(x) =&C^{1,1,1}(f_1, f_2, f_3)(x) f_4(x) \\-& T_{ \{\xi_1 < \xi_2 < \xi_3\} \cap \{ \xi_2 + \xi_3 + \xi_4 >0\}}(f_1, f_2, f_3, f_4)(x) \nonumber \\-&C^{1,1,-1}(f_1, f_2, C^{-2, 1}(f_4, f_3))(x) \nonumber.
\end{align}
Because both $T_{ \{\xi_1 < \xi_2 < \xi_3\} \cap \{ \xi_2 + \xi_3 + \xi_4 >0\}} $ and $C^{1,1,-1}(f_1, f_2, C^{-2, 1}(f_4, f_3))$ satisfy no $L^p$ estimates, identity \eqref{Est:Iden2} sheds little light on the boundedness properties of $C^{1,1,1,-2}$. A natural question in light of these facts is whether the degeneracy condition is necessary for $L^p$ estimates to fail. While not answering this question fully, we content ourselves in this section with establishing two main results. The first is
 
\begin{restatable}{theorem}{MT}\label{MT}
$C^{1,1,-2} : L^{p_1} (\mathbb{R}) \times L^{p_2}(\mathbb{R}) \times L^{p_3}(\mathbb{R}) \rightarrow L^{p_4^\prime}(\mathbb{R})$ for all $(\frac{1}{p_1}, \frac{1}{p_2}, \frac{1}{p_3},\frac{1}{p_4}) \in \mathbb{A} \cap \mathbb{A}^\prime$, $1< p_1, p_2, p_3 \leq \infty$ and $0 < p_4^\prime<\infty$, where $\mathbb{A}$ is the interior convex hull of  $\left\{A_j \right\}_{j=1}^9$ given by
\begin{align*}
A_1=& \left( 1, \frac{1}{2}, \frac{1}{2}, -1\right) , A_2 =\left( \frac{1}{2}, \frac{1}{2}, 1, -1\right), A_3= \left(\frac{1}{2}, 1, \frac{1}{2}, -1\right)  \\  A_4= &\left(-\frac{3}{2}, \frac{1}{2}, 1, 1 \right) , A_5 = \left(-\frac{3}{2}, 1, \frac{1}{2}, 1\right) , A_6=\left( \frac{1}{2}, -\frac{1}{2}, 1, 0\right)\\  A_ 7=& \left(0,  -\frac{1}{2}, 1, \frac{1}{2} \right), A_8= \left(0, 1, -\frac{1}{2}, \frac{1}{2} \right) , A_9=\left(\frac{1}{2}, 1, -\frac{1}{2}, 0\right)
\end{align*}
and $\mathbb{A}^\prime$ is the interior convex hull of the collection $\left\{ A_1^\prime,...A_9^\prime \right\}$ where each $A^\prime_j$ is gotten by the corresponding $A_j$ by swapping the $1$st and $3$rd indices. For example, $A_2^\prime = (1, \frac{1}{2}, \frac{1}{2}, -1)$. 

\end{restatable}

 To prove Theorem \ref{MT}, we follow the standard procedure introduced in \cite{MR2127985} of carving $1_{\xi_1 < \xi_2 < - \xi_3/2}$ into three localized pieces, discretizing each piece into a wave packet model, and then obtaining satisfactory estimates for the models. 
Our second main result is 
\begin{restatable}{theorem}{MTV}\label{MT*}
$C^{1,1,1,-2} : L^{p_1} (\mathbb{R}) \times L^{p_2}(\mathbb{R}) \times L^{p_3}(\mathbb{R}) \times L^{p_4}(\mathbb{R}) \rightarrow L^{p_5^\prime}(\mathbb{R})$ for all $(\frac{1}{p_1}, \frac{1}{p_2},\frac{1}{p_3}, \frac{1}{p_4}, \frac{1}{p_5}) \in \mathbb{B} \cap \mathbb{B}^\prime$, $1 < p_1, p_2, p_3, p_4 \leq \infty$ and $ 0 < p_5^\prime<\infty $, where $\mathbb{B}$ is the interior convex hull of $\left\{ B_j\right\}_{j=1}^{16}$ given by
\begin{align*}
B_1 =& \left( 1, 1, \frac{1}{2}, \frac{1}{2}, -2 \right) , B_2 = \left( 1, \frac{1}{2}, \frac{1}{2}, 1 , -2 \right) , B_3 = \left(1, \frac{1}{2}, 1, \frac{1}{2}, -2 \right) \\ B_4 =& \left( -2, 1, \frac{1}{2}, \frac{1}{2}, 1 \right) , B_5 = \left( -2, \frac{1}{2}, 1, \frac{1}{2}, 1 \right), B_6 = \left( -2, \frac{1}{2}, \frac{1}{2}, 1, 1 \right) \\ 
B_7 =& \left(0  , -\frac{3}{2}, \frac{1}{2}, 1, 1  \right), B_8= \left( 1, -\frac{3}{2}, \frac{1}{2}, 1, 0 \right), B_9 = \left( 0, - \frac{3}{2}, 1, \frac{1}{2}, 1 \right)  \\  B_{10} =& \left( 1 , -\frac{3}{2}, 1, \frac{1}{2}, 0 \right) , 
B_{11} = \left( 0 , \frac{1}{2} ,  -\frac{1}{2}  , 1  ,  0\right)  , B_{12} = \left( \frac{1}{2}, 0, -\frac{1}{2}, 1, 0 \right) \\ B_{13} =& \left( 0, 0, -\frac{1}{2}, 1, \frac{1}{2} \right) , 
B_{14} = \left( 0, \frac{1}{2}, 1, -\frac{1}{2}, 0 \right) , B_{15} = \left( \frac{1}{2}, 0, 1, -\frac{1}{2}, 0 \right) \\ B_{16} =& \left( 0,0, 1, - \frac{1}{2}, \frac{1}{2} \right)
\end{align*}
and $\mathbb{B}^\prime$ is the interior convex hull of the collection $\left\{ B^\prime_j \right\}_{j=1}^{16}$, where each $B^\prime_j$ is obtained from the corresponding $B_j$ by the permutation $1 \mapsto 1, 2 \mapsto 4, 3 \mapsto 2, 4 \mapsto 3$.  In particular, $B_3^\prime  = \left( 1, 1 , \frac{1}{2} , \frac{1}{2}, -2 \right)$. Moreover, $(1, \frac{2}{3}, \frac{2}{3}, \frac{2}{3}, -2)  \in \overline{ \mathbb{B} \cap \mathbb{B}^\prime}$ and $C^{1,1,1,-2}$ maps into $L^{r}(\mathbb{R})$ for all $\frac{1}{3} < r \leq 1$. 
\end{restatable}

A nice feature of our results is that the $L^p$ target ranges for both $C^{1,1,-2}$ and $C^{1,1,1,-2}$ are the best possible. Indeed, that $C^{1,1,-2}$ cannot map into $L^{\frac{1}{2}}(\mathbb{R})$ or below and $C^{1,1,1,-2}$ cannot map into $L^{\frac{1}{3}}(\mathbb{R})$ or below follows from explicit counterexamples included in \S{3}.  This sharpness is quite different from what is known in the non-degenerate setting, where the generic $BHT$ model produces estimates only down to $L^{\frac{2}{3}+\epsilon}(\mathbb{R})$, and there are no known counterexamples at this time to rule out the $BHT$ mapping all the way down to $L^{\frac{1}{2}+\epsilon}(\mathbb{R})$. 

The proofs of Theorems \ref{MT} and \ref{MT*} proceed by showing generalized restricted type estimates for various model sums near the extremal points in $\mathbb{A}, \mathbb{A}^\prime, \mathbb{B}$ and $\mathbb{B}^\prime$. The study of $L^p$ estimates for semi-degenerate simplex multipliers is motivated by the fact that the sizes and energies appearing in \cite{MR2127985} to deal with the non-degenerate ``Biest'' operator are no longer sufficient to produce summability over all the necessary time-frequency parameters. The main technical innovation in this paper is the introduction of an $\ell^1$-based energy to supplement the standard $\ell^2$-based energy. Checking that this energy ``boost" yields the desired restricted type estimates for various model sums requires some effort. 

Another indication of the delicacy present in the semi-degenerate setting is our observation in \eqref{Est:Iden2} that $C^{1,1,1,-2}$ can be naturally decomposed as the sum of one bounded operator and two unbounded operators; it is perhaps surprising that $C^{1,1,1-2}$ satisfies any $L^p$ estimates.   Theorem \ref{MT*} guarantees that there is substantial destructive interference between $T_{ \{\xi_1 < \xi_2 < \xi_3\} \cap \{ \xi_2 + \xi_3 + \xi_4 >0\}}(f_1, f_2, f_3, f_4)$ and $-C^{1,1,-1}(f_1, f_2, C^{-2, 1}(f_4, f_3))$, which would seem quite difficult to detect without time-frequency analysis. 

Theorems \ref{MT} and \ref{MT*} prompt other questions. For example, do we have the same $L^p$ estimates if the symbols $1_{\{\xi_1 < \xi_2 < - \xi_3/2\}}$ and $1_{\{ \xi_1 < \xi_2 < \xi_3 < - \xi_4/2\}}$ are respectively replaced with 

\begin{eqnarray*}
 b_1(\xi_1, \xi_2) \cdot b_2(\xi_2, \xi_3)~\text{and}~c_1(\xi_1, \xi_2)\cdot  c_2(\xi_3, \xi_3)\cdot  c_3(\xi_3, \xi_4), 
\end{eqnarray*}
where $b_1, c_1, c_2$ are each adapted to $\{\xi_1 = \xi_2\}$ 
 and $b_2, c_3$ are each adapted to $\{ \xi_1 = - \xi_2/2\}$ in the sense that for $\Gamma = \{\xi_1 = \xi_2\}$ and $\bar{\Gamma} = \{ \xi_1 = - \xi_2/2\}$

\begin{align*}
| \partial^{\vec{\alpha}} b_1 (\vec{\xi})| + | \partial^{\vec{\alpha}} c_1 (\vec{\xi})| +| \partial^{\vec{\alpha}} c_2 (\vec{\xi})|    \lesssim \frac{1}{dist(\vec{\xi}, \Gamma)^{|\vec{\alpha}|}} \qquad \forall \vec{\xi} \in \mathbb{R}^2 \\
| \partial^{\vec{\alpha}} b_2 (\vec{\xi})| +| \partial^{\vec{\alpha}} c_3 (\vec{\xi})| \lesssim \frac{1}{dist(\vec{\xi}, \bar{\Gamma})^{|\vec{\alpha}|}} \qquad \forall \vec{\xi} \in \mathbb{R}^2
\end{align*}
for sufficiently many multi-indices $\vec{\alpha} \in \mathbb{Z}_{\geq 0}^2$?
The answer to this question is probably yes; however, the proofs in the generic case become longer, less reader-friendly, and tend to obscure the important points of the analysis, and so the details of these arguments are omitted.  Nonetheless, we have all the tools necessary to carry out the proof and now provide the briefest possible sketch. Generic trilinear multipliers $m(\xi_1, \xi_2, \xi_3)$ with symbols of the form $b_1(\xi_1, \xi_2) \cdot b_2(\xi_2, \xi_3)$ may be reduced to ``Biest" models combined with error terms with even better mapping properties by following the arguments in \cite{2013arXiv1311.1574J}. Showing the same estimates for generic $4$-linear multipliers with symbols of the form $c_1(\xi_1, \xi_2) \cdot c_2(\xi_2, \xi_3)\cdot c_3(\xi_3, \xi_4)$ probably requires a local discretization similar to that used in \cite{2016arXiv160905964K} to handle
 
 \begin{eqnarray*}
 B[a_1, a_2] :(f_1, f_2, f_3) \mapsto \int_\mathbb{R} a_1(\xi_1, \xi_2) a_2(\xi_2, \xi_3) \hat{f}_1(\xi_1) \hat{f}_2(\xi_2) \hat{f}_3(\xi_3) e^{2 \pi i x (\xi_1 + \xi_2+ \xi_3)} d \xi_1 d \xi_2 d \xi_3
 \end{eqnarray*}
where $a_1, a_2: \mathbb{R}^2 \rightarrow \mathbb{C}$ are both adapted to the degenerate line $\left\{ \xi_1 + \xi_2= 0 \right\}$ in addition to the $\ell^1$-energy methods of this paper. Another question is which estimates, if any, hold for $C^{1,1,1,1,-2}$ or, for that matter, any $C^{\vec{\alpha}}$ with $\vec{\alpha} \in \mathbb{R}^n$ semi-degenerate. Suspecting that $C^{1,1,1,-2}$ already features a good deal of the pathological behavior exhibited by semi-degenerate simplex multipliers, we are led to state

\begin{conjecture}\label{C1}
For every semi-degenerate $\vec{\alpha} \in \mathbb{R}^n$, $C^{\vec{\alpha}}$ satisfies some $L^p$ estimates.
\end{conjecture}
The structure of this paper is as follows: \S{2} supplies from Lyons' work in \cite{MR1654527} a pointwise bound for simplex multipliers in terms of various powers of the variational Carleson and Bi-Carleson operators, \S{3} provides counterexamples showing that $C^{1,1,-2}$ cannot map below $L^{1/2}$ and $C^{1,1,1,-2}$ cannot map below $L^{1/3}$, \S{4} contains the proof of Theorem \ref{MT}, and \S{5} contains the proof of Theorem \ref{MT*}. 

 \section{Pointwise Domination by Variational Operators}
Closely related to the $a.e.$ convergence of the Fourier series of $L^p(\mathbb{R})$ functions is an important result of Carleson and Hunt, which says that the map initially defined on $ \mathcal{S}(\mathbb{R})$ given by

\begin{eqnarray*}
C: f \mapsto \sup_{N \in \mathbb{R}} \left| \int_{(\infty, N]} \hat{f}(\xi) e^{ 2\pi i x \xi}d \xi \right|
\end{eqnarray*}
can be extended to all of $L^p(\mathbb{R})$ and satisfies $||C(f)||_p \lesssim_p || f||_p$ for every $1 < p <\infty$ and $f \in L^p(\mathbb{R})$.  The variational Carleson estimates from \cite{MR2881301} are a generalization of this result: 
for any $2 < \rho \leq \infty$,

\begin{eqnarray*}
\mathcal{C}^\rho : f \mapsto \sup_{k \in \mathbb{N}} \sup_{\xi_1 < \xi_2 < ... < \xi_k} \left( \sum_{n=1}^{k-1} \left| \int_{\xi_n < \eta< \xi_{n+1}} \hat{f}(\eta)e^{2 \pi i x \eta } d\eta \right|^\rho \right)^{1/\rho} 
\end{eqnarray*}
extends to a map of $L^p(\mathbb{R}) \rightarrow L^p(\mathbb{R})$ for all $ \rho^\prime< p <\infty$. When $\rho = \infty$, we recover the Carleson operator estimates.  It is well known that $\rho >2$ is necessary for any $L^p$ estimates to hold. In light of the variational Carleson estimates, it is natural to ask whether estimates hold for the variational Bi-Carleson, which is defined for variation exponent $0<\rho\leq \infty$ and with domain $\mathcal{S}^2(\mathbb{R})$ to be
 
 \begin{eqnarray*}
 \mathcal{BC}^\rho: (f_1, f_2) \mapsto \sup_{k \in \mathbb{N}} \sup_{\xi_1 < \xi_2 < ... < \xi_k} \left( \sum_{n=1}^{k-1} \left| \int_{\xi_n < \eta_1 < \eta_2 < \xi_{n+1}} \hat{f}_1(\eta_1) \hat{f}_2 (\eta_2) e^{2 \pi i x (\eta_1 + \eta_2) } d\eta_1 d \eta_2 \right|^\rho \right)^{1/\rho}.
 \end{eqnarray*}
 If $\rho= \infty$, then $\mathcal{BC}^\infty$ is the Bi-Carleson operator, for which estimates were obtained in \cite{MR2221256} and shown to coincide with the known $BHT$ estimates $1 < p_1, p_2 \leq \infty$ and $ 0< \frac{1}{p_1} + \frac{1}{p_2} < \frac{3}{2}$.  More recently, Do, Muscalu, and Thiele prove in \cite{MR3596720}
 
 \begin{eqnarray*}
 \mathcal{BC}^{\rho} : L^{p_1}(\mathbb{R}) \times L^{p_2}(\mathbb{R}) \rightarrow L^{\frac{p_1 p_2}{p_1+ p_2}}(\mathbb{R})
 \end{eqnarray*}
 provided $\max\{ 1, \frac{2\rho}{3\rho-4} \} < p_1, p_2 \leq \infty, \max\{ \frac{2}{3}, \frac{\rho^\prime}{2} \} < p_3 < \infty$. 
 In particular, $\mathcal{BC}^{1+\epsilon} : L^2(\mathbb{R}) \times L^2 (\mathbb{R}) \rightarrow L^1(\mathbb{R})$ for every $\epsilon >0$.  We next present a striking inequality due to Lyons in \cite{MR1654527}, which provides a pointwise bound for trilinear simplex multipliers in terms of various powers of the variational Carleson and Bi-Carleson operators.  For all $2 < r <3$, we in fact have
 \begin{eqnarray}\label{Est:Var}
C^{1,1,1}(f_1, f_2, f_3)(x)& \leq& \left[Var^r (f_1, f_2, f_3 )(x)\right]^3
\end{eqnarray}
where 
\begin{align*}
Var^r(f_1, f_2, f_3)(x) :=& \mathcal{C}^r(f_1) (x) + \mathcal{C}^r(f_2)(x) + \mathcal{C}^r(f_3)(x) \\+& \left[ \mathcal{BC}^{r/2}(f_1, f_2)(x) \right]^{1/2} + \left[ \mathcal{BC}^{r/2} (f_2, f_3)(x) \right]^{1/2}+ \left[\mathcal{BC}^{r/2} (f_1, f_3)(x) \right]^{1/2}.
  \end{align*}
 Taking $r \simeq 3$ and $p_1= p_2 = p_3 \simeq 3/2$ and using the variational Carleson and variational Bi-Carleson estimates gives the extremal mapping $L^{3/(2-\epsilon)} \times L^{3/(2-\epsilon)} \times L^{3 /(2-\epsilon)} \rightarrow L^{1/(2-\epsilon)}$. By interpolation, one recovers all estimates in the convex hull of $ \mathcal{B} \cup \Delta$, where $\mathcal{B}$ denotes the set of all interior Banach estimates and $\Delta$ is a diagonal of quasi-Banach estimates, i.e. 
 
 \begin{eqnarray*}
 \mathcal{B} &=& \left\{ \left(\frac{1}{p_1}, \frac{1}{p_2}, \frac{1}{p_3}, 1-\frac{1}{p_1} - \frac{1}{p_2} - \frac{1}{p_3}\right): 1 < p_1, p_2, p_3 <\infty, \frac{1}{p_1} + \frac{1}{p_2} + \frac{1}{p_3} < 1 \right\} \\ 
 \Delta&=& \bigcup_{0< \epsilon< 1} \left(\frac{2-\epsilon}{3}, \frac{2-\epsilon}{3} , \frac{2-\epsilon}{3} , -1 +\epsilon \right). 
 \end{eqnarray*}
 Moreover, one can write down a similar pointwise bound for $C^{1,1,-2}$ that gives the same collection of estimates. Our proof of Theorem \ref{MT} has the two-fold advantage of avoiding the variational Carleson estimates and variational Bi-Carleson estimates and producing $L^p$ estimates beyond the convex hull of $\mathcal{B} \cup \Delta$.

\section{Counterexamples}
We begin with  
\begin{prop}\label{CntEx}
$C^{1,1,-2}$ does not map into $L^r(\mathbb{R})$ for $ r \leq 1/2$.
\end{prop}
\begin{proof}
It suffices to prove the claim for $\tilde{C}^{1,1,-2}$. Fix $1 <  p_1, p_2, p_3 \leq \infty$ with $\frac{1}{p_1} + \frac{1}{p_2} + \frac{1}{p_3} \geq 2$. Let $f_1=f_2=f_3 = \check{1}_{[-1,1]}$. It follows that $\prod^3_{j=1} \left| \left| f_j \right|\right|_{p_j} <\infty$ and 

\begin{align*}
& \tilde{C}^{1,1,-2} (f_1, f_2, f_3)(x) \\=& \int_{-1 < \xi_1 < \xi_2 < \xi_3 <1} e^{2 \pi i x ( \xi_1 + \xi_2 - 2\xi_3)} d\xi_1 d\xi_2 d\xi_3 \\ =& \frac{1}{2 \pi i x}\int_{-1 < \xi_2 < \xi_3 <1} \left[ e^{2 \pi i x(2\xi_2 - 2\xi_3)}- e^{2 \pi i x (-1 + \xi_2 - 2\xi_3)} \right]d\xi_2 d\xi_3  \\ =&\left[ \frac{1}{2 \pi i x}\right]^2 \int_{-1 < \xi_3 <1} \left[ \frac{1}{2}-e^{2 \pi i x(-1-\xi_3)}+ \frac{e^{2 \pi i x ( -2 - 2\xi_3)}}{2} \right] d\xi_3 \\ =&\left[ \frac{1}{2 \pi i x}\right]^2 -\left[ \frac{1}{2 \pi i x}\right]^3 \left[ \frac{3}{4}- e^{-4 \pi i x}+\frac{e^{-8 \pi i x}}{4} \right].
\end{align*}
Therefore, $\tilde{C}^{1,1,-2} (\vec{f})(x)$ decays like $\frac{-1}{4 \pi^2 x^2}$ far enough away from the origin and so cannot belong to $L^r(\mathbb{R})$ for $r \leq 1/2$. If $p_i=1$ for some $j \in \{1, 2, 3\}$, then one can instead take $f_1 =f_2 = f_3 =\mathcal{F}^{-1} \left[ \phi \right]$ for some non-trivial, non-negative $\phi  \in C^\infty([-1,1])$ and use integration by parts to deduce the same quadratic decay as before. 

\end{proof}
The analogous statement for $C^{1, 1,1,-2}$ is 
\begin{prop}
$C^{1,1,1,-2}$ does not map into $L^r(\mathbb{R})$ for $r \leq \frac{1}{3}$. 
\end{prop}
\begin{proof}
It suffices to prove the claim for $\tilde{C}^{1,1,1,-2}$. 
If $C^{1,1,1,-2}$ did map into $L^r(\mathbb{R})$ for some $r \leq \frac{1}{3}$, there would exist a $4$-tuple  $(p_1, p_2, p_3, p_4)$ satisfying $1 \leq  p_1 , p_2, p_3, p_4 \leq \infty$ and $\frac{1}{p_1} + \frac{1}{p_2} + \frac{1}{p_3} + \frac{1}{p_4} \geq  3$ for which 

\begin{eqnarray*}
\left| \left| C^{1,1,1,-2} (f_1, f_2, f_3, f_4) \right| \right|_{L^{\frac{1}{\frac{1}{p_1} + \frac{1}{p_2} + \frac{1}{p_3} + \frac{1}{p_4}}}(\mathbb{R})} \lesssim_{\vec{p}} || f_1||_{L^{p_1}(\mathbb{R})} || f_2||_{L^{p_2}(\mathbb{R})} || f_3||_{L^{p_3}(\mathbb{R})} ||f_4||_{L^{p_4}(\mathbb{R})} 
\end{eqnarray*}
for all $f_j \in L^{p_j}(\mathbb{R})$ and $j \in \{1, 2, 3, 4\}$. 

\vspace{5mm}
CASE 1: $\frac{1}{p_2} + \frac{1}{p_3} + \frac{1}{p_4} >2$. Then take $f_2 = f_3= f_4 = \mathcal{F}^{-1} [\phi ]$ along with $f^N_1 = \mathcal{F}^{-1} \left[  N \phi(N ( \cdot +2)) \right] = \mathcal{F}^{-1} \left[ \phi \right] (N^{-1}x) e^{-2 \pi i 2 x} $ where $\phi \in C^\infty([-1,1])$ is again some non-trivial, non-negative function. Then for large enough $N$, 

\begin{eqnarray*}
C^{1,1,1,-2} (f^N_1, f_2, f_3, f_4)  = f^N_1 (x) C^{1,1,-2}(f_2, f_3, f_4)(x),
\end{eqnarray*}
so that

\begin{eqnarray*}
\left| \left| C^{1,1,1,-2} (f^N_1, f_2, f_3, f_4) \right| \right|_{L^{\frac{1}{\frac{1}{p_1} + \frac{1}{p_2} + \frac{1}{p_3} + \frac{1}{p_4}}}(\mathbb{R})}  \simeq N^{\frac{1}{p_1} + \frac{1}{p_2} + \frac{1}{p_3} + \frac{1}{p_4}-2}. 
\end{eqnarray*}
However, $ || f^N _1||_{L^{p_1}(\mathbb{R})} \prod_{j=2}^4 || f_j || _{L^{p_j}(\mathbb{R})} \simeq N^{1/p_1}.$ Taking $N$ arbitrarily large proves the claim. 

\vspace{5mm}
CASE 2: $\frac{1}{p_2} + \frac{1}{p_3} + \frac{1}{p_4} =2$. Then $p_1 =1$. Setting 

\begin{eqnarray*}
f_1^N(x) =\mathcal{F}^{-1} \left[N \phi( N ( \cdot +2)) \right] (x) \mathcal{F}^{-1} \left[ 1_{[-1,1]} \right](x)
\end{eqnarray*}
for the same $\phi$ as before ensures that $C^{1,1,1,-2} (f^N_1, f_2, f_3, f_4)(x) = f^N_1 (x) C^{1,1,-2}(f_2, f_3, f_4)(x)$ for large enough $N$. Hence,

\begin{eqnarray*}
\left| \left| C^{1,1,1,-2} (f^N_1, f_2, f_3, f_4) \right| \right|_{L^{\frac{1}{\frac{1}{p_1} + \frac{1}{p_2} + \frac{1}{p_3} + \frac{1}{p_4}}}(\mathbb{R})} \simeq \left( \ln N \right)^3,
\end{eqnarray*}
whereas $ || f^N _1||_{L^{p_1}(\mathbb{R})} \prod_{j=2}^4 || f_j || _{L^{p_j}(\mathbb{R})} \simeq \ln N.$ Taking $N$ arbitrarily large again proves the claim and therefore shows the proposition.  

\end{proof}
\section{$C^{1,1,-2}$ Estimates}
Our goal in this section is to prove
\MT*

As $\mathbb{A} \cap \mathbb{A}^\prime$ strictly contains the estimates obtained by interpolating between the diagonal 

\begin{eqnarray*}
\Delta := \left\{ ( p, p, p, 1-3p) : 1/3 \leq  p <2/3 \right\}
\end{eqnarray*}
and the interior Banach estimates $\mathcal{B} = \left\{ \vec{p} : 1 < p_j \leq \infty~\forall~j \in \{1, 2, 3, 4\} \right\}$, Theorem \ref{MT} provides estimates not obtainable from estimate \eqref{Est:Var} and the variational Bi-Carleson estimates. For instance, $C^{1,1,-2} : L^{p_1} (\mathbb{R}) \times L^{p_2}(\mathbb{R}) \times L^{p_3}(\mathbb{R}) \rightarrow L^{p_4^\prime}(\mathbb{R})$ for tuples $(p_1, p_2, p_3, p_4)$ in a small neighborhood of $(1,\frac{1}{2}, \frac{1}{2}, -1)$.

Before proceeding to the proof of Theorem \ref{MT}, we need to collect several definitions. 
\subsection{Time-Frequency Definitions}

\begin{definition}
Let $n \geq 1$ and $\sigma \in \{ 0, \frac{1}{3}, \frac{2}{3} \}^n$. We define the shifted $n$-dyadic mesh $D= D^n_\sigma$ to be the collection of cubes of the form 

\begin{eqnarray*}
D^n_\sigma := \left\{ 2^j(k+(0,1)^n + (-1)^j \sigma) : j \in \mathbb{Z}, k \in \mathbb{Z}^n \right\}
\end{eqnarray*}

\end{definition}
Observe that for every cube $Q$, there exists a shifted dyadic cube $Q^\prime$ such that $Q \subseteq  \frac{7}{10} Q^\prime$ and $|Q^\prime| \sim |Q|$; this property clearly follows from verifying the $n=1$ case. The constant $\frac{7}{10}$ is not especially important here. 

\begin{definition}
A subset $D^\prime$ of a shifted $n$-dyadic grid $D$ is called sparse, if for any two cubes $Q, Q^\prime$ in $D$ with $Q \not = Q^\prime$ we have $|Q| < |Q^\prime|$ implies $|10^9 Q| < |Q^\prime|$ and $|Q|=|Q^\prime|$ implies $10^9 Q \cap 10^9 Q^\prime = \emptyset$.
\end{definition}
It is immediate from the above definition that any shifted $n$-dyadic grid can be split into $O(C^n)$ sparse subsets. 
\begin{definition}\label{TT}
Let $\sigma = (\sigma_1, \sigma_2, \sigma_3) \in \{0, \frac{1}{3}, \frac{2}{3} \}^3$, and let $1 \leq i \leq 3$. An $i$-tile with shift $\sigma_i$ is a pair $P = (I_P, \omega_P)$ with $|I_P| \times |\omega_P| =1$ and with $I_P \in D_0^1, \omega_P \in D^1_{\sigma_i}$. A tri-tile with shift $\sigma$ is a $3$-tuple $\vec{P} = (P_1, P_2, P_3)$ such that each $P_i$ is an $i$-tile with shift $\sigma_i$, and the $I_{P_i} = I_{\vec{P}}$ are independent of $i$. The frequency cube $\omega_{\vec{P}}$ of a tri-tile $\vec{P}$ is defined to be $\prod_{i=1}^3 \omega_{P_i}$. Define generalized tiles and tri-tiles to be the same as tiles and tri-tiles except that each frequency cube has edges belonging to $2^k D_0^1 + \eta$ for some $k \in [0,1]$ and $\eta \in \mathbb{R}$. 
\end{definition}
\begin{definition}

For each interval $I = \left[ c_I - \frac{|I|}{2}, c_I + \frac{|I|}{2} \right]$, let $\chi_I(x) = \frac{1}{ 1+  \frac{|x-c_I|}{|I|} }$. 
\end{definition}

\begin{definition}\label{WP}
Let $P= (I_P, \omega_P)$ be a pair of intervals for which $|I_P| \times |\omega_P| =1$. A wave packet on $P$ is any function $\Phi_P$ that has fourier support in $\frac{9}{10} \omega_P$ and obeys the estimate 

\begin{eqnarray*}
\left|\frac{d^k}{dx^k} \left[ e^{-2 \pi i c_{\omega_{P}} \cdot }\Phi_P \right] (x)\right| \leq C_k |I_{P}|^{-(1/2+k)} \chi^{L}_{I_{P}}(x) \qquad \forall x \in \mathbb{R}  ~\forall 0 \leq k \leq K
\end{eqnarray*}
where $c_{\omega_{P}}$ is the center of $\omega_P$ and $ K,L \gg 1$ are absolute constants we do not specify further. 
Therefore, $\Phi_P$ is $L^2$-normalized and localized to the Heisenberg box $(I_{P}, \omega_P )$.  
\end{definition}

\begin{definition}
A set $\mathbb{P}$ of tri-tiles is called sparse, if all the tri-tiles in $\mathbb{P}$ have the same shift $\sigma$ and the set of frequency cubes $\{Q_{\vec{P}}= (\omega_{P_1}, \omega_{P_2}, \omega_{P_3}): \vec{P} \in \mathbb{P} \}$ 
 is sparse. 
\end{definition}
We next introduce the tile ordering $<$ from \cite{MR2127985}, which is in the spirit of Lacey and Thiele.
\begin{definition}
Let $P$ and $P^\prime$ be tiles. We write $P^\prime < P$ if $I_{P^\prime} \subsetneq I_{P}$ and $3 \omega_P \subseteq 3 \omega_{P^\prime}$, and $P^\prime \leq P$ if $P^\prime <P $ or $P^\prime = P$. We write $P^\prime \lesssim P$ if $I_{P^\prime} \subseteq I_P$ and $10^7 \omega_P \subseteq 10^7 \omega_{P^\prime}$. We write $ P^\prime \lesssim^\prime P$ if $P^\prime \lesssim P$ and $P^\prime \not \leq P$. 
\end{definition}
\begin{definition}
A set $T$ of tri-tiles is a $j$-tree for some $j \in \{1,2,3\}$ provided there is a tile $P_T$ so that $P_j \leq P_{T}$ for all $\vec{P} \in T$. 
\end{definition}

\begin{definition}\label{R-1D}
A collection $\mathbb{P}$ of tri-tiles is said to have the rank-$1$ property if for all $\vec{P}, \vec{P}^\prime \in \mathbb{P}$: 

If $\vec{P} \not = \vec{P}^\prime$, then $P_j \not = P_j^\prime$ for all $j =1,2,3$. 

If $P_j^\prime \leq P_j$ for some $j= 1,2,3$, then $P_i^\prime \lesssim P_i$ for all $1 \leq i \leq 3$. 

If we further assume that $|I_{\vec{P}^\prime} |> 10^9 |I_{\vec{P}}|$, then $P_i^\prime \lesssim^\prime P_i$ for all $i \not = j$.

\end{definition}
\begin{definition}\label{Def:StronglyDisjoint}
Let $j \in \{1,2,3\}$. A finite sequence of trees $T_1, ..., T_M$ is said to be a chain of strongly $j$-disjoint trees if and only if 

\begin{align*}
& (i) ~P_j \not = P_j^\prime~\text{for every}~ P\in T_{\ell_1}~\text{and}~ P^\prime \in T_{\ell_2}~ \text{with}~ \ell_1 \not = \ell_2. \\ 
& (ii) ~\text{Whenever} ~P \in T_{\ell^1} ~\text{and}~ P^\prime \in T_{\ell^2}~ \text{with}~ \ell_1 \not = \ell_2 ~ \text{are such that} ~ 2 \omega_{P_i} \cap 2 \omega_{P_i^\prime}  \not = \emptyset \\ & \hspace{8mm} \text{then if} ~|\omega_{P_i}| < | \omega_{P_i^\prime}| ~\text{one has} ~ I_{P^\prime} \cap I_{T_{\ell^1}} = \emptyset \\ &\hspace{8mm} \text{and if} ~|\omega_{P^\prime_i}| < |\omega_{P_i}|~\text{one has}~ I_P \cap I_{T_{\ell_2}} = \emptyset .\\ 
& (iii) ~\text{Whenever}~ P\in T_{\ell_1} ~\text{and}~ P^\prime \in T_{\ell_2} ~ \text{with} ~\ell_1< \ell_2 ~\text{are such that}~2 \omega_{P_i} \cap 2 \omega_{P^\prime_i}  \not = \emptyset \\ & \hspace{8mm} \text{then if} ~ |\omega_{P_i}| = |\omega_{P^\prime_i}|~ \text{one has} ~I_{P^\prime} \cap I_{T_{\ell_1}} = \emptyset. 
\end{align*}
\end{definition}

\begin{definition}
For any two intervals $\omega_1 $and $\omega_2$, $\omega_1 \subset \subset \omega_2$ means $|\omega_1| \ll |\omega_2|$ for some absolute (and sufficiently small) implicit constant and $\omega_1 \subset \frac{9}{10} \omega_2.$

\end{definition}
\begin{definition}\label{Def:Adapted}
To say a tri-tile $\vec{P} = (P_1, P_2, P_3)$ is adapted to a subset of $\Gamma \subset \mathbb{R}^3$ means that the frequency cube $\vec{\omega}= (\omega_{P_1}, \omega_{P_2}, \omega_{P_3})$ satisfies the Whitney property with respect to $\Gamma$, i.e. 

\begin{eqnarray*}
dist(\vec{\omega}, \Gamma) \simeq |I_{\vec{P}}|^{-1}
\end{eqnarray*}
for some implicit absolute constants that we will not state explicitly.

\end{definition}
\begin{definition}
 For any collection of tri-tiles $\mathbb{Q}$, shifted dyadic interval $\omega$, and $\alpha \in [0,1]$ let

\begin{eqnarray*}
BHT_{\omega}^{\alpha, \mathbb{Q}} (f_1, f_2) (x) = \sum_{\vec{Q} \in \mathbb{Q} : \omega_{Q_3}  \subset \subset \omega_{P_2}}\frac{1}{|I_{\vec{Q}}|^{1/2}} \langle f_1, \Phi^\alpha_{Q_1,3} \rangle \langle f_2, \Phi^\alpha_{Q_2,4}  \rangle \Phi^\alpha_{Q_3,5}(x),
\end{eqnarray*}
where each $\Phi^\alpha_{Q_k, j}$ is a wave packet on $Q_k$ in accordance with Definition \ref{WP}. 
\end{definition}
\begin{definition}
A $\Lambda_1$-model is any 4-form writable as
\begin{align*}
\Lambda_1^{\mathbb{P}, \mathbb{Q}}(f_1, f_2, f_3, f_4) = \sum_{ \vec{P} \in \mathbb{P}}\frac{1}{|I_{\vec{P}}|^{1/2}} \langle f_1, \Phi_{P_1,1} \rangle \langle f_4 ,\Phi_{P_4,4} \rangle \left \langle BHT^{0, \mathbb{Q}}_{\omega_{P_2}}(f_2, f_3), \Phi_{P_2,0} \right\rangle, 
\end{align*}
where $\mathbb{P}$ and $\mathbb{Q}$ are rank-1 (non-degenerate) collections of tri-tiles. Tri-tiles are defined in Definition \ref{TT}, and the rank-1 condition is detailed in Definition \ref{R-1D}. Specifically, $\mathbb{P}$  is a collection of tri-tiles for which $(\omega_{P_1}, \omega_{P_2}, \omega_{P_4})$ is adapted to the non-degenerate line $\{\xi_1 =\xi_2/2 = - \xi_3/3 \}$ in the sense of Definition \ref{Def:Adapted}, and $\mathbb{Q}$ is a rank-1 collection of tri-tiles for which $(\omega_{Q_2}, \omega_{Q_2}, \omega_{Q_3})$ is adapted to the non-degenerate line $\{\xi_1 = -\xi_2/2 =- \xi_3\}$. Each $\Phi_{P_k, j}$ is a wave packet on the tile $P_k$ in accordance with Definition \ref{WP}. 
\end{definition}

\begin{definition}\label{Def:Lambda-2}
A $\Lambda_2$-model is any $4$-form writable as

\begin{align*}
\Lambda_2^{\mathbb{P}, \mathbb{Q}}(f_1, f_2, f_3, f_4)= \sum_{ \vec{P} \in \mathbb{P}}\frac{1}{|I_{\vec{P}}|^{1/2}} \langle f_1, \Phi_{P_1,1} \rangle \langle f_4, \Phi_{P_4,4} \rangle \left \langle  \int_0^1 BHT^{\alpha, \mathbb{Q}}_{\omega_{P_2}} (f_2, f_3) d\alpha , \Phi_{P_2,0} \right\rangle,
\end{align*}
where $\mathbb{P}$  is a collection of tri-tiles for which $\omega_{\vec{P}}=(\omega_{P_1}, \omega_{P_2}, \omega_{P_4})$ is adapted to $\{\xi_1 =-\xi_2; \xi_3=0\}$ in the sense of Definition \ref{Def:Adapted}, and $\mathbb{Q}$ is a rank-1 collection of tri-tiles for which $\omega_{\vec{Q}}=(\omega_{Q_1}, \omega_{Q_2}, \omega_{Q_3})$ is adapted to $\{ \xi_1=-\xi_2/2= -\xi_3 \}$. Each $\Phi_{P_k,j}$ is a wave packet on the tile $P_k$ in accordance with Definition \ref{WP}. 
\end{definition}

\begin{definition}\label{Def:Lambda-3}
A $\Lambda_3$-model is any $5$-form writable as

\begin{align*}
& \Lambda_3^{\mathbb{P}, \mathbb{Q}, \mathbb{R}}(f_1, f_2, f_3, f_4, f_5) \\ =& \sum_{\vec{P} \in \mathbb{P}}\frac{1}{|I_{\vec{P}}|^{1/2}}  \left \langle \sum_{\vec{R} \in \mathbb{R}: \widetilde{\omega_{R_1}}\ni (c_{\omega_{P_2}}-c_{\omega_{P_3}})/2 }  \frac{ \langle f_1, \Phi_{R_1, 1} \rangle \langle f_5, \Phi_{R_2, 5}\rangle}{|I_{\vec{R}}|^{1/2}} \Phi^{n-l}_{R_3, 0} , \Phi^{lac}_{P_1,6}\right  \rangle \\ & \hspace{20mm} \times \langle f_2, \Phi_{P_2, 2} \rangle \left\langle  \int_0^1  BHT_{\omega_{P_3}}^{\alpha, \mathbb{Q}} (f_3, f_4) d \alpha,  \Phi_{P_3, 7}\right \rangle 
\end{align*}
where $\mathbb{P}$ is a collection of tri-tiles for which $(\omega_{P_1}, \omega_{P_2}, \omega_{P_3})$ is adapted to $\{\xi_1=0, \xi_2 +\xi_3 =0\}$ in the sense of Definition \ref{Def:Adapted}, $\mathbb{R}$ is generalized tri-tile collection for which $\omega_{R_2}=-\omega_{R_1}$, $\widetilde{\omega_{R_1}}=\omega_{R_1}+|\omega_{R_1}|$, $\omega_{R_3} = [-|\omega_{R_3}|/2, |\omega_{R_3}|/2]$, and $\mathbb{Q}$ is a rank-$1$ collection of tri-tiles for which $(\omega_{Q_1}, \omega_{Q_2}, \omega_{Q_3})$ is adapted to $\{\xi_1=-\xi_2/2 =-\xi_3\}$. Each $\Phi_{P_k,j}$ is a wave packet on the tile $P_k$, and each $\Phi_{R_k,j}$ is a wave packet on the tile $R_k$ in accordance with Definition \ref{WP}. 
\end{definition}
\begin{definition}\label{Def:RestrType}
Let $\Lambda$ be an $n$-linear form and $\vec{\alpha}$ an admissible tuple. By an admissible tuple, we mean any $\vec{\alpha} \in \mathbb{R}^n$ for which $\sum_{j=1}^n \alpha_j=1, \alpha_i \leq 1$ for all $j \in \{1,...,n\}$ and there is at most one bad index $i \in \{1, ..., n\}$ for which $\alpha_i  \leq 0$. Then $\Lambda$ is generalized restricted type $\vec{\alpha}$ at an admissible tuple $\vec{\alpha}$ provided for any $n$-tuple $(E_1, ...,  E_n)$ of measurable subsets of $\mathbb{R}$ and $(f_1,...,f_n)$
satisfying $|f_j| \leq 1_{E_j}$ for $j=1,...,n$, then for the bad index $i$, if one exists, there is a major subset $E_i^\prime \subset E_i$  in the sense that $|E_i^\prime| \geq |E_i|/2$ such that the following inequality holds for $\left\{ f_j^\prime\right\}_{j=1}^n$ where $f_j^\prime: = f_j$ if $j \not = i$ and $f_i^\prime:= f_i1_{E_i^\prime}$:

\begin{eqnarray*}
\Lambda(f_1, ..., f_n) \lesssim_{\vec{\alpha}} |E_1|^{\alpha_1}...  |E_n|^{\alpha_n}. 
\end{eqnarray*}
\end{definition}
\subsection{Reduction to the $\Lambda_2$-Model}
The discretized and localized version of Theorem \ref{MT} is 

\begin{theorem}\label{DT}
Every $4$-form of type $\Lambda_2$ as described in Definition \ref{Def:Lambda-2} is generalized restricted type $\vec{\alpha}$ for all admissible tuples $\vec{\alpha}$ sufficiently close to the extremal points in $\mathbb{A}$.
If $\vec{\alpha}$ has a bad index $j$, the restricted type estimate is uniform in the sense that the major subset $E_j^\prime$ can be chosen uniformly in the parameters

\begin{eqnarray*}
\mathbb{P}, \mathbb{Q}, \left\{ \Phi_{P_k, j} \right\}, \left\{ \Phi^\alpha_{Q_k, j} \right\}. 
\end{eqnarray*}

\end{theorem} 
Before showing Theorem \ref{DT}, we prove the following
 \begin{prop}\label{MP}
 Theorem \ref{DT} implies Theorem \ref{MT}. 
 \end{prop}
 \begin{proof}

Our analysis of $C^{1,1,-2}$ begins as in the case of the ``Biest" in \cite{MR2127985} by localizing the symbol $1_{\xi_1 < \xi_2 < -\frac{\xi_3}{2}}$ inside the three regions: 

\begin{align*}
\mathcal{R}_1=& \left\{ \xi_1 < \xi_2 < - \xi_3/2 \right\} \cap \left\{ |\xi_1 - \xi_2| \ll |\xi_2 + \xi_3/2| \right\} \\ 
\mathcal{R}_2 =& \left\{ \xi_1 < \xi_2 < - \xi_3/2 \right\} \cap \left\{ |\xi_1 - \xi_2| \simeq |\xi_2 + \xi_3/2| \right\} \\
\mathcal{R}_3 =& \left\{ \xi_1 < \xi_2 < - \xi_3/2 \right\} \cap \left\{ |\xi_1 - \xi_2| \gg |\xi_2 + \xi_3/2| \right\}.
\end{align*}
To this end, let us recall from Section 6.1 in \cite{MR3052499} that on $\{\xi_1 + \xi_2 + \xi_3=0\} \subset \mathbb{R}^3$
\begin{eqnarray}\label{Def:Splitting1}
 1_{\{ \xi_2 < -\xi_3/2\}}(\xi_1,\xi_2, \xi_3) = \sum_{\vec{\sigma} \in \{ 0, \frac{1}{3}, \frac{2}{3} \}^3} \sum_{\vec{k} \in \mathbb{Z}^3} \sum_{\vec{\omega}_{\vec{Q}} \in \mathcal{Q}^{\vec{\sigma}}} c_{k_1,1} c_{k_2, 2} c_{k_3, 3} \cdot \hat{\eta}_{-\omega_{Q_3}, 5}^{\sigma_1, k_1}(\xi_1) \hat{\eta}_{\omega_{Q_1}, 2}^{\sigma_2, k_2} (\xi_2) \hat{\eta}_{\omega_{Q_2}, 3}^{\sigma_3 , k_3} (\xi_3)
\end{eqnarray}
where $\sigma_1, \sigma_2, \sigma_3$ are dyadic shifts, $k_1,k_2, k_3$ are oscillation parameters, $\omega_{\vec{Q}}=(\omega_{Q_1},\omega_{Q_2}, \omega_{Q_3})$ is a Whitney cube for the set $\Gamma := \{\xi_1=-\xi_2 /2=  -\xi_3\}$ in the usual sense that the side-length of $\omega_{\vec{Q}}$ is proportional to  $dist(\omega_{\vec{Q}}, \Gamma)$, and $supp~ \hat{\eta}^{\sigma_j,k_j}_{\omega_{Q_j},j} \subset \frac{9}{10} \omega_{Q_j}$ for all $j \in \{1,2,3\}$. Another important property of this decomposition is the decay valid for all $N \geq 1$, $k \in \mathbb{Z}$, and $j \in \{1,2,3\}$

\begin{eqnarray*}
|c_{k,j}| \lesssim_N \frac{1}{1+k^N}.
\end{eqnarray*}
Similarly, we have on $\{ \xi_1 + \xi_2 + \xi_3=0 \}$

\begin{eqnarray}\label{Def:Splitting2}
 1_{\{ \xi_1 < -\xi_2\}}(\xi_1, \xi_2, \xi_3) = \sum_{ \vec{\gamma} \in \{ 0, \frac{1}{3}, \frac{2}{3} \}^3} \sum_{\vec{l} \in \mathbb{Z}^3} \sum_{\omega_{\vec{P}} \in \mathcal{P}^{\vec{\gamma}}} d_{l_1, 1} d_{l_2, 2} d_{l_3,3} \cdot \hat{\eta}_{\omega_{P_1}, 1}^{\gamma_1, k_1} (\xi_1) \hat{\eta}_{\omega_{P_2}, 0}^{\gamma_2 , k_2} (\xi_2)  \hat{\eta}^{\gamma_3, k_3}_{\omega_{P_4},4} (\xi_3)
\end{eqnarray}
where each $\omega_{\vec{P}} = (\omega_{P_1}, \omega_{P_2}, \omega_{P_4})$ is a Whitney cube for the set $\tilde{\Gamma} = \left\{ \xi_1 = -\xi_2; \xi_3=0 \right\}$. 
As before, an important property is the decay valid for all $N \geq 1$, $k \in \mathbb{Z}$, and $j \in \{1,2,3\}$

\begin{eqnarray*}
|d_{k,j}| \lesssim_N \frac{1}{1+k^N}.
\end{eqnarray*}
The main trick we use now is that inside $\mathcal{R}_3$, $\xi_1 < \xi_2 < - \xi_3/2$ holds iff $\xi_1 < -(\xi_2 + \xi_3) ; \xi_2 < - \xi_3/2$ holds. Therefore, setting $c_{\vec{k}}=\prod_{j=1}^3 c_{k_j,j}$, $d_{\vec{l}}=\prod_{j=1}^3 d_{l_j,j}$, and

\begin{align*}
\tilde{\phi}_{\mathcal{R}_3}(\xi_1, \xi_2, \xi_3, \xi_4)  =&  
   \sum_{\vec{\sigma} \in \{ 0, \frac{1}{3}, \frac{2}{3} \}^3} \sum_{\vec{k} \in \mathbb{Z}^3} ~\sum_{\omega_{\vec{Q}} \in \mathcal{Q}^{\vec{\sigma}}} c_{\vec{k}} \hat{\eta}_{-\omega_{Q_3}, 5}^{\sigma_1, k_1}(\xi_1+\xi_4) \hat{\eta}_{\omega_{Q_1}, 2}^{\sigma_2, k_2} (\xi_2) \hat{\eta}_{\omega_{Q_2}, 3}^{\sigma_3 , k_3} (\xi_3) \\ \times& \left[ \sum_{\vec{\gamma} \in \{ 0, \frac{1}{3}, \frac{2}{3} \}^2} \sum_{\vec{l} \in \mathbb{Z}^3} ~\sum_{\omega_{\vec{P}} \in \mathcal{P}^{\vec{\gamma}}: | \omega_{\vec{P}}| \gg |\omega_{\vec{Q}}|} d_{\vec{l}} ~\hat{\eta}_{\omega_{P_1}, 1}^{\gamma_1, l_1} (\xi_1) \hat{\eta}_{\omega_{P_2}, 0}^{\gamma_2 , l_2} (\xi_2+\xi_3)  \hat{\eta}^{\gamma_3, l_3}_{\omega_{P_4},4} (\xi_4) \right],
\end{align*}
it follows that there are two constants $C_1$ and $C_2$ such that $\tilde{\phi}_{\mathcal{R}_3}(\xi_1, \xi_2, \xi_3,\xi_4) \equiv 1$ on the set $\left\{ \xi_1 < \xi_2 < - \xi_3/2 \right\} \cap \left\{ |\xi_1 - \xi_2| \geq C_1 |\xi_2 + \xi_3/2| \right\} \cap \left\{ \xi_1 + \xi_2 + \xi_3 + \xi_4 =0\right\}$ and is supported on $\left\{ \xi_1 < \xi_2 < - \xi_3/2 \right\} \cap \left\{ |\xi_1 - \xi_2| \geq C_2 |\xi_2 + \xi_3/2| \right\} \cap \left\{ \xi_1 + \xi_2 + \xi_3 + \xi_4=0 \right\}.$ We may similarly construct $\tilde{\phi}_{\mathcal{R}_1}$: on $\{\xi_1 + \xi_2 + \xi_3=0\} \subset \mathbb{R}^3$
\begin{eqnarray*}
 1_{\{ \xi_1 < \xi_2\}}(\xi_1,\xi_2, \xi_3) = \sum_{\vec{\sigma}\in \{ 0, \frac{1}{3}, \frac{2}{3} \}^3} \sum_{\vec{k} \in \mathbb{Z}^3} \sum_{\omega_{\vec{Q}} \in \mathcal{Q}^{\vec{\sigma}^\prime}} c^\prime_{k_1,1} c_{k_2, 2} c_{k_3, 3} \cdot \hat{\eta}_{-\omega_{Q_3}, 5}^{\prime,\sigma_1, k_1}(\xi_1) \hat{\eta}_{\omega_{Q_1}, 2}^{\prime,\sigma_2, k_2} (\xi_2) \hat{\eta}_{\omega_{Q_2}, 3}^{\prime,\sigma_3 , k_3} (\xi_3)
\end{eqnarray*}
where $\sigma_1, \sigma_2, \sigma_3$ are again dyadic shifts, $k_1,k_2, k_3$ are oscillation parameters, each $\vec{Q}=(\omega_{Q_1},\omega_{Q_2}, \omega_{Q_3})$ is a Whitney cube for the set $\Gamma := \{\xi_1=\xi_2=  -\xi_3/2\}$ in that the side-length of $\omega_{\vec{Q}}$ is proportional to  $dist(\omega_{\vec{Q}}, \Gamma)$, and $supp~ \hat{\eta}^{\prime, \sigma_j, k_j}_{\omega_{Q_j},j} \subset \frac{9}{10} \omega_{Q_j}$ for all $j \in \{1,2,3\}$. An important property of this decomposition is the decay valid for all $N \geq 1$, $k \in \mathbb{Z}$, and $j \in \{1,2,3\}$

\begin{eqnarray*}
|c^\prime_{k,j}| \lesssim_N \frac{1}{1+k^N}.
\end{eqnarray*}
The main trick now is that inside $\mathcal{R}_1$, $\xi_1 < \xi_2 < - \xi_3/2$ holds iff $(\xi_1 +\xi_2)< - \xi_3 ; \xi_1 <  \xi_2$ holds. Therefore, setting $c^\prime_{\vec{k}}=\prod_{j=1}^3 c^\prime_{k_j,j}$, recalling $d_{\vec{l}}=\prod_{j=1}^3 d_{l_j,j}$, and letting

\begin{align*}
\tilde{\phi}_{\mathcal{R}_1}(\xi_1, \xi_2, \xi_3, \xi_4)  =&  
   \sum_{\vec{\sigma} \in \{ 0, \frac{1}{3}, \frac{2}{3} \}^3} \sum_{\vec{k} \in \mathbb{Z}^3} ~\sum_{\omega_{\vec{Q}} \in \mathcal{Q}^{\vec{\sigma}}} c^\prime_{\vec{k}} \hat{\eta}_{-\omega_{Q_3}, 5}^{\prime,\sigma_1, k_1}(\xi_3+\xi_4) \hat{\eta}_{\omega_{Q_1}, 2}^{\prime,\sigma_2, k_2} (\xi_1) \hat{\eta}_{\omega_{Q_2}, 3}^{\prime,\sigma_3 , k_3} (\xi_2) \\ \times& \left[ \sum_{\vec{\gamma} \in \{ 0, \frac{1}{3}, \frac{2}{3} \}^2} \sum_{\vec{l} \in \mathbb{Z}^3} ~\sum_{\omega_{\vec{P}} \in \mathcal{P}^{\vec{\gamma}}: | \omega_{\vec{P}}| >> |\omega_{\vec{Q}}|} d_{\vec{l}} ~\hat{\eta}_{\omega_{P_1}, 1}^{\gamma_1, l_1} (\xi_1+\xi_2) \hat{\eta}_{\omega_{P_2}, 0}^{\gamma_2 , l_2} (\xi_3)  \hat{\eta}^{\gamma_3, l_3}_{\omega_{P_4},4} (\xi_4) \right],
\end{align*}
it follows that there are two constants $C_1$ and $C_2$ such that $\tilde{\phi}_{\mathcal{R}_3}(\xi_1, \xi_2, \xi_3,\xi_4) \equiv 1$ on the set $\left\{ \xi_1 < \xi_2 < - \xi_3/2 \right\} \cap \left\{ |\xi_1 - \xi_2| \leq C_1 |\xi_2 + \xi_3/2| \right\} \cap \left\{ \xi_1 + \xi_2 + \xi_3 + \xi_4 =0\right\}$ and is supported on $\left\{ \xi_1 < \xi_2 < - \xi_3/2 \right\} \cap \left\{ |\xi_1 - \xi_2| \leq C_2 |\xi_2 + \xi_3/2| \right\} \cap \left\{ \xi_1 + \xi_2 + \xi_3 + \xi_4=0 \right\}.$ Now set 

\begin{eqnarray*}
\phi_{\mathcal{R}_1}(\xi_1, \xi_2, \xi_3) := \tilde{\phi}_{\mathcal{R}_1} (\xi_1, \xi_2, \xi_3, -\xi_1-\xi_2-\xi_3) \\ 
\phi_{\mathcal{R}_3}(\xi_1, \xi_2, \xi_3) := \tilde{\phi}_{\mathcal{R}_3} (\xi_1, \xi_2, \xi_3, -\xi_1-\xi_2-\xi_3)
\end{eqnarray*}
and observe the identity

\begin{eqnarray*}
&&1_{\xi_1 < \xi_2 < - \xi_3/2} \\ &=&  1_{\xi_1 < \xi_2 < - \xi_3/2} (1- \phi_{\mathcal{R}_1})(1-\phi_{\mathcal{R}_3}) + 1_{\xi_1 < \xi_2 < - \xi_3/2} \phi_{\mathcal{R}_1} + 1_{\xi_1 < \xi_2 < - \xi_3/2} \phi_{\mathcal{R}_3} - 1_{\xi_1 < \xi_2 < - \xi_3/2} \phi_{\mathcal{R}_1} \phi_{\mathcal{R}_3}  \\&:=& I + II + III + IV. 
\end{eqnarray*} 
Letting $\mathcal{R}_2 = \{ \xi_1 = \xi_2 = - \xi_3/2\}  \subset \mathbb{R}^3$,  it is straightforward to observe that $\phi_{\mathcal{R}_2}:= I$ is a Mikhlin-H\"{o}rmander symbol adapted to the region $\mathcal{R}_2$, i.e. 

\begin{eqnarray*}
|\partial^{\vec{\alpha}} \phi_{\mathcal{R}_2} (\vec{\xi}) | \leq C_{\vec{\alpha}} \frac{1}{dist(\vec{\xi}, \mathcal{R}_2)^{|\vec{\alpha}|}}
\end{eqnarray*}
for sufficiently many multi-indices $\vec{\alpha} \in \mathbb{Z}_{\geq 0}^3$; moreover, $IV\equiv 0$ for large enough implicit constants governing the separation of scales between the frequency cubes in $\mathcal{Q}^{\vec{\sigma}}$ and $\mathcal{P}^{\vec{\gamma}}$ in the definition of $\tilde{\phi}_{\mathcal{R}_3}$ and $\tilde{\phi}_{\mathcal{R}_1}$. We handle terms $II$ and $III$ by first noting  

\begin{eqnarray*}
1_{\xi_1 < \xi_2 < - \xi_3/2} \phi_{\mathcal{R}_1} = \phi_{\mathcal{R}_1} \\ 
1_{\xi_1 < \xi_2 < - \xi_3/2} \phi_{\mathcal{R}_3} = \phi_{\mathcal{R}_3}
\end{eqnarray*}
and then proving the desired $L^p$ estimate for $T_{\phi_{\mathcal{R}_1}}$ and $T_{\phi_{\mathcal{R}_3}}$. However, by symmetry, it suffices to obtain estimates for $T_{\phi_{\mathcal{R}_3}}$. To this end, we dualize by introducing $f_4$ as follows: 

\begin{eqnarray*}
&& \int_\mathbb{R}  T_{\phi_{\mathcal{R}_3}}(f_1, f_2, f_3) (x)  f_4(x) dx \\&=&\sum_\prime c_{\vec{k}}  d_{\vec{l}}  \int_\mathbb{R}  f_1*\eta ^{\gamma_1, l_1}_{\omega_{P_1}, 1} \cdot  f_4 *\eta^{\gamma_3, l_3}_{\omega_{P_4},4}  \cdot \left[ f_2*\eta^{\sigma_1, k_1}_{\omega_{Q_1}, 2}  f_3*\eta^{\sigma_2, k_2}_{\omega_{Q_2}, 3} \right] * \eta^{\sigma_3, k_3}_{\omega_{Q_3}, 5} *\eta^{\gamma_2, l_2}_{\omega_{P_2}, 0}  dx,
\end{eqnarray*}
where $\sum_\prime =\sum_{\vec{\sigma} \in \{ 0, \frac{1}{3}, \frac{2}{3} \}^3} \sum_{\vec{\gamma} \in \{ 0, \frac{1}{3}, \frac{2}{3} \}^3} \sum_{\vec{k} \in \mathbb{Z}^3} \sum_{\vec{l} \in \mathbb{Z}^3}  ~\sum_{\vec{Q} \in \mathcal{Q}^{\vec{\sigma}}}  ~\sum_{\vec{P} \in \mathcal{P}^{\vec{\gamma}}: |\omega_{\vec{P}}| \gg |\omega_{\vec{Q}}|}$.
We may now discretize in time with respect to the $\mathcal{Q}$ and $\mathcal{P}$ Whitney cubes. The details required for this process are well-established and discussed in complete detail in Section 6.1 of \cite{MR3052499}. This procedure will yield that

\begin{eqnarray*}
&& \sum_\prime c_{\vec{k}}  d_{\vec{l}}  \int_\mathbb{R}  f_1*\eta ^{\gamma_1, l_1}_{\omega_{P_1}, 1} \cdot f_4 *\eta^{\gamma_3, l_3}_{\omega_{P_4},4} \cdot \left[ f_2*\eta^{\sigma_1, k_1}_{\omega_{Q_1}, 2}  f_3*\eta^{\sigma_2, k_2}_{\omega_{Q_2}, 3} \right] * \eta^{\sigma_3, k_3}_{\omega_{Q_3}, 5} *\eta^{\gamma_2, l_2}_{\omega_{P_2}, 0}dx  
\end{eqnarray*} 
can be written as a double average of the form
\begin{eqnarray*}
&& \int_0^1 \int_0^1\sum_{\vec{\sigma} \in \{ 0, \frac{1}{3}, \frac{2}{3} \}^3} \sum_{\vec{\gamma} \in \{ 0, \frac{1}{3}, \frac{2}{3} \}^3} \sum_{\vec{k} \in \mathbb{Z}^3} \sum_{\vec{l} \in \mathbb{Z}^3}  \sum_{ \vec{P} \in \mathbb{P}} \frac{1}{|I_{\vec{P}}|^{1/2}}  \langle f_1, \Phi^{\alpha^\prime, \gamma_1, l_1}_{P_1,1} \rangle \langle f_4, \Phi^{\alpha^\prime, \gamma_3, l_3}_{P_4,4} \rangle \\ && \hspace{20mm} \times \left \langle \sum_{\vec{Q} \in \mathbb{Q} : \omega_{Q_3} \subset \subset \omega_{P_2}} \frac{\langle f_2, \Phi^{\alpha, \sigma_1, k_1}_{Q_1,2} \rangle \langle f_3, \Phi^{\alpha, \sigma_2, l_2}_{Q_2,3}  \rangle \Phi^{\alpha,\sigma_3, k_3}_{Q_3,5}}{|I_{\vec{Q}}|^{1/2}} , \Phi^{\alpha^\prime, \gamma_2, l_2}_{P_2,0} \right\rangle  d\alpha d\alpha^\prime ,
\end{eqnarray*}
where $\mathbb{Q}$ is a rank-1 collection of tri-tiles for which $(\omega_{Q_1}, \omega_{Q_2}, \omega_{Q_3})$ is adapted to $\{xi_1 = - \xi_2/2 = - \xi_3\}$ and $\mathbb{P}$ is a collection of tri-tiles for which $(\omega_{P_1}, \omega_{P_2}, \omega_{P_4})$ is adapted to the degenerate line $\{ \xi_1 + \xi_2 = 0, \xi_3=0\}$. Moreover, each $\Phi_{P_j, i(j)}$ is a wave packet on the tile $P_j= (I_{\vec{P}}, \omega_{P_j})$ for each $j \in \{1,2,4\}$, and each $\Phi_{R_j,i(j)}$ is a wave packet on the tile $R=(I_{\vec{R}}, \omega_{R_j})$ for each $j \in \{1,2,3\}$. The reader should also note that the condition $|\omega_{\vec{P}}| \gg 
|\omega_{\vec{Q}}|$ appearing in $\sum_\prime$ has been replaced by $\omega_{Q_3} \subset \subset \omega_{P_2}$ in the fully discretized version. This is permissible because under the assumption $\left \langle \Phi^{\alpha,\sigma_3, k_3}_{Q_3,5}, \Phi^{\alpha^\prime, \gamma_2, l_2}_{P_2,0} \right \rangle \not = 0$, the conditions $|\omega_{\vec{P}}| \gg |\omega_{\vec{Q}}|$ and $\omega_{Q_3} \subset \subset \omega_{P_3}$ are the same by the fourier support properties of $\Phi^{\alpha,\sigma_3, k_3}_{Q_3,5}$ and $ \Phi^{\alpha^\prime, \gamma_2, l_2}_{P_2,0}$. 

By Theorem \ref{DT}, we know generalized restricted type estimates hold for forms of type $\Lambda_2$ near the extremal points in $\mathbb{A}$. Therefore, due to the rapid decay of the coefficients $c_{\vec{k}}$ and $d_{\vec{l}}$, we know $T_{\phi_{\mathcal{R}_3}}$ satisfies generalized restricted type estimates in the entire interior convex hull of $\mathbb{A}$. Moreover, by symmetry, $T_{\phi_{\mathcal{R}_1}}$ must satisfy generalized restricted type estimates in the entire interior convex hull of $\mathbb{A}^\prime$.  Using generalized restricted type interpolation from Chapter 3 of \cite{MR2199086} gives the desired $L^p$ estimates for $T_{\phi_{\mathcal{R}_1}}$ and $T_{\phi_{\mathcal{R}_3}}$. 

To prove the proposition, we only need to show the desired $L^p$ estimates for $T_{\phi_{\mathcal{R}_2}}$. Again using the discretization argument from Section 6.1 of \cite{MR3052499}, it suffices to obtain restricted type estimates arbitrarily close to the extremal points in $\mathcal{A}$ for the $4$-form 

\begin{eqnarray}\label{Def:4-form-basic}
\sum_{\vec{P} \in \mathbb{P}}  \frac{1}{|I_{\vec{P}}|} \langle f_1, \Phi_{P_1, 1} \rangle \langle f_2, \Phi_{P_2, 2}\rangle \langle f_3, \Phi_{P_3, 3} \rangle \langle f_4, \Phi_{P_4, 4} \rangle ,
\end{eqnarray}
where $\vec{P}= (P_1, P_2, P_3, P_4)$ is a $4$-tile, where each $\Phi_{P_j,j}$ is a wave packet on $P_j= (I_{\vec{P}}, \omega_{P_j})$ for $j=1, 2,3, 4$ and $ (\omega_{P_1}, \omega_{P_2}, \omega_{P_3}, \omega_{P_4})$ is a Whitney cube with respect to $\{\xi_1 = \xi_2 = - \xi_3/2, \xi_4=0\}$. By the results in \cite{MR1887641} for multipliers adapted to singularities of small dimension, it is straightforward to obtain generalized restricted type estimates for \eqref{Def:4-form-basic} and all $\vec{\alpha}$ near the extremal points in $\mathbb{A}$, where the exceptional set can be taken independently of all the necessary time-frequency parameters.

 \end{proof}
The remainder of \S{4} is dedicated to the proof of Theorem \ref{DT}. 
  \subsection{Generalized Restricted Type Estimates near $A_1, A_2, A_3$}
  
 \subsubsection{Tile Decomposition}
Fix tri-tile collections $\mathbb{P}$ and $\mathbb{Q}$ once and for all. For convenience, we shall subsequently use $f_j$ to denote $f_j^\prime$ for $j=1,2,3,4$ as described in Definition \ref{Def:RestrType}. Furthermore, we assume that $|E_4|=1$ by rescaling and that the collections $\mathbb{P}$ and $\mathbb{Q}$ are sparse. For each $\tilde{d} \geq 0$, let $\mathbb{Q}^{\tilde{d}} := \left\{ \vec{Q} \in \mathbb{Q} : 1 + \frac{ dist(I_{\vec{Q}}, \tilde{\Omega}^c) }{|I_{\vec{Q}}|} \simeq 2^{\tilde{d}} \right\}$ and set 

\begin{align}\label{Exc:1}
\tilde{\Omega} =& \left\{ M1_{E_1} \gtrsim |E_1| \right\} \bigcup \left\{ M1_{E_2} \gtrsim |E_2| \right\} \bigcup \left\{ M1_{E_3} \gtrsim |E_3| \right\} \\
\Omega_1^0 =& \left\{ M\left(  \int_0^1   BHT^{\alpha, \mathbb{Q}^0} (f_2, f_3) d \alpha \right) \gtrsim |E_2|^{1/2} |E_3|^{1/2}\right\} \\ 
\Omega_1^{\tilde{d}}=& \left\{ M \left(  \left[ \int_0^1  \sum_{\vec{Q} \in \mathbb{Q}^{\tilde{d}}} \frac{ |\langle f_2, \Phi^\alpha_{Q_1, 2} \rangle \langle f_3, \Phi^\alpha_{Q_2, 3} \rangle|}{|I_{\vec{Q}}|} \chi^M_{I_{\vec{Q}}} d \alpha\right]^2 \right) \gtrsim 2^{2 \tilde{d}} |E_2| |E_3| \right\}.
\end{align}
Lastly, construct 
\begin{eqnarray*}
\Omega=\tilde{\Omega} \bigcup \Omega_1^0 \bigcup_{\tilde{d} \geq 1}  \Omega_1^{\tilde{d}}. 
\end{eqnarray*}
\begin{lemma}
For large enough implicit constants, $|\Omega| \leq 1/2$ and $\tilde{E}_4 := E_4 \cap \Omega^c$ is a major subset of $E_4$. 
\end{lemma}
\begin{proof}
It is immediate from the weak-$\ell^1$ bounds for the maximal function that the implicit constants in \eqref{Exc:1} can be taken large enough to ensure 
\begin{eqnarray*}
|\tilde{\Omega} | \leq 1/6.
\end{eqnarray*}
Moreover, $| \Omega_1^0| \leq 1/6$
for a large enough implicit constant, which follows from the standard $BHT$ estimates. To estimate $\bigcup_{\tilde{d} \geq 1} \Omega_2^{\tilde{\delta}}$, we use the fact that $\left\{ I_{\vec{Q}}: \vec{Q} \in \mathbb{Q}^{\tilde{d}} \right\}$ consists of pairwise disjoint (shifted) dyadic intervals. So, it suffices to observe

\begin{eqnarray*}
| \Omega_1^{\tilde{d}}| \lesssim  2^{-2 \tilde{d}} \frac{ \int_0^1 \left| \left|  \sum_{\vec{Q} \in \mathbb{Q}^{\tilde{d}}} \frac{ |\langle f_2, \Phi^\alpha_{Q_1, 2} \rangle \langle f_3, \Phi^\alpha_{Q_2, 3} \rangle|}{|I_{\vec{Q}}|} \chi^M_{I_{\vec{Q}}} \right| \right|_2^2 d \alpha}{|E_2| |E_3|}. 
\end{eqnarray*}
However,  
\begin{align*}
\left| \left|  \sum_{\vec{Q} \in \mathbb{Q}^{\tilde{d}}} \frac{ |\langle f_2, \Phi^\alpha_{Q_1, 2} \rangle \langle f_3, \Phi^\alpha_{Q_2, 3} \rangle|}{|I_{\vec{Q}}|} \chi^M_{I_{\vec{Q}}} \right| \right|_2^2 \lesssim& \left| \left|  \sum_{\vec{Q} \in \mathbb{Q}^{\tilde{d}}} \frac{ |\langle f_2, \Phi^\alpha_{Q_1, 2} \rangle \langle f_3, \Phi^\alpha_{Q_2, 3} \rangle|}{|I_{\vec{Q}}|} 1_{I_{\vec{Q}}} \right| \right|_2^2\\ \lesssim& \sum_{\vec{Q} \in \mathbb{Q}^{\tilde{d}}} \frac{ |\langle f_2, \Phi^\alpha_{Q_1, 2} \rangle|^2 | \langle f_3, \Phi^\alpha_{Q_2, 3} \rangle|^2}{|I_{\vec{Q}}|^2} \\ \lesssim& 2^{ \tilde{d}} |E_2| |E_3|. 
\end{align*}
Therefore, $\sum_{ \tilde{d} \geq 1} | \Omega_1^{\tilde{d}}| \leq 1/6$ for large enough constants. 
\end{proof}
We keep the average over the parameter $\alpha$ in our expression for $\Lambda_2^{\mathbb{P}, \mathbb{Q}}(\vec{f})$ because otherwise we would need to consider an exceptional set depending  on $\alpha$. However, in producing restricted weak type estimates near $\{A_1, A_2, A_3\}$, the exceptional set should be independent of $\alpha$ as we are interested in obtaining quasi-Banach estimates in which the target $L^p$ index is less than $1$. If we only sought Banach estimates, then it would suffice to obtain an $L^p$ estimate uniform in $\alpha$ and then handle the average over $\alpha$ by Minkowski's inequality.

Our goal is now to obtain the estimate $|\Lambda_2^{\mathbb{P}, \mathbb{Q}}(\vec{f})| \lesssim |E_1|^{\alpha_1} |E_2|^{\alpha_2} |E_3|^{\alpha_3} $ for all $(\alpha_1, \alpha_2, \alpha_3, 1-\alpha_1- \alpha_2- \alpha_3)$ in a small neighborhood near an extremal point $\vec{\beta}  \in \{A_1, A_2, A_3\}$ for all $(f_1, f_2, f_3, f_4)$ satisfying  $|f_1|\leq 1_{E_1}, |f_2| \leq 1_{E_2}, |f_3| \leq 1_{E_3}, |f_4|\leq 1_{E_4 \cap \Omega^c}$. To this end, we shall need the following notions:

\begin{definition}
For any collection of tri-tiles $\tilde{\mathbb{P}} \subset \mathbb{P}$, let

\begin{eqnarray*}
Size_1(f_1, \tilde{\mathbb{P}}) = \sup_{T \subset \tilde{\mathbb{P}}} \frac{1}{|I_T|^{1/2}} \left( \sum_{\vec{P} \in T} |\langle f_1, \Phi_{P_1,1} \rangle|^2 \right)^{1/2},
\end{eqnarray*}
where the supremum is over all $2$-trees $T \subset \tilde{\mathbb{P}}$. 
\end{definition}

\begin{definition}
For each $\tilde{d} \geq 0$, and collection of tri-tiles $\tilde{\mathbb{P}}$, let
\begin{align*}
 & Size_0^{\tilde{d}} (f_2,f_3, \tilde{\mathbb{P}}) :=\\ &   \sup_{T \subset \tilde{\mathbb{P}}} \frac{1}{|I_T|^{1/2}} \left( \sum_{\vec{P} \in T}\left | \left \langle \int_0^1 BHT^{\alpha, \mathbb{Q}^{\tilde{d}}} (f_2, f_3) d \alpha, \Phi_{P_2, 0} \right \rangle \right|^2 \right. \\   & \hspace{20mm} + \left.  \left | \left \langle \int_0^1 \sum_{\vec{Q} \in \mathbb{Q}^{\tilde{d}}: \omega_{Q_3} \supset \supset \omega_{P_2}} \frac{1}{|I_{\vec{Q}}|^{1/2}} \langle f_2, \Phi^\alpha_{Q_1, 2} \rangle \langle f_3, \Phi^\alpha_{Q_2,3} \rangle \Phi^\alpha_{Q_3,5} d\alpha, \Phi_{P_2,0} \right \rangle  \right|^2 \right. \\ & \hspace{20mm} + \left.\left | \left \langle \int_0^1 \sum_{\vec{Q} \in \mathbb{Q}^{\tilde{d}}: |\omega_{Q_3}| \simeq  |\omega_{P_2}|} \frac{1}{|I_{\vec{Q}}|^{1/2}} \langle f_2, \Phi^\alpha_{Q_1, 2} \rangle \langle f_3, \Phi^\alpha_{Q_2,3} \rangle \Phi^\alpha_{Q_3,5} d\alpha, \Phi_{P_2,0} \right \rangle  \right|^2  \right)^{1/2},
\end{align*}
where the supremum is over all $1$-trees $T \subset \tilde{\mathbb{P}}$. 
\end{definition}

Now let $\mathbb{P}^d := \left\{ \vec{P} \in \mathbb{P}: 1+ \frac{dist(I_{\vec{P}}, \Omega^c)}{|I_{\vec{P}}|} \simeq 2^d \right\}$. We now recall the following tree selection algorithm for $\mathbb{P}^d$ (essentially) from Section 6.3 of \cite{MR3052499}:
\begin{lemma}\label{TDL}
Fix $d, \tilde{d} \geq 0$. Then there exist two decompositions of $\mathbb{P}^d$, namely $\bigcup_{n_1 \geq N_1(d)}  \mathbb{P}^{d}_{n_1,1}$ and $ \bigcup_{ \mathfrak{d} \geq N_2(d, \tilde{d})} \mathbb{P}^{d, \tilde{d}}_{\mathfrak{d}, 2}$ such that $Size_1(f_1, \mathbb{P}_{n_1, 1}^d) \lesssim 2^{-n_1}$ and $Size_0^{\tilde{d}}(f_2, f_3, \mathbb{P}_{\mathfrak{d},2}^{d, \tilde{d}}) \lesssim 2^{-\mathfrak{d}}$. Moreover, $\mathbb{P}^d_{n_1, 1}$ and $\mathbb{P}_{\mathfrak{d}, 2}^{d, \tilde{d}}$ can each be written as a union of trees, i.e. 

\begin{align}\label{Def:P-sets}
\mathbb{P}^d_{n_1, 1} =& \bigcup_{T \in \mathcal{T}_{n_1, 1}^d} \bigcup_{\vec{P} \in T} \vec{P} \\ \mathbb{P}^{d, \tilde{d}}_{\mathfrak{d}, 2} =& \bigcup_{T \in \mathcal{T}^{d, \tilde{d}}_{\mathfrak{d}, 2}} \bigcup_{\vec{P} \in T} \vec{P}, 
\end{align}
such that $\sum_{T \in \mathcal{T}_{n_1, 1}^d} |I_T| \lesssim  2^{2n_1} \sum_{T \in \mathcal{T}_{n_1, 1,*}^d} \sum_{\vec{P} \in T} \left| \langle f_1, \Phi_{P_1, 1} \rangle \right|^2$ and

\begin{align*}
&\sum_{T \in \mathcal{T}_{\mathfrak{d}, 2}^{d, \tilde{d}}} |I_T| \\ \lesssim& 2^{2 \mathfrak{d}}\sum_{T \in \mathcal{T}_{\mathfrak{d}, 2,*}^{d, \tilde{d}}}  \sum_{\vec{P} \in T}\left| \left \langle \int_0^1 BHT^{\alpha,\mathbb{Q}^{\tilde{d}}}_{\omega_{P_2}}(f_2, f_3) d \alpha, \Phi_{P_2,0} \right\rangle \right| ^2 \\+&  2^{2 \mathfrak{d}}\sum_{T \in \mathcal{T}_{\mathfrak{d}, 2,*}^{d, \tilde{d}}}  \sum_{\vec{P} \in T} \left|\left \langle \int_0^1 \sum_{\vec{Q} \in \mathbb{Q}^{\tilde{d}}: \omega_{Q_3} \supset \supset \omega_{P_2}} \frac{1}{|I_{\vec{Q}}|^{1/2}} \langle f_2, \Phi^\alpha_{Q_1, 2} \rangle \langle f_3, \Phi^\alpha_{Q_2,3} \rangle \Phi^\alpha_{Q_3,4} d\alpha, \Phi_{P_2,0} \right \rangle  \right| ^2  \\+& 2^{2 \mathfrak{d}}\sum_{T \in \mathcal{T}_{\mathfrak{d}, 2,*}^{d, \tilde{d}}}  \sum_{\vec{P} \in T} \left|\left \langle \int_0^1 \sum_{\vec{Q} \in \mathbb{Q}^{\tilde{d}}:| \omega_{Q_3}| \simeq |\omega_{P_2}|} \frac{1}{|I_{\vec{Q}}|^{1/2}} \langle f_2, \Phi^\alpha_{Q_1, 2} \rangle \langle f_3, \Phi^\alpha_{Q_2,3} \rangle \Phi^\alpha_{Q_3,4} d\alpha, \Phi_{P_2,0} \right \rangle  \right| ^2   ,
\end{align*}
where $\mathcal{T}_{n_1, 1,*}^d \subset  \mathcal{T}_{n_1, 1}^d $ is a collection of $2$-trees and $\mathcal{T}_{\mathfrak{d}, 2,*}^{d, \tilde{d}} \subset  \mathcal{T}_{\mathfrak{d}, 2}^{d, \tilde{d}}$ is a collection of $1$-trees.  We further decompose

\begin{align*}\label{Def:Tree-set}
\mathcal{T}_{n_1, 1,*}^d =& \mathcal{T}_{n_1, 1,*, +}^d \bigcup \mathcal{T}_{n_1, 1,*, -}^d \\ 
\mathcal{T}_{\mathfrak{d}, 2,*}^{d, \tilde{d}}=& \mathcal{T}_{\mathfrak{d}, 2,*, +}^{d, \tilde{d}} \bigcup \mathcal{T}_{\mathfrak{d}, 2,*, -}^{d, \tilde{d}}.
\end{align*}
where $\mathcal{T}_{n_1, 1,*, +}^d$ and $\mathcal{T}_{n_1, 1,*, -}^d$ form $2$ strongly $1$-disjoint chains and $\mathcal{T}_{\mathfrak{d}, 2,*, +}^{d, \tilde{d}}$ and $\mathcal{T}_{\mathfrak{d}, 2,*, -}^{d, \tilde{d}}$ form $2$ strongly $2$-disjoint chains.

\end{lemma}

\begin{proof}
We describe the procedure for producing the collection $\mathcal{T}^d_{n_1, 1}$, as the decomposition into trees in the collection $\mathcal{T}^{d, \tilde{d}}_{\mathfrak{d}, 2} $ is very similar.  Let $N_1(d)$ be the smallest integer for which $Size_1(f_1, \mathbb{P}^d) \geq 2^{-N_1(d)}$. We may assume without loss of generality that there are only finitely many tri-tiles in the collection $\mathbb{P}^d$, and our bounds will be independent of the cardinality of tiles. 
Assume the collection $\mathbb{P}^d_{m, 1}$ has already been constructed with all the desired properties for $m < n_1$. We now perform the following standard tile selection algorithm on the tri-tile collection $\mathbb{P}^d \cap \left[ \bigcup_{ N_1(d) \leq m < n_1} \mathbb{P}^d_{n_1, 1} \right]^c$ to produce $\mathbb{P}^d_{n_1, 1}$ with the desired properties. To this end, introduce the following notation: if $P$ is a tile, let $\xi_P$ denote the center of $\omega_P$. If $P$ and $P^\prime$ are tiles, we write $P^\prime \lesssim^+ P$ if $P^\prime \lesssim^\prime P$ and $\xi_{P^\prime} > \xi_P$, and $P^\prime \lesssim^- P$ if $P^\prime \lesssim^\prime P$ and $\xi_{P^\prime} < \xi_P$.  Now consider the set of $2$-trees  in $\mathbb{P}^d \cap \left[ \bigcup_{ N_1(d) \leq m < n_1} \mathbb{P}^d_{n_1, 1} \right]^c$ which are upward in the sense that 

\begin{eqnarray*}
P_j \lesssim^+ P_{T, j}~for~all~ \vec{P} \in T
\end{eqnarray*}
and which satisfies $\sum_{\vec{P} \in T} |\langle f_1, \Phi_{P_1,1} \rangle|^2 \geq 2^{-2n_1-3}|I_T|$. If there are no trees with this property, terminate the algorithm. Otherwise, choose $T$ among all such trees so that the center $\xi_{T,1}$ of $\omega_{P_{T}, t}$ is maximal and that $T$ is maximal with respect to the set inclusion. Moreover, let $T^\prime$ denote that $1$-tree 

\begin{eqnarray*}
T^\prime := \left\{ \vec{P} \in \mathbb{P}^d \cap T^c : P_1 \leq P_{T, 1} \right\}.
\end{eqnarray*}
Now remove $T$ and $T^\prime$ from $\mathbb{P}^d$. Then repeat the tile selection process with the remaining tri-tiles $\mathbb{P}^d \cap (T \cup T^\prime)^c$ until there are no more upward trees satisfying the size condition. Again, by our finiteness assumption, the algorithm terminates in a finite number of steps, producing trees $T_1, T_1^\prime, T_2, T_2^\prime, ..., T_M, T^\prime_M$, where each $T_j$ is a $2$-tree and each $T_j^\prime$ is a $1$-tree. Set

\begin{align*}
\mathcal{T}^d_{n_1, 1, +} =&  \bigcup_{j=1}^M \left[ T_j \cup T^\prime_j \right]\\
\mathcal{T}^d_{n_1, 1, *, +} =&  \bigcup_{j=1}^M T_j.
\end{align*}
The claim is now that $T_1,..., T_M$ form a chain of strongly $1$-disjoint trees. Indeed, it is clear that $T_s \cap T_{s^\prime} = \emptyset$ when $s \not = s^\prime$. Therefore, we must have $P_1 \not = P_1^\prime$ for all $\vec{P} \in T_s$, $\vec{P}^\prime \in T_{s^\prime}, s \not = s^\prime$. 
Suppose for a contradiction that there were tri-tiles $\vec{P} \in T_s, \vec{P}^\prime \in T_{s^\prime}$ such that $2\omega_{P_1} \subsetneq 2 \omega_{P_1^\prime}$ and $I_{P^\prime_2} \subset I_{T_s}$. By sparseness, we thus have $|\omega_{P_1^\prime}| \geq 10^9 |\omega_{P_1}|$. Since $P_1 \lesssim^+ P_{T_s, 1}$ and $P_1^\prime \lesssim^+ P_{T_{s^\prime}, 1}$, we thus see that $\xi_{P_{T_{s^\prime}, 1}} < \xi_{P_{T_s}, 1}$. By our select algorithm, this implies $s < s^\prime$. Also, since $|\omega_{P_1^\prime}| \geq 10^9 |\omega_{P_1}|$, $I_{P_1^\prime} \subset I_{T_s}$, and $P_1 \lesssim P_{T_s, 1}$, it must be that $P_1^\prime \leq P_{T_s, 1}$. Since $s < s^\prime$, this means that $\vec{P}^\prime \in T^\prime_{s}$. But $T^\prime_s$ and $T_{s^\prime}$ are disjoint trees by construction, which is a contradiction. 
Now repeat the previous algorithm, but replace $\lesssim^+$ by $\lesssim^-$, so the trees $T$ are downward pointing instead of upward pointing, and select the trees $T$ so that the center $\xi_{T, j}$ is minimized rather than maximized. This yields two  further collection of trees $\mathcal{T}^d_{n_1, 1, -}$ and $\mathcal{T}^d_{n_1, 1,*,-}$ such that for any $2$-tree $T$ consisting of unselected tiles 

\begin{align*}
\sum_{\vec{P} \in T : P_2 \lesssim^- P_{T_2}} |\langle f_1, \Phi_{P_1, 1} \rangle|^2< 2^{-2n-3} |I_T|. 
\end{align*}
Letting $\mathcal{T}^d_{n_1, 1} = \mathcal{T}^d_{n_1, 1, +} \bigcup  \mathcal{T}^d_{n_1, 1, -}$ and $ \mathcal{T}^d_{n_1, 1, *} = \mathcal{T}^d_{n_1, 1, *,+} \bigcup  \mathcal{T}^d_{n_1, 1, *,-}$, it follows that

\begin{align*}
Size_1\left (f_1, \mathbb{P}^d \cap \left[ \bigcup_{N_1(d) \leq m \leq n_1} \mathcal{T}_{m, 1}^d \right]^c\right) < 2^{-2(n_1+1) }.
\end{align*}

\end{proof}
Letting $\mathbb{P}^{d, \tilde{d},*}_{n_1, \mathfrak{d}} = \mathbb{P}^{d}_{n_1, 1} \cap \mathbb{P}^{d, \tilde{d}}_{\mathfrak{d}, 2}$, we obtain the decomposition

\begin{align}
\mathbb{P} \times \mathbb{Q} =& \bigcup_{ d, \tilde{d} \geq 0}   \bigcup_{n_1 \geq N_1(d)} \bigcup_{ \mathfrak{d} \geq N_2(d, \tilde{d})} \mathbb{P}^{d, \tilde{d},*}_{n_1, \mathfrak{d}} \times \mathbb{Q}^{\tilde{d}}.
\end{align}

\subsubsection{Tree Estimates}

 First, let $T \subset \mathbb{P}^{d, \tilde{d}}_{n_1, \mathfrak{d}}$ be a 2-tree.  Then
\begin{align*}
&\left|  \sum_{\vec{P} \in T}  \frac{  \langle f_1, \Phi_{P_1,1} \rangle \langle f_4, \Phi^{lac}_{P_4,4} \rangle \left \langle \int_0^1 BHT^{\mathbb{Q}^{\alpha, \tilde{d}}}_{\omega_{P_2}}(f_2, f_3) d \alpha ,\Phi_{P_2,0} \right \rangle}{|I_{\vec{P}}|^{1/2}} \right| \\ \leq&\frac{ \left( \sum_{\vec{P} \in T} |\langle f_1, \Phi_{P_1,1} \rangle|^2 \right)^{1/2}}{|I_T|^{1/2}} \cdot \frac{  \left( \sum_{\vec{P} \in T} |\langle f_4, \Phi_{\vec{P},4}^{lac} \rangle|^2 \right)^{1/2}}{|I_T|^{1/2}} \\ \times&  \sup_{\vec{P} \in T}\left[  \frac{\left| \left \langle \int_0^1 BHT^{\alpha, \mathbb{Q}^{\tilde{d}}}_{\omega_{P_2}}(f_2, f_3) d \alpha, \tilde{\Phi}^{\infty}_{P_2,0} \right\rangle \right|}{|I_{\vec{Q}}|} \right] |I_T|\\ \lesssim &~2^{-N d} 2^{-n_1}2^{-\frak{d}} |I_T|. 
\end{align*}
Now let $T \subset \mathbb{P}^{d, \tilde{d}}_{n_1, \mathfrak{d}}$ be a $1$-tree. Then 

\begin{align*}
& \left| \sum_{\vec{P} \in T}  \frac{ \langle f_1, \Phi_{P_1} \rangle \langle f_4, \Phi^{lac}_{P_4,4} \rangle \left \langle \int_0^1 BHT^{\alpha, \mathbb{Q}^{ \tilde{d}}}_{\omega_{P_2}}(f_2, f_3) d\alpha,\Phi_{P_2,0} \right \rangle}{|I_{\vec{P}}|^{1/2}} \right| \\ \lesssim&  \left[ \sup_{\vec{P} \in T} \frac{ |\langle f_1, \Phi_{P_1,1} \rangle|}{|I_{\vec{P}}|^{1/2}} \right]  \left( \sum_{\vec{P} \in T} \frac{|\langle f_4, \Phi_{\vec{P},4}^{lac} \rangle|^2}{|I_T|} \right)^{1/2} \\ \times&  \left( \sum_{\vec{P} \in T} \frac{ \left| \left \langle \int_0^1 BHT^{\alpha, \mathbb{Q}^{\tilde{d}}}_{\omega_{P_2}}(f_2, f_3) d \alpha, \Phi_{P_2,0} \right\rangle\right|^2}{|I_T|}\right)^{1/2} |I_T| \\ \lesssim& 2^{-N d} 2^{-n_1}2^{-\frak{d}} |I_T|.
\end{align*}

\subsubsection{Size Restrictions}
Before proceeding to the main lemma of this section, we record

\begin{lemma}\label{L:BiestSize}
Let $\tilde{\mathbb{P}} \subset \mathbb{P}$ be any subcollection of tri-tiles. Then, there is $M \gg 1$ such that for any $0 < \theta <1$ and $1$-tree $T \subset \tilde{\mathbb{P}}$ 
\begin{align*}
& \left[ \frac{1}{|I_T|} \sum_{\vec{P} \in T} \left|\left \langle  \int_0^1 \sum_{\vec{Q} \in \mathbb{Q}^{\tilde{d}} : \omega_{Q_3} \supset \supset  \omega_{P_2}}\ \frac{1}{|I_{\vec{Q}}|^{1/2}} \langle f_2, \Phi^\alpha_{Q_1,2} \rangle \langle f_3, \Phi^\alpha_{Q_2,3} \rangle \Phi^\alpha_{Q_3, 5} d\alpha , \Phi_{P_2,0}\right  \rangle \right|^2 \right]^{1/2}\\ \lesssim_\theta& \left[ \sup_{\vec{P} \in \tilde{\mathbb{P}}} \frac{1}{|I_{\vec{P}}|} \int 1_{E_2} \chi^M_{I_{\vec{P}}} dx \right]^\theta \left[ \sup_{\vec{P} \in \tilde{\mathbb{P}}} \frac{1}{|I_{\vec{P}}|} \int 1_{E_3} \chi^M_{I_{\vec{P}}} dx \right]^{1-\theta}.
\end{align*}
\end{lemma}
\begin{proof} 
The above result is a slightly modified version of Lemma 9.1 in \cite{MR2127985} in which the roles of $\mathbb{P}$ and $\mathbb{Q}$ are reversed and the restriction $\omega_{Q_3} \supset \supset  \omega_{P_2}$ is replaced with $\omega_{Q_3} \supseteq \omega_{P_2}$.  However, this modification does not create any difficulties. A step by step repeat of the proof of Lemma 9.1 substituting $\supset \supset$ for $\supseteq$ does the job. 
\end{proof}
The next result we shall need is
\begin{lemma}\label{L:John-Nirenberg}
Let $T$ be a $1$-tree. Then there is $M \gg 1$ for which

\begin{align*}
\left( \sum_{\vec{P} \in T} \frac{1}{|I_{\vec{P}}|}\left| \left \langle f, \Phi_{P_2} \right \rangle \right|^2 \right)^{1/2} \lesssim \sup_{\vec{P}\in T} \frac{1}{|I_{\vec{P}}|} \int_\mathbb{R} |f| \chi^M_{I_{\vec{P}}} dx. 
\end{align*}
\begin{proof}
There are many available references. For example, see Lemma 6.13 in \cite{MR3052499}. 
\end{proof}

\end{lemma}
The next result gives some necessary conditions for $\mathbb{P}_{n_1, \mathfrak{d}}^{d, \tilde{d}}$ to be nonempty:
\begin{lemma}\label{Est:Size-Restrictions}
Fix $d, \tilde{d}, n_1, \mathfrak{d} \geq 0$ such that $\mathbb{P}^{d, \tilde{d}}_{n_1, \mathfrak{d}}$ is nonempty. Then for all $N \geq 1$
\begin{align*}
2^{-n_1}& \lesssim 2^d|E_1|\\
2^{-\mathfrak{d}} &\lesssim_N 2^{-N(\tilde{d}-d) } |E_2|^{1/2} |E_3|^{1/2}.
\end{align*}
As a consequence, we know that in the tile decomposition

\begin{align*}
\mathbb{P} \times \mathbb{Q} = \bigcup_{d, \tilde{d} \geq 0} \bigcup_{n_1 \geq N_1(d)} \bigcup_{\mathfrak{d} \geq N_2(d, \tilde{d})} \mathbb{P}^{d, \tilde{d}}_{n_1, \mathfrak{d}} \times \mathbb{Q}^{\tilde{d}}
\end{align*}
$2^{-N_1(d)} \lesssim 2^d |E_1|$ and $2^{-N_2(d, \tilde{d})} \lesssim_N 2^{-N(\tilde{d}-d)} |E_2|^{1/2} |E_3|^{1/2}$. 

\end{lemma}
\begin{proof}
The estimate $2^{-n_1} \lesssim 2^d |E_1|$ is an immediate consequence of Lemma \ref{L:John-Nirenberg} and the definition of $\mathbb{P}^d$, so the details are omitted.
Therefore, it suffices to prove $2^{-\mathfrak{d}} \lesssim_N 2^{-N(\tilde{d}-d)} |E_2|^{1/2} |E_3|^{1/2}$. 

CASE 1: Assume $ \tilde{d} \lesssim d$. It clearly suffices to show that for every $\tilde{\mathbb{P}}_{n_1, \delta}^{d, \tilde{d}}$-tree $T$

\begin{align*}
&  \left( \frac{1}{|I_T|} \sum_{\vec{P} \in T}  \left| \left \langle \int_0^1 BHT^{\alpha , \mathbb{Q}^{\tilde{d}}}(f_2, f_3) d \alpha, \Phi_{P_2,3} \right \rangle\right|^2\right)^{1/2} \\+& \left( \frac{1}{|I_T|} \sum_{\vec{P} \in T}  \left| \left \langle \int_0^1 \sum_{\vec{Q} \in \mathbb{Q}^{\tilde{d}}: \omega_{Q_3} \supset \supset \omega_{P_2}} \frac{1}{|I_{\vec{Q}}|^{1/2}} \langle f_2, \Phi^\alpha_{Q_1, 2} \rangle \langle f_3, \Phi^\alpha_{Q_2,3} \rangle \Phi^\alpha_{Q_3,4} d\alpha, \Phi_{P_2,0} \right \rangle \right|^2 \right)^{1/2} \\+& \left( \frac{1}{|I_T|} \sum_{\vec{P} \in T} \left| \left \langle  \int_0^1 \sum_{\vec{Q} \in \mathbb{Q}^{\tilde{d}} : |\omega_{Q_3}| \simeq |\omega_{P_2}| }\frac{\langle f_2, \Phi^\alpha_{Q_1, 2} \rangle \langle f_3, \Phi^\alpha_{Q_2, 3} \rangle}{|I_{\vec{Q}}|^{1/2}} \Phi^\alpha_{Q_3, 5}d \alpha , \Phi_{P_2,0} \right \rangle\right|^2\right)^{1/2}\\ =:& ~I + II +III \\ \lesssim& ~2^d |E_2|^{1/2} |E_3|^{1/2}. 
\end{align*}
By Lemma \ref{L:John-Nirenberg}, $I \lesssim \sup_{\vec{P} \in T} \frac{1}{|I_{\vec{P}}|} \int_\mathbb{R} \left| \int_0^1 BHT^{\alpha, \mathbb{Q}^{\tilde{d}}}(f_2, f_3) d \alpha\right| \chi^M_{I_{\vec{P}}}dx \lesssim 2^d 2^{ \tilde{d}} |E_2|^{1/2} |E_3|^{1/2}$. This is clearly acceptable by the assumption $d \geq \tilde{d}$. 
Invoking Lemma \ref{L:BiestSize} gives $II \lesssim 2^d|E_2|^{1/2} |E_3|^{1/2}$. To handle term $III$, it suffices to observe for some $M \gg 1$
\begin{align*}
III^2 &\lesssim  \frac{1}{|I_T|} \sum_{l \in \mathbb{Z}} \frac{1}{1+l^M} \sum_{\vec{Q}\in \mathbb{Q}^{\tilde{d}}: I_{\vec{Q}} \subset I_T + l |I_T|} \int_0^1 \frac{ |\langle f_2, \Phi^\alpha_{Q_1,2} \rangle|^2 |\langle f_3, \Phi^\alpha_{Q_2, 3} \rangle|^2}{|I_{\vec{Q}}|} d \alpha\\&\lesssim 2^{2\tilde{d}} |E_2| |E_3|   \sum_{l \in \mathbb{Z}} \frac{1}{1+l^M} \frac{  ||1_{E_2} \chi^M_{I_{T}+l |I_T|} ||_{L^2} \cdot ||1_{E_3} \chi^M_{I_{T}+l |I_T|}||_{L^2}}{|I_T|} \\ &\lesssim  2^{2 d} |E_2| |E_3|. 
\end{align*}

 CASE 2: Assume $\tilde{d} >> d$. It suffices to prove that for every $\mathbb{P}_{n_1, \delta}^{d, \tilde{d}}$-tree $T$ and some $N \gg 1$

\begin{align*}
& \left( \frac{1}{|I_T|} \sum_{\vec{P} \in T}  \left| \left \langle \int_0^1 BHT^{\alpha , \mathbb{Q}^{\tilde{d}}}(f_2, f_3) d \alpha, \Phi_{P_2,0} \right \rangle\right|^2\right)^{1/2} \\+&  \left( \frac{1}{|I_T|} \sum_{\vec{P} \in T}  \left| \left \langle \int_0^1 BHT^{\alpha , \mathbb{Q}^{\tilde{d}}}_{\omega_{P_2}}(f_2, f_3) d \alpha, \Phi_{P_2,0} \right \rangle\right|^2\right)^{1/2}\\+& \left( \frac{1}{|I_T|} \sum_{\vec{P} \in T} \left| \left \langle  \int_0^1 \sum_{\vec{Q} \in \mathbb{Q}^{\tilde{d}} : |\omega_{Q_3}| \simeq |\omega_{P_2}| }\frac{\langle f_2, \Phi^\alpha_{Q_1, 2} \rangle \langle f_3, \Phi^\alpha_{Q_2, 3} \rangle}{|I_{\vec{Q}}|^{1/2}} \Phi^\alpha_{Q_3, 5}d \alpha , \Phi_{P_2,0} \right \rangle\right|^2\right)^{1/2} \\:=& ~I + II +III\\\lesssim&~ 2^{-N \tilde{d}} |E_2|^{1/2} |E_3|^{1/2}.
\end{align*}
We now want to exploit the fact that the $\mathbb{Q}$-tiles appearing inside the $\mathbb{P}$-sum have finer frequency localization than $\omega_{P_2}$ and so are adapted to larger time intervals than $I_{\vec{P}}$. To this end, observe that whenever $(\vec{P}, \vec{Q}) \in \mathbb{P}^d \times \mathbb{Q}^{\tilde{d}}$ satisfy $|I_{\vec{P}}| \leq |I_{\vec{Q}}|$ and $\tilde{d} \gg d$, then $dist(I_{\vec{P}}, I_{\vec{Q}}) \gtrsim 2^{\tilde{d}} |I_{\vec{Q}}|.$
 This is because $\vec{P} \in \mathbb{P}^d$ implies $1+\frac{ dist(I_{\vec{P}}, \Omega^c) }{|I_{\vec{P}}|} \simeq 2^d$. Therefore, using $\Omega \supset \tilde{\Omega}$, 

\begin{eqnarray*}
 dist(I_{\vec{P}}, \tilde{\Omega}^c) \lesssim 2^d |I_{\vec{P}}|.
\end{eqnarray*}
If the proposed inequality did not hold, then 

\begin{eqnarray*}
dist(I_{\vec{Q}}, \tilde{\Omega}^c) \leq dist(I_{\vec{Q}}, I_{\vec{P}}) + |I_{\vec{P}}| + dist(I_{\vec{P}}, \tilde{\Omega}^c) \ll  2^{\tilde{d}} |I_{\vec{Q}}|,
\end{eqnarray*} which would violate the assumption $\vec{Q} \in \mathbb{Q}^{\tilde{d}}$.  With this observation, we now can write down using Lemma \ref{L:John-Nirenberg} and the definition of the exceptional set $\Omega$

\begin{align*}
I \lesssim& \sup_{\vec{P} \in T} \frac{1}{|I_{\vec{P}}|} \int_\mathbb{R} \left| \int_0^1 \sum_{\vec{Q} \in \mathbb{Q}^{\tilde{d}}}  \frac{ \left| \langle f_2, \Phi^\alpha_{Q_1, 2} \rangle \langle f_3, \Phi^\alpha _{Q_2, 3} \rangle \right|}{|I_{\vec{Q}}|} \chi^M_{I_{\vec{Q}}}d \alpha  \right| \chi^M_{I_{\vec{P}}} dx\\  \lesssim& 2^{-N \tilde{d}}  \sup_{\vec{P} \in \mathbb{P}^{d, \tilde{d}}_{n_1, \mathfrak{d}}} \inf_{x \in I_{\vec{P}}} M \left(  \int_0^1 \sum_{\vec{Q} \in \mathbb{Q}^{\tilde{d}}} \frac{ \left| \langle f_2, \Phi^\alpha_{Q_1, 2} \rangle \langle f_3, \Phi^\alpha_{Q_2, 3} \rangle \right|}{|I_{\vec{Q}}|} \chi^M_{I_{\vec{Q}}} d\alpha  \right) (x) \\ \lesssim& 2^{-N \tilde{d}/2} 2^{d/2} 2^{\tilde{d}} |E_2|^{1/2} |E_3|^{1/2} \\\lesssim& 2^{-N \tilde{d}/4} |E_2|^{1/2} |E_3|^{1/2}.  
\end{align*}
Furthermore, for term $II$, it suffices to note
\begin{align*}
&  \sum_{\vec{P} \in T}  \left| \left \langle \int_0^1 BHT^{\alpha, \mathbb{Q}^{\tilde{d}}}_{\omega_{P_2}}(f_2, f_3) d\alpha, \Phi_{P_2,0} \right \rangle\right|^2 \\=&  \sum_{\vec{P} \in T} \left| \left \langle \int_0^1  \sum_{\vec{Q} \in \mathbb{Q}^{\tilde{d}} : \omega_{Q_3} \subset \subset  \omega_{P_2}} \frac{\langle f_2, \Phi^\alpha_{Q_1, 2} \rangle \langle f_3, \Phi^\alpha_{Q_2, 3} \rangle}{|I_{\vec{Q}}|^{1/2}} \Phi^\alpha_{\vec{Q}_3, 5} d \alpha , \Phi_{P_2,0} \right \rangle\right|^2 \\ \lesssim&   2^{-N \tilde{d}} \sum_{\omega \in \Omega_2\{T\}}  \sum_{\vec{P} \in T : \omega_{P_2} = \omega}\left| \left \langle  \int_0^1 \sum_{\vec{Q} \in \mathbb{Q}^{\tilde{d}} : \omega_{Q_3} \subset \subset \omega} \frac{ \left| \langle f_2, \Phi^\alpha_{Q_1, 2} \rangle \langle f_3, \Phi^\alpha_{Q_2, 3} \rangle \right|}{|I_{\vec{Q}}|} \chi^M_{I_{\vec{Q}}}  d \alpha, \chi^M_{I_{\vec{P}}} \right \rangle\right|^2 \\ \lesssim&   2^{-N \tilde{d}} \sum_{\omega \in \Omega_2\{T\}} \left| \left| \int_0^1 \sum_{\vec{Q} \in \mathbb{Q}^{\tilde{d}} : \omega_{Q_3} \subset \subset \omega} \frac{ \left| \langle f_2, \Phi^\alpha_{Q_1, 2} \rangle \langle f_3, \Phi^\alpha_{Q_2, 3} \rangle \right|}{|I_{\vec{Q}}|} d\alpha \chi^M_{I_{\vec{Q}}} \chi_{I_T}^M \right| \right|_2^2.  
 \end{align*}
The last line of the above display can be majorized by 
\begin{eqnarray*}
2^{-N \tilde{d}} \left| \left| \int_0^1\sum_{\vec{Q} \in \mathbb{Q}^{\tilde{d}}}  \frac{ \left| \langle f_2, \Phi^\alpha_{Q_1, 2} \rangle \langle f_3, \Phi^\alpha _{Q_2, 3} \rangle \right|}{|I_{\vec{Q}}|} \chi^M_{I_{\vec{Q}}} \chi^M_{I_T}d \alpha \right| \right|_2^2.
\end{eqnarray*}
Therefore, using the definition of the exceptional set $\Omega$ once more, 

\begin{align*}
&  \left( \frac{1}{|I_T|} \sum_{\vec{P} \in T}  \left| \left \langle BHT^{\mathbb{Q}^{\tilde{d}}}_{\omega_{P_2}}(f_2, f_3), \Phi_{P_2,0} \right \rangle\right|^2 \right)^{1/2}\\ \lesssim&2^{-N \tilde{d}}  \left(  \sup_{\vec{P} \in \mathbb{P}^{d, \tilde{d}}_{n_1, \mathfrak{d}}} \inf_{x \in I_{\vec{P}}} M \left(  \left[ \int_0^1 \sum_{\vec{Q} \in \mathbb{Q}^{\tilde{d}}} \frac{ \left| \langle f_2, \Phi^\alpha_{Q_1, 2} \rangle \langle f_3, \Phi^\alpha_{Q_2, 3} \rangle \right|}{|I_{\vec{Q}}|} \chi^M_{I_{\vec{Q}}} d\alpha \right] ^2 \right) (x)\right)^{1/2} \\ \lesssim& 2^{-N \tilde{d}/2} 2^{d/2} 2^{\tilde{d}} |E_2|^{1/2} |E_3|^{1/2} \\\lesssim& 2^{-N \tilde{d}/4} |E_2|^{1/2} |E_3|^{1/2}. 
\end{align*}
By combining this estimate with the size estimate in Lemma \ref{L:BiestSize}, we deduce the desired claim for $II$.
It remains to bound term $III$. However, this is immediate from 

\begin{align*}
III \lesssim& 
2^{-N \tilde{d}} \sum_{l \in \mathbb{Z}} \frac{1}{1+l^N}  \int_0^1  \left( \frac{1}{|I_T|}\sum_{\vec{Q} \in \mathbb{Q}^{\tilde{d}} : I_{\vec{Q}} \subset I_T + l |I_T|} \frac{1}{|I_{\vec{Q}}|} |\langle f_2, \Phi^\alpha_{Q_1, 2} \rangle|^2 | \langle f_3, \Phi^\alpha_{Q_2, 3} \rangle|^2 \right)^{1/2} d \alpha \\ \lesssim& 2^{-N \tilde{d}} |E_2|^{1/2} |E_3|^{1/2}. 
\end{align*}

\end{proof}

\subsubsection{$\ell^2$-Energy Estimates}
In preparation for the main energy estimates, we first record
\begin{lemma}\label{Est:BHTloc}
Fix $\alpha \in [0, 1]$ and $\theta \in (0,1)$. Then 
\begin{align*}
\left| \left| BHT^{\alpha, \mathbb{Q}^{\tilde{d}}} (f_2, f_3) \right| \right|_2 \lesssim _{\theta}& \left[ |E_2|^{1/2}\sup_{\vec{Q} \in \mathbb{Q}^{\tilde{d}}} \frac{ \int_{E_3} \chi^M_{I_{\vec{Q}}} dx }{|I_{\vec{Q}}|} \right]^{1-\theta} \cdot \left[ |E_3|^{1/2} \sup_{\vec{Q} \in \mathbb{Q}^{\tilde{d}}}\frac{ \int_{E_2}\chi^M_{I_{\vec{Q}}} dx }{|I_{\vec{Q}}|} \right]^\theta \\ \lesssim& 2^{\tilde{d}} |E_2| ^{\frac{1+\theta}{2}} |E_3|^{\frac{2-\theta}{2}},
\end{align*}
where the implicit constant in the above display can be taken independently of $\alpha$. 
\end{lemma}
\begin{proof}
Apply the localized $BHT$ size/energy estimate from the proof of Theorem 1.1 in \cite{MR2127985} and then use the definition of $\mathbb{Q}^{\tilde{d}}$. 
\end{proof}

\begin{lemma}\label{Est:BHTEnergy}
Let $\mathbb{T} = \{ T \}$ be a chain of strongly $j$-disjoint $i$-trees for which $i \not = j$ with comparable tree sizes and for which each subtree $T^\prime \subset T$ satisfies

\begin{eqnarray*}
\frac{1}{|I_{T^\prime}|} \sum_{\vec{P} \in T^\prime} \langle f, \Phi_{P_j} \rangle|^2 \lesssim \frac{1}{|I_T|} \sum_{\vec{P} \in T} |\langle f , \Phi_{P_j} \rangle|^2.
\end{eqnarray*}
Then 

\begin{eqnarray*}
\sum_{T \in \mathbb{T}} \sum_{\vec{P} \in T} |\langle f , \Phi_{P_j} \rangle|^2 \lesssim || f||_2^2. 
\end{eqnarray*}

\end{lemma}
\begin{proof}
See section 6.5 on Bessel-type inequalities in \cite{MR3052499}.
\end{proof}

We also need the following energy bound: 

\begin{lemma}\label{Est:BiestEnergy}
Let $\mathbb{T} = \left\{ T \right\}$ be a chain of strongly $2$-disjoint $1$-trees satisfying for every $T_1, T_2 \in \mathbb{T}$ 

\begin{align*}
\frac{1}{|I_{T_1}|} \sum_{\vec{P} \in T_1}  \left|\left \langle \int_0^1 \sum_{\vec{Q} \in \mathbb{Q}^{\tilde{d}}: \omega_{Q_3} \supset \supset \omega_{P_2}} \frac{1}{|I_{\vec{Q}}|^{1/2}} \langle f_2, \Phi^\alpha_{Q_1, 2} \rangle \langle f_3, \Phi^\alpha_{Q_2,3} \rangle \Phi^\alpha_{Q_3,4} d\alpha, \Phi_{P_2,0} \right \rangle \right|^2 \\ \simeq  \frac{1}{|I_{T_2}|} \sum_{\vec{P} \in T_2}  \left|\left \langle \int_0^1 \sum_{\vec{Q} \in \mathbb{Q}^{\tilde{d}}: \omega_{Q_3} \supset \supset \omega_{P_2}} \frac{1}{|I_{\vec{Q}}|^{1/2}} \langle f_2, \Phi^\alpha_{Q_1, 2} \rangle \langle f_3, \Phi^\alpha_{Q_2,3} \rangle \Phi^\alpha_{Q_3,4} d\alpha, \Phi_{P_2,0} \right \rangle  \right|^2
\end{align*}
in addition to the property that for each subtree $T^\prime \subset T \in \mathbb{T}$

\begin{align*}
& \frac{1}{|I_{T^\prime}|} \sum_{\vec{P} \in T^\prime} \left| \left \langle \int_0^1 \sum_{\vec{Q} \in \mathbb{Q}^{\tilde{d}}: \omega_{Q_3} \supset \supset \omega_{P_2}} \frac{1}{|I_{\vec{Q}}|^{1/2}} \langle f_2, \Phi^\alpha_{Q_1, 2} \rangle \langle f_3, \Phi^\alpha_{Q_2,3} \rangle \Phi^\alpha_{Q_3,4} d\alpha, \Phi_{P_2,0} \right \rangle  \right|^2\\ \lesssim& \frac{1}{|I_{T}|} \sum_{\vec{P} \in T} \left| \left \langle \int_0^1 \sum_{\vec{Q} \in \mathbb{Q}^{\tilde{d}}: \omega_{Q_3} \supset \supset \omega_{P_2}} \frac{1}{|I_{\vec{Q}}|^{1/2}} \langle f_2, \Phi^\alpha_{Q_1, 2} \rangle \langle f_3, \Phi^\alpha_{Q_2,3} \rangle \Phi^\alpha_{Q_3,4} d\alpha, \Phi_{P_2,0} \right \rangle  \right|^2 .
\end{align*}
 Then for every $0 < \theta <1$, 

\begin{eqnarray*}
\sum_{T \in \mathbb{T}} \sum_{\vec{P} \in T} \left|\left \langle \int_0^1 \sum_{\vec{Q} \in \mathbb{Q}^{\tilde{d}}: \omega_{Q_3} \supset \supset \omega_{P_2}} \frac{1}{|I_{\vec{Q}}|^{1/2}} \langle f_2, \Phi^\alpha_{Q_1, 2} \rangle \langle f_3, \Phi^\alpha_{Q_2,3} \rangle \Phi^\alpha_{Q_3,4} d\alpha, \Phi_{P_2,0} \right \rangle  \right|^2 \lesssim_\theta 2^{2 \tilde{d}} |E_2|^{1+\theta} |E_3|^{2-\theta}. 
\end{eqnarray*}

\end{lemma}
\begin{proof}
The argument is a slight modification of Lemma 9.2 in \cite{MR2127985}. There, the result is phrased as an estimate for the $3$-energy, which is equivalent to what we need to show except that $\supset \supset$ is replaced by $\supseteq$. However, this does not affect the validity of the proof found there. 
\end{proof}

Lemmas \ref{Est:BHTloc}, \ref{Est:BHTEnergy}, and \ref{Est:BiestEnergy} are the main ingredients needed for the following $\ell^2$ energy estimate: 
\begin{lemma}\label{Est:l2Energy}
Fix $d, \tilde{d} \geq 0$ along with $\mathfrak{d} \geq N_2(d, \tilde{d})$.   Then for any $0< \theta <1$, 

\begin{align*}
\sum_{T \in \mathcal{T}_{\mathfrak{d},2}^{d, \tilde{d}}} |I_T| \lesssim_\theta 2^{2 \mathfrak{d}} 2^{2 \tilde{d}} |E_2|^{1+\theta} |E_3|^{2 - \theta}.
 \end{align*}

\end{lemma}

\begin{proof}
Recall from Lemma \ref{TDL} that

\begin{align*}
\sum_{T \in \mathcal{T}_{\mathfrak{d},2}^{d, \tilde{d}}} |I_T| \lesssim \sum_{T \in \mathcal{T}_{\mathfrak{d},2, *}^{d, \tilde{d}}} |I_T| =  \sum_{T \in \mathcal{T}_{\mathfrak{d},2, *,+}^{d, \tilde{d}}} |I_T|  +  \sum_{T \in \mathcal{T}_{\mathfrak{d},2, *,-}^{d, \tilde{d}}} |I_T| . 
\end{align*}
We further decompose the trees in $\mathcal{T}_{\mathfrak{d},2, *,+}$ into the following union:

\begin{eqnarray*}
\mathcal{T}_{\mathfrak{d},2, *,+} = \mathcal{T}_{\mathfrak{d},2, *,+, I}  \bigcup \mathcal{T}_{\mathfrak{d},2, *,+, II}  \bigcup \mathcal{T}_{\mathfrak{d},2, *,+, III}
\end{eqnarray*}
where $\mathcal{T}_{\mathfrak{d},2, *,+, I}, \mathcal{T}_{\mathfrak{d},2, *,+, II}$, and $\mathcal{T}_{\mathfrak{d},2, *,+, III}$ are respectively given by 

\begin{eqnarray*}
&&  \left\{ T :  \sum_{\vec{P} \in T} \left|\left \langle \int_0^1 BHT^{\alpha, \mathbb{Q}^{\tilde{d}}}(f_2, f_3) d \alpha, \Phi_{P_2, 0}  \right \rangle \right|^2 \geq 2^{-2\mathfrak{d}-5} |I_T| \right\}  \\
&& \left\{ T  :  \sum_{\vec{P} \in T} \left| \left \langle \int_0^1 \sum_{ \substack{ \vec{Q} \in \mathbb{Q}^{\tilde{d}} \\  \omega_{Q_3} \supset \supset \omega_{P_2}}} \frac{1}{|I_{\vec{Q}}|^{1/2}} \langle f_2, \Phi^\alpha_{Q_1, 2} \rangle \langle f_3, \Phi^\alpha_{Q_2,3} \rangle \Phi^\alpha_{Q_3,4} d\alpha, \Phi_{P_2,0} \right \rangle \right|^2 \geq 2^{-2\mathfrak{d}-5} |I_T| \right\} \\ &&
\left\{ T  :  \sum_{\vec{P} \in T} \left| \left \langle \int_0^1 \sum_{ \substack{ \vec{Q} \in \mathbb{Q}^{\tilde{d}} \\  |\omega_{Q_3}| \simeq  |\omega_{P_2}|}} \frac{1}{|I_{\vec{Q}}|^{1/2}} \langle f_2, \Phi^\alpha_{Q_1, 2} \rangle \langle f_3, \Phi^\alpha_{Q_2,3} \rangle \Phi^\alpha_{Q_3,4} d\alpha, \Phi_{P_2,0} \right \rangle \right|^2 \geq 2^{-2\mathfrak{d}-5} |I_T| \right\}. 
\end{eqnarray*}
Similarly, we have the decomposition $\mathcal{T}_{\mathfrak{d},2, *,-} = \mathcal{T}_{\mathfrak{d},2, *,-, I}  \bigcup \mathcal{T}_{\mathfrak{d},2, *,-, II} \bigcup  \mathcal{T}_{\mathfrak{d},2, *,-, III} $.  Therefore, 

\begin{align*}
 \sum_{T \in \mathcal{T}_{\mathfrak{d},2, *}^{d, \tilde{d}}} |I_T| \leq& \sum_{T \in \mathcal{T}_{\mathfrak{d},2, *,+,I}^{d, \tilde{d}}} |I_T| + \sum_{T \in \mathcal{T}_{\mathfrak{d},2, *,+,II}^{d, \tilde{d}}} |I_T| +\sum_{T \in \mathcal{T}_{\mathfrak{d},2, *,+,III}^{d, \tilde{d}}} |I_T| \\+&  \sum_{T \in \mathcal{T}_{\mathfrak{d},2, *,-, I}^{d, \tilde{d}}} |I_T| + \sum_{T \in \mathcal{T}_{\mathfrak{d},2, *,-,II}^{d, \tilde{d}}} |I_T|+\sum_{T \in \mathcal{T}_{\mathfrak{d},2, *,-,III}^{d, \tilde{d}}} |I_T|.
\end{align*} 
By Lemmas \ref{Est:BHTloc} and \ref{Est:BHTEnergy}, 

\begin{align*}
\sum_{T \in \mathcal{T}_{\mathfrak{d},2, *,+,I}^{d, \tilde{d}}} |I_T| \lesssim & 2^{2 \mathfrak{d}} \sum_{T \in \mathcal{T}_{\mathfrak{d},2, *,+, I}}  \sum_{\vec{P} \in T} \left|\left \langle \int_0^1 BHT^{\alpha, \mathbb{Q}^{\tilde{d}}}(f_2, f_3) d \alpha, \Phi_{P_2, 0}  \right \rangle \right|^2\\  \lesssim& 2^{2 \mathfrak{d}}  \left| \left| \int_0^1 BHT^{\alpha, \mathbb{Q}^{\tilde{d}}}(f_2, f_3) d \alpha \right| \right|_2^2 \\ \lesssim& 2^{2 \mathfrak{d}} 2^{2\tilde{d}} |E_2|^{1+\theta} |E_3|^{2-\theta}. 
\end{align*}
Moreover, by Lemma \ref{Est:BiestEnergy},

\begin{align}
 & \sum_{T \in \mathcal{T}_{\mathfrak{d},2, *,+,II}^{d, \tilde{d}}} |I_T| \\ \lesssim&   2^{2 \mathfrak{d}} \sum_{T \in \mathcal{T}_{\mathfrak{d},2, *,+, II}}  \sum_{\vec{P} \in T} \left|\left \langle \sum_{\vec{Q} \in \mathbb{Q}^{\tilde{d}}: \omega_{Q_3} \supset \supset \omega_{P_2}} \frac{1}{|I_{\vec{Q}}|^{1/2}} \langle f_2, \Phi^\alpha_{Q_1, 2} \rangle \langle f_3, \Phi^\alpha_{Q_2,3} \rangle \Phi_{Q_3,4}, \Phi^\alpha_{P_2,0} \right \rangle  \right|^2 \nonumber \\ \lesssim& 2^{2 \mathfrak{d}} 2^{2 \tilde{d}} |E_2|^{1+\theta} |E_3|^{2-\theta} \nonumber .
 \end{align}
 It is also straightforward to observe
 
 \begin{align}\label{Est:BiestEnergy,III}
 & \sum_{T \in \mathcal{T}_{\mathfrak{d},2, *,+,III}^{d, \tilde{d}}} |I_T| \\ \lesssim& 2^{2 \mathfrak{d}} \sum_{T \in \mathcal{T}_{\mathfrak{d},2, *,+,III}^{d, \tilde{d}}}\sum_{\vec{P} \in T} \left| \left \langle \sum_{\vec{Q} \in \mathbb{Q}^{\tilde{d}}: |\omega_{Q_3}| \simeq |\omega_{P_2}|} \frac{1}{|I_{\vec{Q}}|^{1/2}} \langle f_2, \Phi^\alpha_{Q_1, 2} \rangle \langle f_3, \Phi^\alpha_{Q_2,3} \rangle \Phi_{Q_3,4}, \Phi^\alpha_{P_2,0} \right \rangle \right|^2 \nonumber \\ \lesssim&  \sum_{\vec{Q} \in \mathbb{Q}^{\tilde{d}}} \frac{1}{|I_{\vec{Q}}|} |\langle f_2, \Phi^\alpha_{Q_1, 2} \rangle|^2  |\langle f_3, \Phi^\alpha_{Q_2,3} \rangle|^2 \nonumber \\ \lesssim& 2^{2 \mathfrak{d}} 2^{2\tilde{d}} |E_2|^{1+\theta} |E_3|^{2-\theta} \nonumber.  
\end{align}
The sums $\sum_{T \in \mathcal{T}_{\mathfrak{d},2, *,-,I}^{d, \tilde{d}}} |I_T|$, $\sum_{T \in \mathcal{T}_{\mathfrak{d},2, *,-,II}^{d, \tilde{d}}} |I_T|$, and $\sum_{T \in \mathcal{T}_{\mathfrak{d},2, *,-,III}^{d, \tilde{d}}} |I_T|$ are handled similarly, so these details are omitted. 
\end{proof}

\subsubsection{$\ell^1$-Energy Estimate}
The standard $BHT$ energy method involves obtaining $l^2$ estimates of the form

\begin{eqnarray*}
\sum_{T \in \mathcal{T}} |I_T| \lesssim 2^{2n} \sum_{T \in \mathcal{T}} \sum_{\vec{P} \in T} |\langle f , \Phi_P \rangle|^2 \lesssim 2^{2n} ||f||_2^2 ,
\end{eqnarray*}
where the trees $T \in \mathcal{T}$ are strongly disjoint. Because our tri-tile collection $\mathbb{P}$ is not rank-1, we are not able to pass the analysis directly to ``Biest" type arguments. Instead, we shall need an $l^1$-type energy estimate. Before stating this result precisely, we shall need to introduce terminology: 

\begin{definition} For a given collection of tri-tiles $\mathbb{P}^0 \subset \mathbb{P}$ and $\tilde{\mathbb{Q}} \subset \mathbb{Q}$, let
  
  \begin{align*}
\mathbb{P}^0_{\mathfrak{b}}(\tilde{\mathbb{Q}})= \left\{ \vec{P} \in \mathbb{P}^0: \frac{1}{|I_{\vec{P}}|^{1/2}} \left |\left \langle \int_0^1 BHT^{\alpha,\tilde{\mathbb{Q}}}_{\omega_{P_2}}(f_2, f_3)d \alpha, \Phi_{P_2,0} \right \rangle \right| \simeq 2^{-\mathfrak{b}} \right\} 
\end{align*}
\end{definition}
Our main $\ell^1$-energy result is the following:
\begin{theorem}\label{l1T}
Let $\mathbb{P}^0 \subset \mathbb{P}$ be any collection of tri-tiles for which $\{P_2\}_{\vec{P} \in \mathbb{P}^0}$ is pairwise disjoint. Moreover, let $\tilde{\mathbb{Q}} \subset \mathbb{Q}$. Then for any $0< \tilde{\epsilon} < 1$
\begin{align*}
\sum_{\vec{P} \in \mathbb{P}_{\mathfrak{b}}^0(\tilde{\mathbb{Q}})} |I_{\vec{P}}| 
 \lesssim_{\tilde{\epsilon}} 2^{\frac{\mathfrak{b}}{1-\tilde{\epsilon}}} |E_2|^{ 1/2} |E_3|^{1/2}.
\end{align*}
\end{theorem}

 \begin{corollary}\label{Cor0}
Let $\mathbb{P}^{d, \tilde{d}}_{\mathfrak{d},2,*} = \bigcup_{T \in \mathcal{T}^{d, \tilde{d}}_{ \mathfrak{d}, 2, *}} T$. For every $d, \tilde{d} \geq 0$, $\mathfrak{b} \gtrsim \mathfrak{d} \geq N_2(d, \tilde{d})$ and $0< \tilde{\epsilon} < 1$,
 \begin{eqnarray*}
 \sum_{\vec{P} \in  \mathbb{P}^{d, \tilde{d}}_{\mathfrak{d},2,*,\mathfrak{b}}( \mathbb{Q}^{\tilde{d}} )}| I_{\vec{P}}| \lesssim_{\tilde{\epsilon}} 2^{\frac{\mathfrak{b}}{1-\tilde{\epsilon}}} |E_2|^{ 1/2} |E_3|^{ 1/2}.
 \end{eqnarray*}
 Note that $\mathfrak{b} \ll  \mathfrak{d}$ implies $ \mathbb{P}^{d, \tilde{d}}_{\mathfrak{d},2,*,\mathfrak{b}}( \mathbb{Q}^{\tilde{d}})= \emptyset$. 
 \end{corollary}
 \begin{proof}
This is immediate from Theorem \ref{l1T}. 
\end{proof}
 We postpone the proof of Theorem \ref{l1T} to state another corollary and discuss how it is used in the proof of Theorem \ref{MT}. 
  \begin{corollary}\label{Cor1}
For every $d, \tilde{d} \geq 0, \mathfrak{d} \geq N_2(d, \tilde{d})$, and $0 < \tilde{\epsilon} \ll 1$, 

\begin{eqnarray*}
\sum_{T \in \mathcal{T}^{d, \tilde{d}}_{\mathfrak{d} ,2}}  |I_T| \lesssim_{\tilde{\epsilon}} 2^{ \frac{ \mathfrak{d}}{1-\tilde{\epsilon}}} |E_2|^{ 1/2} |E_3|^{1/2}.
\end{eqnarray*}
\end{corollary}

 \begin{proof}
 Using Corollary \ref{Cor0}, we have that
 \begin{align*}
\sum_{T \in \mathcal{T}^{d, \tilde{d}}_{\mathfrak{d} ,2}}  |I_T|& \lesssim 2^{2 \mathfrak{d}} \sum_{T \in \mathcal{T}^{d, \tilde{d}}_{\mathfrak{d} ,2,*}}~ \sum_{\vec{P} \in T} \left |\left \langle \int_0^1 BHT^{\alpha, \mathbb{Q}^{\tilde{d}}}_{\omega_{P_2}} (f_2, f_3) d \alpha, \Phi_{P_2,0}^{2} \right\rangle \right|^2 \\ &= 2^{ 2 \mathfrak{d}} \sum_{\mathfrak{b} \gtrsim \mathfrak{d}} \sum_{\vec{P} \in \mathbb{P}^{d, \tilde{d}}_{\mathfrak{d},2,*,\mathfrak{b}}( \mathbb{Q}^{\tilde{d}})} \left |\left \langle \int_0^1 BHT^{\alpha, \mathbb{Q}^{\tilde{d}}}_{\omega_{P_2}} (f_2, f_3)d \alpha, \Phi_{P_2,0} \right\rangle \right|^2\\& \lesssim  2^{ 2 \mathfrak{d}}\sum_{\mathfrak{b} \gtrsim \mathfrak{d}}2^{-2\mathfrak{b}}  \sum_{\vec{P} \in \mathbb{P}^{d, \tilde{d}}_{\mathfrak{d},2,*,\mathfrak{b}}( \mathbb{Q}^{\tilde{d}})} |I_{\vec{P}}| \\ &\lesssim 2^{2 \mathfrak{d}} \sum_{\mathfrak{b} \geq \mathfrak{d}}2^{-\left[ 2-\frac{1}{1-\tilde{\epsilon}}\right] \mathfrak{b}}   |E_2|^{ 1/2} |E_3|^{ 1/2}  \\ &\lesssim 2^{ \frac{ \mathfrak{d}}{1-\tilde{\epsilon}}} |E_2|^{ 1/2} |E_3|^{1/2}.
\end{align*}
 \end{proof}
 We now discuss how Corollary \ref{Cor1} is used in proof of Theorem \ref{MT}. Interpolating the $\ell^2$-energy bound  $\sum_{T \in \mathcal{T}_{\mathfrak{d},2}^{d, \tilde{d}}} |I_T| \lesssim 2^{2 \mathfrak{d}} |E_2|^{3/2} |E_3|^{3/2}$ with the $\ell^1$-energy bound $\sum_{T \in \mathcal{T}_{\mathfrak{d},2}^{d, \tilde{d}}} |I_T| \lesssim 2^{ \sim\mathfrak{d}} |E_2|^{1/2} |E_3|^{1/2}$ ensures

\begin{eqnarray*}
\sum_{T \in \mathcal{T}_{\mathfrak{d}, 2}^{d, \tilde{d}}} |I_T| \lesssim 2^{\sim 3\mathfrak{d}/2 } |E_2| |E_3|.
\end{eqnarray*}
It follows that we should have for every $0 \leq \theta_1 , \theta_2 \leq 1$ such that $\theta_1 + \theta_2 =1$
\begin{align*}
|\Lambda_2^{\mathbb{P},\mathbb{Q}} (f_1, f_2, f_3, f_4)|& \lesssim \sum_{n_1,n_4, \mathfrak{d} \geq 0} 2^{-n_1}2^{-n_4}  2^{-\mathfrak{d}} \min\left\{ 2^{2n_1} |E_1|, 2^{\sim 3/2 \mathfrak{d}} |E_2| |E_3| \right\} \\ &\lesssim \sum_{n_1,n_4, \mathfrak{d} \geq 0} 2^{-n_1(1-2\theta_1)}2^{-n_4}  2^{-\mathfrak{d}(1- (\sim 3\theta_2/2))}  |E_1|^{\theta_1}|E_2|^{ \theta_2} |E_3|^{ \theta_2} .
\end{align*}
Choosing $\theta_1 \simeq 1/3, \theta_2 \simeq 2/3$ gives $|\Lambda_2^{\mathbb{P}, \mathbb{Q}} (f_1, f_2, f_3, f_4)| \lesssim |E_1|^{\sim 2/3} |E_2|^{\sim 2/3} |E_3|^{\sim 2/3}$, so that $C^{1,1,-2}$ should map into $L^r(\mathbb{R})$ for all $r$ in a small neighborhood near $1/2$. With this sketch in mind, it  remains to fill in the details.

 \begin{proof}{[Theorem \ref{l1T}]}
 For convenience, set $\mathbb{P}_{\mathfrak{d}}^0 (\tilde{\mathbb{Q}})= \mathbb{P}_{\mathfrak{d}}^0$. Now using the definition of $\mathbb{P}^0_{\mathfrak{b}}$ and dualizing, we obtain
\begin{align*}
\sum_{\vec{P} \in \mathbb{P}^0_{\mathfrak{b}}} |I_{\vec{P}}| &\simeq 2^{\mathfrak{b}} \sum_{\vec{P} \in  \mathbb{P}^0_{\mathfrak{b}}} \left| \left \langle \int_0^1  BHT^{\alpha, \mathbb{Q}^{\tilde{d}}}_{\omega_{P_2}} (f_2, f_3), \Phi_{P_2,0}^\infty \right \rangle \right| \\ &:=2^{\mathfrak{b}} \sum_{\vec{P} \in \mathbb{P}^0_{\mathfrak{b}}} h_{\vec{P}} \left \langle \int_0^1 BHT^\alpha_{\omega_{P_2}} (f_2, f_3) d\alpha, \Phi_{P_2, 0}^\infty \right\rangle, 
\end{align*}
where $|h_{\vec{P}}| =1$ for all $\vec{P} \in \mathbb{P}^0_{\mathfrak{b}}$ and $\Phi_{P_2,0}:= |I_{\vec{P}}|^{1/2} |\Phi_{P_2,0}|$ is $L^\infty$-normalized.  Rewriting the above display, we find

\begin{align*}
2^{-\mathfrak{b}} \sum_{\vec{P} \in \mathbb{P}^0_{\mathfrak{b}}} |I_{\vec{P}}| &\simeq \int_0^1  \sum_{\vec{P} \in \mathbb{P}^0_{\mathfrak{b}}}~ \sum_{ \substack{ \vec{Q} \in \mathbb{Q}^{\tilde{d}} \\  \omega_{Q_3} \subset  \subset \omega_{P_2}}} \frac{1}{|I_{\vec{Q}}|^{1/2}} \langle f_2, \Phi^\alpha_{Q_1,2} \rangle \langle f_3, \Phi^\alpha_{Q_2,3} \rangle \langle \Phi^\alpha_{Q_3,5}, h_{\vec{P}} \chi^M_{I_{\vec{P}}}  \rangle d \alpha  \\ &= \int_0^1 \sum_{\vec{Q} \in \mathbb{Q}^{\tilde{d}}}~\frac{1}{|I_{\vec{Q}}|^{1/2}} \langle f_2, \Phi^\alpha_{Q_1,1} \rangle \langle f_3, \Phi^\alpha_{Q_2,2} \rangle \left \langle \Phi^\alpha_{Q_3,5}, \sum_{ \substack{ \vec{P} \in \mathbb{P}^0_{\mathfrak{b}} \\  \omega_{P_2} \supset \supset\omega_{Q_3}}} h_{\vec{P}} \Phi_{P_2, 0}^\infty  \right \rangle d \alpha.
\end{align*}
Observe that when the $\mathbb{Q}$-tiles are restricted to a single tree, the sum over $\vec{P} \in \mathbb{P}^{\mathfrak{b}}_{\mathfrak{d}}$ containing a frequency of the tree satisfies $\sum_{\vec{P} \in \mathbb{P}^\mathfrak{b}_{\mathfrak{d}} (T)} 1_{I_{\vec{P}}} \lesssim 1$.  Moreover, 

\begin{eqnarray*}
\left| \left| \sum_{\vec{P} \in \mathbb{P}^0_{\mathfrak{b}}: \omega_{P_2} \supset \supset \omega_{Q_3}, I_{\vec{P}} \subset I_{\vec{Q}} ~for ~some~\vec{Q} \in \mathbb{Q}^{\tilde{d}}}  h_{\vec{P}} \Phi_{P_2, 0}^\infty\right| \right|_2 ^2 \lesssim \sum_{\vec{P} \in \mathbb{P}^0_{\mathfrak{b}}} |I_{\vec{P}}|. 
\end{eqnarray*}
It suffices to prove the following estimate uniformly in $\alpha \in [0,1]$:
 
 \begin{align}
  & \left| \sum_{\vec{Q} \in \mathbb{Q}^{\tilde{d}}}~\frac{1}{|I_{\vec{Q}}|^{1/2}} \langle f_2, \Phi^\alpha_{Q_1,1} \rangle \langle f_3, \Phi^\alpha_{Q_2,2} \rangle \left \langle \Phi^\alpha_{Q_3,5}, \sum_{\vec{P} \in \mathbb{P}^0_{\mathfrak{b}}: \omega_{P_2} \supset  \supset \omega_{Q_3}}h_{\vec{P}} \Phi_{P_2, 0}^\infty   \right \rangle  \right|\label{Est:Goal}  \\& \lesssim  |E_2|^{\frac{1}{2}(1-\tilde{\epsilon})} |E_3|^{\frac{1}{2}(1-\tilde{\epsilon})} \left[ \sum_{\vec{P} \in \mathbb{P}^0_{\mathfrak{b}}} |I_{\vec{P}}| \right]^{\tilde{\epsilon}} \nonumber
 \end{align}
for some $0 < \tilde{\epsilon} <<1$. Indeed, as consequence, we would have

\begin{eqnarray*}
\sum_{\vec{P} \in \mathbb{P}^0_{\mathfrak{b}}} |I_{\vec{P}}| \lesssim_{\tilde{\epsilon}} 2^{\mathfrak{b}}  |E_2|^{\frac{1}{2} (1-\tilde{\epsilon})} |E_3|^{\frac{1}{2}(1-\tilde{\epsilon})} \left[ \sum_{\vec{P} \in \mathbb{P}^{\mathfrak{b}}_{\mathfrak{d}}} |I_{\vec{P}}| \right]^{\tilde{\epsilon}}, 
\end{eqnarray*}
which easily implies 

\begin{eqnarray*}
\sum_{\vec{P} \in \mathbb{P}^0_{\mathfrak{b}}} |I_{\vec{P}}|  \lesssim_{\tilde{\epsilon}}  2^{\frac{ \mathfrak{b}}{1-\tilde{\epsilon}}}  |E_2|^{1/2} |E_3
|^{ 1/2}.
\end{eqnarray*}
Estimate \eqref{Est:Goal} is eventually shown in Proposition \ref{P:Goal}. We first need to record several elementary results.
\begin{lemma}\label{L:Trivial-v2}
For all $M \gg 1$, 

\begin{eqnarray*}
 \left(\frac{1}{|I_T|} \int_\mathbb{R} \left| \sum_{\vec{P} \in \mathbb{P}^0_\mathfrak{b}(T)} \chi^M_{I_{\vec{P}}}\right|^2 \chi^M_{I_T} dx \right)^{1/2} \lesssim 1. 
 \end{eqnarray*}
 \end{lemma}
 \begin{proof}
 First observe that for any $M \gg 1$ and $|I_{P_1}| \leq |I_{P_2}|$
\begin{eqnarray}\label{Est:Trivial}
\int_\mathbb{R}\chi^{M}_{I_{P_1}} \chi^M_{I_{P_2}} dx \lesssim \frac{|I_{P_1}|}{ 1+ \left[ \frac{ dist( I_{P_1}, I_{P_2})}{|I_{P_2}|} \right]^M} \lesssim \int_\mathbb{R} 1_{I_{P_1}} \chi^M_{I_{P_2}} dx.
\end{eqnarray}
Estimate \eqref{Est:Trivial} then implies
 \begin{align*}
&    \frac{1}{|I_T|} \int_\mathbb{R} \left| \sum_{\vec{P} \in \mathbb{P}^0_\mathfrak{b}(T): I_{\vec{P}} \subset I_T +l |I_T|} \chi^M_{I_{\vec{P}}}\right|^2 \chi^M_{I_T} dx \\  \lesssim& \frac{1}{1+l^M} \frac{1}{|I_T|} \int_\mathbb{R}  \sum_{P_1, P_2 \in \mathbb{P}^0_\mathfrak{b}(T):  |I_{P_2}| \geq |I_{P_1}|: I_{P_1}, I_{P_2} \subset I_T + l|I_T|} \chi^M_{I_{P_1}} \chi^M_{I_{P_2}}   dx \\ \lesssim&\frac{1}{1+l^M}\frac{1}{|I_T|} \int_\mathbb{R}  \sum_{P_2 \in \mathbb{P}^0_\mathfrak{b}(T): I_{P_2} \subset I_T +l|I_T|} \chi^{M}_{I_{P_2}}  dx \\   \lesssim &\frac{1}{1+l^M}. 
\end{align*}
Consequently, 

\begin{align*}
 \left( \frac{1}{|I_T|} \int_\mathbb{R} \left| \sum_{\vec{P} \in \mathbb{P}^0_\mathfrak{b}(T)} \chi^M_{I_{\vec{P}}}\right|^2 \chi^M_{I_T} dx  \right)^{1/2}  &\lesssim \sum_{l \in \mathbb{Z}} \left( \frac{1}{|I_T|} \int_\mathbb{R} \left| \sum_{\vec{P} \in \mathbb{P}^0_\mathfrak{b}(T):I_{\vec{P}} \subset I_T | l |I_T|} \chi^M_{I_{\vec{P}}}\right|^2 \chi^M_{I_T} dx  \right)^{1/2} \\ & \lesssim \sum_{l \in \mathbb{Z}} \frac{1}{1+l^M} \\ &\lesssim 1. 
 \end{align*}

\end{proof}
Before proving our next result, we shall need

\begin{definition}
For $j \in \{1,2,3\}$ and $\tilde{\mathbb{Q}} \subset \mathbb{Q}$, let  

\begin{align*}
S_j(\{a_{\vec{Q},j}\}_{\vec{Q} \in \mathbb{Q}}) =\left( \sup_{T \subset \tilde{Q}} \frac{1}{|I_T|} \sum_{\vec{Q}\in T} |c_{\vec{Q}}|^2 \right)^{1/2}
\end{align*}
where the supremum is over all $i$-trees $T \subset \tilde{\mathbb{Q}}$ for some $i \not = j$.

\end{definition}

\begin{prop}\label{P89}
For $j \in \{1,2\}$ set $a_{\vec{Q},1}= \langle f_2, \Phi_{Q_1,2} \rangle; a_{\vec{Q},2}= \langle f_3, \Phi_{Q_2, 3} \rangle.$ 
Then 

\begin{align*}
S_1\left(\{a_{\vec{Q},1}\}_{\vec{Q} \in \tilde{\mathbb{Q}}}  \right) \lesssim 1 , S_2\left(\{a_{\vec{Q},2}\}_{\vec{Q} \in \tilde{\mathbb{Q}}} \right) \lesssim 1.
\end{align*}

\end{prop}
\begin{proof}
This standard energy result follows from Lemma \ref{L:John-Nirenberg}.

\end{proof}

\begin{prop}\label{P90}
Let $\left\{h_{\vec{P}} \right\}_{\vec{P} \in \mathbb{P}^0_{\mathfrak{b}}}$ satisfy $|h_{\vec{P}}| =1$ for all $\vec{P}\in \mathbb{P}^0_{\mathfrak{b}}$ and 

\begin{align*}
a_{\vec{Q},3}=\left \langle \Phi^\alpha_{Q_3,5}, \sum_{\vec{P} \in \mathbb{P}^0_{\mathfrak{b}}: \omega_{P_2} \supset  \supset \omega_{Q_3} }h_{\vec{P}} \Phi_{P_2, 0}^ \infty \right \rangle . 
\end{align*}
Then $S_3(\{a_{\vec{Q},3} \}_{\vec{Q} \in \mathbb{Q}^{\tilde{d}}}) \lesssim 1$.

\end{prop}
\begin{proof}
 Fix a $1$- or $2$-tree $T \subset \mathbb{Q}^{\tilde{d}}$.  Observe that since each $\Phi^\infty_{P_2,0}$ satisfies 
\begin{eqnarray*}
supp~\hat {\Phi}^\infty_{P_2,0} \subset \left [ c_{\omega_{P_2}} - \frac{9 }{20}  |\omega_{P_2}|, c_{\omega_{P_2}} + \frac{9}{20 }| \omega_{P_2}| \right],
\end{eqnarray*}

\begin{eqnarray*}
\mathbb{P}^0_\mathfrak{b}(T) := \left\{ \vec{P} \in \mathbb{P}^0_\mathfrak{b} : \exists \vec{Q} \in T, \langle \Phi^\alpha_{Q_3, 5}, \Phi_{P_2, 0}^\infty \rangle \not = 0,  |\omega_{P_2}| \gg  |\omega_{Q_3}| \right\}
\end{eqnarray*}
 consists of tiles with disjoint time projections. Indeed, it is easy to check that for large enough implicit constant in the definition of $\mathbb{P}_{\mathfrak{b}}^0(T)$ (depending on the implicit constants in the definition of a tree), every tile $P_2: \vec{P} \in \mathbb{P}^0_\mathfrak{b} (T)$ must contain the $\mathbb{Q}^{\tilde{d}}$-tree 's top frequency band $\omega_{T_3}$, and $\mathbb{P}^0_\mathfrak{b}$ consists of disjoint tri-tiles.  
By the frequency restriction $\omega_{P_2} \supset \supset \omega_{Q_3}$ and disjointness of the tiles $\left\{P_2 : \vec{P} \in \mathbb{P}^0_{\mathfrak{b}} \right\}$, we observe

\begin{align*}
  \left( \sum_{\vec{Q} \in T} \left|  \left \langle \Phi^\alpha_{Q_3,5}, \sum_{ \substack{ \vec{P} \in \mathbb{P}^{0}_{\mathfrak{b}} \\  \omega_{P_2} \supset  \supset \omega_{Q_3}}} h_{\vec{P}} \Phi_{P_2, 0}^ \infty \right \rangle \right|^2 \right)^{1/2} =&\left( \sum_{\vec{Q} \in T} \left|  \left \langle \Phi^\alpha_{Q_3,5}, \sum_{ \substack{ \vec{P} \in \mathbb{P}^0_\mathfrak{b}(T) \\  \omega_{P_2} \supset \supset \omega_{Q_3}}} h_{\vec{P}} \Phi_{P_2, 0}^\infty  \right \rangle \right|^2 \right)^{1/2} \\\lesssim & \left( \sum_{\vec{Q} \in T} \left|  \left \langle \Phi^\alpha_{Q_3,5}, \sum_{\vec{P} \in \mathbb{P}^0_\mathfrak{b}(T)}h_{\vec{P}} \Phi_{P_2, 0}^\infty  \right \rangle \right|^2 \right)^{1/2}\\+&  \left( \sum_{\vec{Q} \in T} \left|  \left \langle \Phi^\alpha_{Q_3,5}, \sum_{ \substack{ \vec{P} \in \mathbb{P}^0_\mathfrak{b}(T) \\ |\omega_{P_2}| \simeq  |\omega_{Q_3}|}} h_{\vec{P}} \Phi_{P_2, 0}^\infty \right \rangle \right|^2 \right)^{1/2} \\ =:&~ I + II. 
 \end{align*}
For the last line, it is straightforward to check that for large enough implicit constant 

\begin{align*}
\left\{ \vec{P} \in \mathbb{P}^0_\mathfrak{b}(T) : \exists \vec{Q} \in T, |\omega_{P_2}| \ll |\omega_{Q_3}|, \langle \Phi^\alpha_{Q_3, 5}, \Phi_{P_2, 0}^\infty \rangle \not = 0 \right\} = \emptyset. 
\end{align*}
Our goal is then to show $I, II \lesssim 1$, in which case $2^{-n_3} \lesssim 1$. 
 For term $I$, note that for some $M \gg 1$
 
 \begin{align*}
 I &\lesssim   \left( \frac{1}{|I_T|} \sum_{\vec{Q} \in T} \left|  \left \langle \Phi^\alpha_{Q_3,5}, \sum_{\vec{P} \in \mathbb{P}^0_\mathfrak{b}(T)}h_{\vec{P}}  \Phi_{P_2, 0}^\infty \right \rangle \right|^2 \right)^{1/2}  \\ &\lesssim \left(  \frac{1}{|I_T|} \int_\mathbb{R} \left| \left|  \sum_{\vec{P} \in \mathbb{P}^0_\mathfrak{b}(T)}  \chi^M_{I_{\vec{P}}}\chi^M_{I_T} \right| \right|^2_2  dx \right)^{1/2} \\& \lesssim 1
 \end{align*}
 where the last line holds by Lemma \ref{L:Trivial-v2}. To prove Proposition \ref{P90}, it suffices to bound $II$. For this, observe
 
 \begin{align*}
 II &\leq \sum_{|k| \leq C} \left( \frac{1}{|I_T|} \sum_{\vec{Q} \in T} \left|  \left \langle \Phi^\alpha_{Q_3,5}, \sum_{\substack{ \vec{P} \in \mathbb{P}^0_\mathfrak{b}(T) \\ |\omega_{P_2}| = 2^k  |\omega_{Q_3}|}}h_{\vec{P}} \Phi_{P_2, 0}^\infty  \right \rangle \right|^2 \right)^{1/2} . 
 \end{align*}
For each scale $\lambda \in \mathbb{Z}$, we easily obtain for some $M \gg 1$
 
 \begin{align*}
 \sum_{\vec{Q} \in T: |\omega_{Q_3}| = 2^{\lambda}} \left|  \left \langle \Phi^\alpha_{Q_3}, \sum_{\substack{ \vec{P} \in \mathbb{P}^0_\mathfrak{b}(T) \\ |\omega_{P_2}| = 2^k  |\omega_{Q_3}|}} h_{\vec{P}} \Phi_{P_2, 0}^\infty \right \rangle \right|^2 & \lesssim \left| \left| \left[ \sum_{\vec{P} \in \mathbb{P}^0_{\mathfrak{b}}(T):|\omega_{P_2}| = 2^{k+\lambda}}\chi^M_{I_{\vec{P}}} \right] \chi^M_{I_T} \right| \right|_2^2
 \end{align*}
Summing this last inequality over all $\lambda \in \mathbb{Z}$ and using Lemma \ref{L:Trivial-v2} again yields
 
 \begin{align*}
II  &\lesssim\frac{1}{|I_T|^{1/2}} \sum_{|k| \leq C} \left| \left|  \left( \sum_{\lambda \in \mathbb{Z}} \left| \sum_{\substack{ \vec{P} \in \mathbb{P}^0_\mathfrak{b}(T) \\ |\omega_{P_2}| = 2^{k+\lambda}}}  \chi^M_{I_{\vec{P}}} \right|^2 \right)^{1/2} \chi^M_{I_T} \right| \right|_2 \\ &\lesssim \frac{1}{|I_T|^{1/2}} \left| \left|  \sum_{\vec{P} \in \mathbb{P}^0_\mathfrak{b}(T)} \chi^M_{I_{\vec{P}}}  \chi^M_{I_T} \right| \right|_2 \\ & \lesssim 1. 
 \end{align*}
 \end{proof}

\subsubsection{$3$-Energy Bound} We begin by recalling the definition of energy from section 6.3 of \cite{MR3052499}:

\begin{definition}[Energy]\label{Def:Energy}
Let $j \in \{1,2,3\}$ and $\tilde{\mathbb{Q}} \subset \mathbb{Q}$. Then

\begin{align*}
\mathcal{E}_j\left( \{a_{Q_j}\}_{\vec{Q} \in \tilde{\mathbb{Q}}} \right) := \sup_n \sup_{\mathbb{T}} 2^n \left( \sum_{T \in \mathbb{T}} |I_T| \right)^{1/2} 
\end{align*}
where the supremum in $\mathbb{T}$ ranges over all chains of strongly $j$-disjoint $i$-trees $T \subset \tilde{\mathbb{Q}}$ for which $i \not =j$ and 

\begin{align*}
\sum_{\vec{Q} \in T} |a_{Q_j}|^2 &\leq 2^n |I_T| \\ 
\sum_{\vec{Q} \in T^\prime} |a_{Q_j}|^2 &\geq 2^{n+1} |I_{T^\prime}|\qquad \forall~\text{subtrees}~T^\prime \subset T. 
\end{align*}

\end{definition}
\begin{lemma}\label{L:Energy-3}
The following $3$-energy estimate holds:

\begin{align*}
\mathcal{E}_3 \left( \left\{\left \langle \sum_{\vec{P} \in \mathbb{P}^0_{\mathfrak{b}}: \omega_{P_2} \supset \supset \omega_{Q_3}}h_{\vec{P}} \Phi_{P_2,0}^\infty , \Phi^\alpha_{Q_3,5}\right \rangle \right\}_{\vec{Q} \in \tilde{\mathbb{Q}}} \right) \lesssim \left( \sum_{\vec{P} \in \mathbb{P}^0_\mathfrak{b}} |I_{\vec{P}}| \right)^{1/2}. 
\end{align*}

\end{lemma}

\begin{proof}
Setting $a_{\vec{Q},3}=\left \langle \sum_{\vec{P} \in \mathbb{P}^0_{\mathfrak{b}}: \omega_{P_2} \supset \supset \omega_{Q_3}}h_{\vec{P}} \Phi_{P_2,0}^\infty , \Phi^\alpha_{Q_3,5}\right \rangle$ for all $\vec{Q} \in \tilde{\mathbb{Q}}$, we observe from Definition \ref{Def:StronglyDisjoint} that for some collection of strongly $3$-disjoint trees $\mathbb{T}=\{T\}$, 

\begin{align*}
\mathcal{E}_3 \left(\{a_{\vec{Q},3}\}_{\vec{Q} \in \mathbb{Q}} \right) &\simeq \left[  \sum_{T \in \mathbb{T}} \sum_{\vec{Q} \in T}  \left| \left \langle \sum_{\vec{P} \in \mathbb{P}^0_{\mathfrak{b}}: \omega_{P_2} \supset \supset \omega_{Q_3}}h_{\vec{P}} \Phi_{P_2,0}^\infty , \Phi^\alpha_{Q_3,5}\right \rangle \right|^2  \right]^{1/2}
\end{align*}
where for all $T_1, T_2 \in \mathbb{T}$ and all subtrees $T_1^\prime \subset T_1$

\begin{align*}
&\frac{1}{|I_{T_1^\prime}| }\sum_{\vec{Q} \in T_1^\prime}  \left| \left \langle \sum_{\vec{P} \in \mathbb{P}^0_{\mathfrak{b}}: \omega_{P_2} \supset \supset \omega_{Q_3}}h_{\vec{P}} \Phi_{P_2,0}^\infty , \Phi^\alpha_{Q_3,5}\right \rangle \right|^2 \\ \lesssim & \frac{1}{|I_{T_1}| }\sum_{\vec{Q} \in T_1}  \left| \left \langle \sum_{\vec{P} \in \mathbb{P}^0_{\mathfrak{b}}: \omega_{P_2} \supset \supset \omega_{Q_3}}h_{\vec{P}} \Phi_{P_2,0}^\infty , \Phi^\alpha_{Q_3,5}\right \rangle \right|^2 \\ \simeq & \frac{1}{|I_{T_2}| }\sum_{\vec{Q} \in T_2}  \left| \left \langle \sum_{\vec{P} \in \mathbb{P}^0_{\mathfrak{b}}: \omega_{P_2} \supset \supset \omega_{Q_3}}h_{\vec{P}} \Phi_{P_2,0}^\infty , \Phi^\alpha_{Q_3,5}\right \rangle \right|^2.
\end{align*}
Defining for all $\vec{Q} \in \bigcup_{T \in \mathbb{T}} T$

\begin{eqnarray*}
c_{Q_3} :=  \overline{ \left \langle \sum_{\vec{P} \in \mathbb{P}^0_{\mathfrak{b}}: \omega_{P_2} \supset \supset \omega_{Q_3}}h_{\vec{P}} \Phi_{P_2, 0}^\infty , \Phi^\alpha_{Q_3,5}\right \rangle} \cdot \left[  \sum_{T \in \mathbb{T}} \sum_{\vec{Q} \in T}  \left| \left \langle \sum_{\vec{P} \in \mathbb{P}^0_{\mathfrak{b}}: \omega_{P_2} \supset \supset  \omega_{Q_3}}h_{\vec{P}} \Phi_{P_2, 0}^\infty, \Phi^\alpha_{Q_3,5}\right \rangle \right|^2  \right]^{-1/2}, 
\end{eqnarray*}
we obtain that for all $\vec{Q} \in \bigcup_{T \in \mathbb{T}} T$

\begin{eqnarray}
|c_{Q_3}| \simeq 2^{n_3}  \left| \left \langle \sum_{\vec{P} \in \mathbb{P}^0_{\mathfrak{b}}: \omega_{P_2} \supset \supset \omega_{Q_3}}h_{\vec{P}} \Phi_{P_2,0}^\infty, \Phi^\alpha_{Q_3,5}\right \rangle \right| \left[ \sum_{T \in \mathbb{T}} |I_T| \right]^{-1/2}
\end{eqnarray}
where $n_3$ realizes the supremum in Definition \ref{Def:Energy}. By construction it follows that for all $T \in \mathbb{T}$

\begin{align*}
 2^{-2n_3} \simeq \frac{1}{|I_T| }\sum_{\vec{Q} \in T}  \left| \left \langle \sum_{\vec{P} \in \mathbb{P}^0_{\mathfrak{b}}: \omega_{P_2} \supset \supset \omega_{Q_3}}h_{\vec{P}} \Phi_{P_2,0}^\infty ,\Phi^\alpha_{Q_3,5}\right \rangle \right|^2
\end{align*}
and for any subtree $T^\prime \subset T$, 

\begin{eqnarray}
\sum_{\vec{Q} \in T^\prime} |c_{Q_3}|^2 \lesssim \frac{ |I_{T^\prime}|}{\sum_{T \in \mathbb{T}} |I_T| }.
\end{eqnarray}
In particular, selecting $T^\prime = \{\vec{Q}\}$ yields

\begin{eqnarray}\label{Est:CQ}
|c_{Q_3}| \lesssim \frac{|I_{\vec{Q}}|^{1/2}}{\left[\sum_{T \in \mathbb{T}}|I_T|\right]^{1/2}}. 
\end{eqnarray}
Abbreviating $\mathcal{E}_3 \left(\{a_{\vec{Q},3}\}_{\vec{Q} \in \mathbb{Q}^{\tilde{d}}}\right) = \mathcal{E}_3$, we obtain 
\begin{align}
\mathcal{E}_3 &\simeq \sum_{T \in \mathbb{T}} \sum_{\vec{Q} \in T} \left \langle \sum_{\substack{ \vec{P} \in \mathbb{P}^0_{\mathfrak{b}} \\  \omega_{P_2} \supset \supset \omega_{Q_3}}} h_{\vec{P}} \Phi_{P_2, 0}^\infty , c_{Q_3} \Phi^\alpha_{Q_3,5}\right \rangle \nonumber \\ &= \sum_{\vec{P} \in \mathbb{P}^0_{\mathfrak{b}}} \left \langle h_{\vec{P}}\Phi^\infty_{P_2,0}, \sum_{T \in \mathbb{T}} \sum_{ \substack{ \vec{Q} \in T \\  \omega_{Q_3} \subset \subset \omega_{P_2}}} c_{Q_3} \Phi^\alpha_{Q_3,5}\right \rangle \nonumber \\ &= \left \langle  \sum_{\vec{P} \in \mathbb{P}^0_{\mathfrak{b}}} h_{\vec{P}}   \Phi^\infty_{P_2,0},  \sum_{T \in \mathbb{T}} \sum_{\vec{Q} \in T} c_{Q_3} \Phi^\alpha_{Q_3,5}\right \rangle \nonumber \\&+   \sum_{\vec{P} \in \mathbb{P}^0_{\mathfrak{b}}} \left \langle h_{\vec{P}}  \Phi^\infty_{P_2,0},  \sum_{T \in \mathbb{T}} \sum_{\substack{ \vec{Q} \in T \\  |\omega_{Q_3}| \gtrsim |\omega_{P_2}|}} c_{Q_3} \Phi^\alpha_{Q_3,5}\right \rangle \nonumber \\ &:= \mathcal{E}_3^I + \mathcal{E}_3^{II} \label{Def:E-term}. 
\end{align}
Before bounding the contributions of $\mathcal{E}_3^I$ and $\mathcal{E}_3^{II}$, we record the following elementary result. 

\begin{lemma}
Assume $|h_{\vec{P}}| \leq 1$ for all $\vec{P} \in  \mathbb{P}_{\mathfrak{b}}^0$. Then 

\begin{eqnarray}
\left| \left| \sum_{\vec{P} \in \mathbb{P}^0_{\mathfrak{b}}}h_{\vec{P}} \Phi^\infty_{P_2,0} \right| \right|_2 \lesssim \left( \sum_{\vec{P}\in  \mathbb{P}^0_{\mathfrak{b}}} |I_{\vec{P}}| \right)^{1/2}.
\end{eqnarray}

\end{lemma}

\begin{proof}
Begin by noting

\begin{eqnarray*}
\left| \left| \sum_{\vec{P} \in \mathbb{P}^0_\mathfrak{b}} h_{\vec{P}} \Phi^\infty_{P_2,0} \right| \right|_2^2 &=& \sum_{\vec{P}, \tilde{\vec{P}} \in \mathbb{P}^0_\mathfrak{b}} h_{\vec{P}} h_{\tilde{\vec{P}}}\left \langle \Phi^\infty_{P_2,0}, \Phi^\infty_{P_2,0} \right\rangle \\ &=& \left( \sum_{|I_{\vec{P}}| \gg  |I_{\tilde{\vec{P}}}| } +\sum_{|I_{\vec{P}}| \simeq |I_{\tilde{\vec{P}}}|} + \sum_{|I_{\vec{P}}| \ll  |I_{\tilde{\vec{P}}}|} \right)h_{\vec{P}} h_{\tilde{\vec{P}}}\left \langle \Phi^\infty_{\vec{P}_2,0}, \Phi^\infty_{\tilde{\vec{P}}_2,0} \right\rangle \\ &=& I + II + III. 
\end{eqnarray*}
It is straightforward to observe $|II| \lesssim \sum_{\vec{P} \in \mathbb{P}^0_\mathfrak{b}} |I_{\vec{P}}|$. By symmetry, it suffices to handle $I$:

\begin{eqnarray}\label{Est:Key}
|I | &\leq &  \sum_{\vec{P} \in \mathbb{P}^0_\mathfrak{b}}\left[  \sum_{ \ \tilde{\vec{P}} \in \mathbb{P}^0_\mathfrak{b} : \omega_{\tilde{P}_2} \supset \supset \omega_{P_2}}\left| \left \langle \Phi^\infty_{P_2,0}, \Phi^\infty_{\tilde{P}_2,0} \right\rangle \right|  \right] \lesssim \sum_{\vec{P} \in \mathbb{P}^\mathfrak{b}_\mathfrak{d}} |I_{\vec{P}}|. 
\end{eqnarray}
Indeed, recall that $\{P_2: \vec{P} \in \mathbb{P}^0\}$ is a disjoint collection of tiles, we have that for every $\vec{P} \in \mathbb{P}$ the collection $\left\{ I_{\tilde{\vec{P}}} : \tilde{\vec{P}} \in \mathbb{P}^0_{\mathfrak{b}}, \omega_{\tilde{P}_2} \supset \supset \omega_{P_2}\right\}$ consists of disjoint intervals and so for some $M \gg 1$

\begin{align*}
& \sum_{ \tilde{\vec{P}} \in \mathbb{P}^0_\mathfrak{b} : \omega_{\tilde{P}_2} \supset \supset \omega_{P_2}}\left| \left \langle \Phi^\infty_{P_2,0}, \Phi^\infty_{\tilde{P}_2,0} \right\rangle \right| \\ \lesssim &   \sum_{l \in \mathbb{Z}} \frac{1}{1+l^M}   \left \langle \chi^M_{I_{\vec{P}}} , \sum_{ \tilde{\vec{P}} \in \mathbb{P}^0_\mathfrak{b} : \omega_{\tilde{P}_2} \supset \supset \omega_{P_2}, I_{\vec{\tilde{P}}} \subset I_{\vec{P}} + l|I_{\vec{P}}|}  1_{I_{\vec{\tilde{P}}}} \right \rangle  \\ \lesssim&  \sum_{l \in \mathbb{Z}}\frac{1}{1+l^M} |I_{\vec{\tilde{P}}}| \\ \lesssim & |I_{\vec{P}}|. 
\end{align*}
Summing on $\vec{P} \in \mathbb{P}^0_{\mathfrak{b}}$ then yields estimate \eqref{Est:Key}.

\end{proof}
Using the fact that the trees $T \in \mathbb{T}$ form a strongly disjoint collection and that for any subtree $T^\prime \subset T \in \mathbb{T}$, $\sum_{\vec{Q} \in T^\prime} |c_{Q_3}|^2 \lesssim \frac{ |I_{T^\prime}|}{\sum_{T \in \mathbb{T}} |I_T| }$, we may use Lemma \ref{Est:BHTEnergy} to deduce from \eqref{Def:E-term}
\begin{eqnarray}\label{Est:E3-1}
\mathcal{E}^I_3=  \left| \left \langle \sum_{\vec{P} \in \mathbb{P}^0_{\mathfrak{b}}} h_{\vec{P}} \Phi^\infty_{P_2,0}, \sum_{T \in \mathbb{T}} \sum_{\vec{Q} \in T} c_{Q_3} \Phi^\alpha_{Q_3, 5} \right \rangle \right|  \lesssim \left| \left|  \sum_{\vec{P} \in \mathbb{P}^0_{\mathfrak{b}}}h_{\vec{P}} \Phi^\infty_{P_2,0} \right| \right|_2 \lesssim \left( \sum_{\vec{P} \in \mathbb{P}^0_{\mathfrak{b}}}  |I_{\vec{P}}| \right)^{1/2}.
\end{eqnarray}
It therefore remains to bound $\mathcal{E}_3^{II}$. Recalling \eqref{Def:E-term} and using estimate \eqref{Est:CQ}, we obtain for some $M \gg 1$

\begin{align*}
\mathcal{E}_3^{II} \lesssim&\sum_{\vec{P} \in \mathbb{P}^0_{\mathfrak{b}}} \left| \left \langle h_{\vec{P}} \Phi^\infty_{P_2,0},   \sum_{T \in \mathbb{T}} \sum_{\vec{Q} \in T: |\omega_{Q_3}| \gtrsim |\omega_{P_2}|} c_{Q_3} \Phi^\alpha_{Q_3, 5}\right \rangle \right| \\ \lesssim & \sum_{\vec{P} \in \mathbb{P}^0_{\mathfrak{b}}} \sum_{T \in \mathbb{T}} \sum_{\substack{ \vec{Q} \in T: |\omega_{Q_3}| \gtrsim |\omega_{P_2}| \\ \omega_{Q_3} \cap \omega_{P_2} \not = \emptyset}} |c_{Q_3}|   \left| \left \langle  \Phi^\infty_{P_2,0},    \Phi^\alpha_{Q_3, 5}\right \rangle \right|  \\ \lesssim& \frac{1}{\left( \sum_{T \in \mathbb{T}} |I_T| \right)^{1/2}}  \sum_{\vec{P} \in \mathbb{P}^0_{\mathfrak{b}}} \sum_{T \in \mathbb{T}} \sum_{\substack{ \vec{Q} \in T: |\omega_{Q_3}| \gtrsim |\omega_{P_2}| \\ \omega_{Q_3} \cap \omega_{P_2} \not = \emptyset}}   \left \langle  \chi^M_{I_{\vec{P}}},  \chi^M_{I_{\vec{Q}}} \right \rangle.  
\end{align*}
By estimate \eqref{Est:Trivial}, the last line of the above display can be majorized by a constant times
\begin{align*}
&  \frac{1}{\left( \sum_{T \in \mathbb{T}} |I_T| \right)^{1/2}} \sum_{\vec{P} \in \mathbb{P}^0_{\mathfrak{b}}}     \left \langle  \chi^M_{I_{\vec{P}}},  \sum_{T \in \mathbb{T}} \sum_{\substack{ \vec{Q} \in T: |\omega_{Q_3}| \gtrsim |\omega_{P_2}| \\ \omega_{Q_3} \cap \omega_{P_2} \not = \emptyset}} 1_{I_{\vec{Q}}} \right \rangle  \\ \lesssim &\frac{\sum_{\vec{P} \in \mathbb{P}^0_{\mathfrak{b}}} |I_{\vec{P}}|}{\left( \sum_{T \in \mathbb{T}} |I_T| \right)^{1/2}}\times \sup_{\vec{P} \in \mathbb{P}^0_{\mathfrak{b}}} \left| \left|  \sum_{T \in \mathbb{T}} ~\sum_{\substack{ \vec{Q} \in T: |\omega_{Q_3}| \gtrsim |\omega_{P_2}| \\  \omega_{Q_3} \cap \omega_{P_2} \not = \emptyset}}  1_{I_{\vec{Q}}} \right| \right|_{L^\infty(\mathbb{R})}. 
\end{align*}
Using the disjointness of the tiles $\{Q_3: \vec{Q} \in T~for~some~T \in \mathbb{T}\}$, 
\begin{eqnarray*}
\sup_{\vec{P} \in \mathbb{P}^0_{\mathfrak{b}}} \left| \left|  \sum_{T \in \mathbb{T}} \sum_{\vec{Q} \in T:|\omega_{Q_3}| \gtrsim |\omega_{P_2}|, \omega_{Q_3} \cap \omega_{P_2} \not = \emptyset}  1_{I_{\vec{Q}}} \right| \right|_{L^\infty(\mathbb{R})} \lesssim 1. 
\end{eqnarray*}
As a consequence, 

\begin{eqnarray}\label{Est:E3-2}
\mathcal{E}_3^{II} \lesssim \frac{ \sum_{\vec{P} \in \mathbb{P}^0_{\mathfrak{b}}} |I_{\vec{P}}|} { \left( \sum_{T \in \mathbb{T}} |I_T| \right)^{1/2}}. 
\end{eqnarray}
Combining the estimates for $\mathcal{E}_3^I$ and $\mathcal{E}_3^{II}$ in \eqref{Est:E3-1} and \eqref{Est:E3-2} yields
 \begin{eqnarray*}
 \mathcal{E}_3 \simeq 2^{-n_3} \left( \sum_{T \in \mathbb{T}} |I_T| \right)^{1/2} \simeq \mathcal{E}_3^I + \mathcal{E}_3^{II}  \lesssim \left( \sum_{\vec{P} \in \mathbb{P}^{\mathfrak{b}}_{\mathfrak{d}}} |I_{\vec{P}}| \right)^{1/2} + \frac{ \sum_{\vec{P} \in \mathbb{P}^0_{\mathfrak{b}}} |I_{\vec{P}}|} { \left( \sum_{T \in \mathbb{T}} |I_T| \right)^{1/2}}. 
 \end{eqnarray*}
 
 CASE 1:  $ \sum_{T \in \mathbb{T}} |I_T|  < \sum_{\vec{P} \in \mathbb{P}^0_{\mathfrak{b}}} |I_{\vec{P}}|$. Then $\sum_{T \in \mathbb{T}} |I_T|\lesssim 2^{n_3} \sum_{\vec{P} \in \mathbb{P}^0_{\mathfrak{b}}} |I_{\vec{P}}|$. 

CASE 2: $ \sum_{T \in \mathbb{T}} |I_T|  \geq \sum_{\vec{P} \in \mathbb{P}^0_{\mathfrak{b}}} |I_{\vec{P}}|$. Then $\sum_{T \in \mathbb{T}} |I_T|\lesssim 2^{2n_3} \sum_{\vec{P} \in \mathbb{P}^0_{\mathfrak{b}}} |I_{\vec{P}}|$. 
In either case 

\begin{eqnarray*}
\sum_{T \in \mathbb{T}}|I_T| \lesssim 2^{2n_3} \sum_{\vec{P} \in  \mathbb{P}^0_{\mathfrak{b}}} |I_{\vec{P}}|
\end{eqnarray*}
and $\mathcal{E}_3 \lesssim \left( \sum_{\vec{P} \in  \mathbb{P}^0_{\mathfrak{b}}} |I_{\vec{P}}| \right)^{1/2}$.
\end{proof}
\begin{prop}\label{P:Goal}
 \begin{eqnarray*}
\sum_{\vec{P} \in \mathbb{P}^0_\mathfrak{b}} |I_{\vec{P}}|  \lesssim 2^{\mathfrak{b}} |E_2|^{\frac{1}{2}(1-\tilde{\epsilon})} |E_3|^{ \frac{1}{2}(1-\tilde{\epsilon})} \left[ \sum_{\vec{P} \in \mathbb{P}^0_\mathfrak{b}} |I_{\vec{P}}| \right]^{\tilde{\epsilon}}.
 \end{eqnarray*}
 \end{prop}
 \begin{proof}
The fundamental size-energy inequality given as Proposition 6.12 in \cite{MR3052499} ensures that for any $0 \leq \theta _1, \theta_2, \theta_3 <1$ such that $\theta_1 + \theta_2 + \theta_3 =1$
\begin{align*}
& \left| \sum_{\vec{Q} \in \mathbb{Q}^{\tilde{d}}}~\frac{1}{|I_{\vec{Q}}|^{1/2}} \langle f_2, \Phi^\alpha_{Q_1,1} \rangle \langle f_3, \Phi^\alpha_{Q_2,2} \rangle \left \langle \Phi^\alpha_{Q_3,5}, \sum_{\vec{P} \in \mathbb{P}^0_{\mathfrak{b}}: \omega_{P_2} \supset  \supset \omega_{Q_3}}h_{\vec{P}} \Phi_{P_2, 0}^\infty   \right \rangle \right| \\\lesssim & \prod_{i=1}^3 \mathcal{E}_j\left( \{a_{\vec{Q},j}\}_{\vec{Q} \in \tilde{\mathbb{Q}}} \right)^{1-\theta_j} S_j\left(\{a_{\vec{Q},j}\}_{\vec{Q} \in \tilde{\mathbb{Q}}} \right) ^{\theta_j}. 
\end{align*}
It follows from Lemma \ref{P89} and \ref{P90} that $S_j\left(\{a_{\vec{Q},j}\}_{\vec{Q} \in \tilde{\mathbb{Q}}} \right) \lesssim 1$ for all $j \in \{1,2,3\}$. Moreover, Lemma \ref{Est:BHTEnergy} gives $\mathcal{E}_1 \left(\{a_{\vec{Q},1}\}_{\vec{Q} \in \tilde{\mathbb{Q}}} \right) \lesssim |E_2|^{1/2}, \mathcal{E}_2\left(\{a_{\vec{Q},2}\}_{\vec{Q} \in \tilde{\mathbb{Q}}}\right) \lesssim |E_3|^{1/2},$ and Lemma \ref{L:Energy-3} yields $\mathcal{E}_3 \left(\{a_{\vec{Q},3}\}_{\vec{Q} \in \tilde{\mathbb{Q}}} \right) \lesssim \left( \sum_{\vec{P} \in \mathbb{P}^0_\mathfrak{b}} |I_{\vec{P}}| \right)^{1/2}$. 
 Therefore, picking $\theta_1 , \theta_2 = \tilde{\epsilon} $ and $\theta_3 = 1-2\tilde{\epsilon}$ guarantees
 
 \begin{eqnarray*}
\sum_{\vec{P} \in \mathbb{P}^0_\mathfrak{b}(\tilde{\mathbb{Q}})} |I_{\vec{P}}|  \lesssim 2^{\mathfrak{b}} |E_2|^{\frac{1}{2}(1-\tilde{\epsilon})} |E_3|^{ \frac{1}{2}(1-\tilde{\epsilon})} \left[ \sum_{\vec{P} \in \mathbb{P}^0_\mathfrak{b}(\tilde{\mathbb{Q}})} |I_{\vec{P}}| \right]^{\tilde{\epsilon}}.
 \end{eqnarray*}
\end{proof}
An immediate consequence of Proposition \ref{P:Goal} is Theorem \ref{l1T}. Corollaries \ref{Cor0} and \ref{Cor1} follow.

\subsubsection{Synthesis}
We now combine the proceeding results to obtain generalized restricted type estimates uniform in small neighborhoods near each point in $\{A_1, A_2, A_3\}$. The decomposition

\begin{align*}
\mathbb{P} \times \mathbb{Q} = \bigcup_{d , \tilde{d} \geq 0} \bigcup_{n_1 \geq N_1(d)}  \bigcup_{\mathfrak{d} \geq N_2(d, \tilde{d})} \mathbb{P}^{d, \tilde{d},*}_{n_1, \mathfrak{d}} \times \mathbb{Q}^{\tilde{d}}
\end{align*}
enables us to rewrite $\Lambda_2^{\mathbb{P}, \mathbb{Q}}(f_1, f_2, f_3, f_4) $  as
\begin{align*}
  \sum_{d, \tilde{d} \geq 0}  \sum_{n_1 \geq N_1(d)} \sum_{\mathfrak{d} \geq N_2(d, \tilde{d})} \sum_{ \vec{P} \in  \mathbb{P}^{d, \tilde{d},*}_{n_1, \mathfrak{d}}}\frac{1}{|I_{\vec{P}}|^{1/2}} \langle f_1, \Phi^{\alpha^\prime}_{P_1,1} \rangle \langle f_4,   \Phi^{\alpha^\prime}_{P_4,4} \rangle\left \langle  \int_0^1 BHT^{\alpha, \mathbb{Q}^{\tilde{d}}}_{\omega_{P_2}}(f_2, f_3) d \alpha , \Phi^{\alpha^\prime}_{P_2,0} \right\rangle .
 \end{align*}
 For fixed $d, \tilde{d} \geq 0, n_1 \geq N_1(d)$ and $\mathfrak{d} \geq N_2(d, \tilde{d})$, we may further rewrite  
 \begin{align*}
 &\sum_{ \vec{P} \in  \mathbb{P}^{d, \tilde{d},*}_{n_1, \mathfrak{d}}}\frac{1}{|I_{\vec{P}}|^{1/2}} \langle f_1, \Phi_{P_1,1} \rangle \langle f_4,   \Phi_{P_4,4} \rangle\left \langle  \int_0^1 BHT^{\alpha, \mathbb{Q}^{\tilde{d}}}_{\omega_{P_2}}(f_2, f_3) d \alpha , \Phi_{P_2,0} \right\rangle \\  =& \sum_{T \in \mathcal{T}_{n_1, 1}^{d} } \sum_{ \vec{P} \in T \cap \mathbb{P}^{d, \tilde{d},*}_{n_1, \mathfrak{d}} } \frac{1}{|I_{\vec{P}}|^{1/2}} \langle f_1, \Phi_{P_1,1} \rangle \langle f_4,   \Phi_{P_4,4} \rangle\left \langle  \int_0^1 BHT^{\alpha, \mathbb{Q}^{\tilde{d}}}_{\omega_{P_2}}(f_2, f_3) d \alpha , \Phi_{P_2,0} \right\rangle \\ =&  \sum_{T \in \mathcal{T}_{\mathfrak{d}, 2}^{d, \tilde{d}}} \sum_{ \vec{P} \in T \cap \mathbb{P}^{d, \tilde{d},*}_{n_1, \mathfrak{d}} } \frac{1}{|I_{\vec{P}}|^{1/2}} \langle f_1, \Phi_{P_1,1} \rangle \langle f_4,   \Phi_{P_4,4} \rangle\left \langle  \int_0^1 BHT^{\alpha, \mathbb{Q}^{\tilde{d}}}_{\omega_{P_2}}(f_2, f_3) d \alpha , \Phi_{P_2,0} \right\rangle. 
\end{align*}
Each tree in $ \mathcal{T}_{n_1, 1}^{d}$ is overlapping in either the 1st or 2nd index.  Using the tree and energy estimates gives

\begin{align*}
& \left| \sum_{T \in \mathcal{T}_{n_1, 1}^{d} } \sum_{ \vec{P} \in T \cap \mathbb{P}^{d, \tilde{d},*}_{n_1, \mathfrak{d}} } \frac{1}{|I_{\vec{P}}|^{1/2}} \langle f_1, \Phi_{P_1,1} \rangle \langle f_4,   \Phi_{P_4,4} \rangle\left \langle  \int_0^1 BHT^{\alpha, \mathbb{Q}^{\tilde{d}}}_{\omega_{P_2}}(f_2, f_3) d \alpha , \Phi_{P_2,0} \right\rangle \right|\\& \lesssim 2^{-N d}  2^{-n_1} 2^{-\mathfrak{d}} \sum_{T \in \mathcal{T}_{n_1, 1}^{d} }  |I_T| \\ &\lesssim 2^{-N d} 2^{-n_1} 2^{-\mathfrak{d}} \left[ 2^{2n_1} |E_1| \right]. 
\end{align*}
Similarly, we have from the tree estimate in section 4.2.2, Lemma \ref{Est:l2Energy}, and Corollary \ref{Cor1}  
\begin{align*}
& \left| \sum_{T \in \mathcal{T}_{\mathfrak{d}, 2}^{d, \tilde{d}}} \sum_{ \vec{P} \in T \cap \mathbb{P}^{d, \tilde{d},*}_{n_1, \mathfrak{d}} } \frac{1}{|I_{\vec{P}}|^{1/2}} \langle f_1, \Phi_{P_1,1} \rangle \langle f_4,   \Phi_{P_4,4} \rangle\left \langle  \int_0^1 BHT^{\alpha, \mathbb{Q}^{\tilde{d}}}_{\omega_{P_2}}(f_2, f_3) d \alpha , \Phi_{P_2,0} \right\rangle \right| \\&\lesssim 2^{-N d} 2^{-n_1} 2^{-\mathfrak{d}} \left[ \sum_{T \in \mathcal{T}_{\mathfrak{d}, 2}^{d, \tilde{d}}}  |I_T| \right] \\ &\lesssim_\theta  2^{-N d} 2^{-n_1} 2^{-\mathfrak{d}} \min\left\{ 2^{2 \mathfrak{d}} 2^{2 \tilde{d}} |E_2|^{1+\gamma} |E_3|^{2- \gamma}, 2^{\frac{\mathfrak{d}}{1-\tilde{\epsilon}}} |E_2|^{\frac{1}{2}} |E_3|^{\frac{1}{2}} \right\} . 
\end{align*}
Hence, for any $0\leq \theta_1, \theta_2, \theta_3 \leq 1$ satisfying $\theta_1 + \theta_2+\theta_3=1$ and $0 < \gamma, \tilde{\epsilon} <1$ one has

\begin{align*}
&\left| \Lambda_2^{\mathbb{P}, \mathbb{Q}}(f_1, f_2, f_3, f_4 1_{\Omega^\prime}) \right|\\ \lesssim &\sum_{\tilde{d} , d \geq 0} \sum_{n_1 \geq N_1(d)} \sum_{\mathfrak{d} \geq N_2(d, \tilde{d})} 2^{-N d} 2^{-n_1} 2^{-\mathfrak{d}} \\ & \hspace{20mm} \times \min 
 \left \{  2^{2n_1} |E_1| , 2^{2 \mathfrak{d}} 2^{2 \tilde{d}} |E_2|^{1+\gamma} |E_3|^{2- \gamma}, 2^{\frac{\mathfrak{d}}{1-\tilde{\epsilon}}} |E_2|^{\frac{1}{2}} |E_3|^{\frac{1}{2}} \right\} \\ \lesssim &  \sum_{\tilde{d} , d \geq 0} \sum_{n_1 \geq N_1(d)}  \sum_{\mathfrak{d} \geq N_2(d, \tilde{d})} 2^{-N d}2^{2 \theta_2 \tilde{d}} 2^{-n_1(1- 2 \theta_1)} 2^{-\mathfrak{d}\left[1- 2 \theta_2 - \frac{\theta_3}{1-\tilde{\epsilon}} \right]} \\ &\hspace{20mm} \times    |E_1|^{\theta_1} |E_2|^{\theta_2(1+\gamma) + \frac{\theta_3}{2} } |E_3|^{ \theta_2(2- \gamma)+\frac{\theta_3}{2}} . 
\end{align*}
To produce generalized restricted type estimates near $A_1 = \left(1, \frac{1}{2}, \frac{1}{2}, - 1 \right)$, it suffices to show that for each $\alpha_1<1$ in some small, fixed neighborhood of  $1$, there is a generalized restricted type estimate $(\alpha_1, \frac{1}{2}, \frac{1}{2} , - \alpha_1)$  for $\Lambda_2^{\mathbb{P}, \mathbb{Q}} (f_1, f_2, f_3, f_41_{\Omega^c})$. To this end, let $\theta_1= \epsilon_1, \theta_2 =  \epsilon_2, \theta_3 = 1-\epsilon_1 -\epsilon_2$. We need to choose $0 < \epsilon_1, \epsilon_2 \ll 1$ and $0 < \tilde{\epsilon}, \gamma <1$ such that

\begin{align}
1-\alpha_1 &>\epsilon_1 \label{Ineq:1st} \\ 
1-2 \epsilon_2 - \frac{1-\epsilon_1 - \epsilon_2}{1-\tilde{\epsilon}} &>0 \label{Ineq:2nd} \\
\frac{1}{2} \left[ 1- 2 \epsilon_2 - \frac{1-\epsilon_1 - \epsilon_2}{1-\tilde{\epsilon}} \right] + \epsilon_2 (1+\gamma) + \frac{1-\epsilon_1 - \epsilon_2}{2} &= \frac{1}{2} \label{Eq:1st} \\ 
\frac{1}{2} \left[ 1- 2\epsilon_2 - \frac{1-\epsilon_1 - \epsilon_2}{1-\tilde{\epsilon}} \right] +  \epsilon_2 (2-\gamma) + \frac{1-\epsilon_1 - \epsilon_2}{2} &=\frac{1}{2}  \label{Eq:2nd}
\end{align}
Take $\epsilon_1 =\frac{1}{2} (1-\alpha_1),  \epsilon_2 = \frac{1}{8}(1-\alpha_1)$. This ensures that \eqref{Ineq:1st} is satisfied. Adding \eqref{Eq:1st} and \eqref{Eq:2nd} gives

\begin{align}\label{Eq:3rd}
\left[ 1- 2 \epsilon_2 - \frac{1 - \epsilon_1 - \epsilon_2}{1-\tilde{\epsilon}} \right]  + 3 \epsilon_2 +\left[ 1- \epsilon_1 - \epsilon_2  \right]= 1. 
\end{align}
Note that $3 \epsilon_2 +1 - \epsilon_1- \epsilon_2 < 1$. 
Moreover, letting $\tilde{\epsilon} \rightarrow 0$ in \eqref{Eq:3rd} observe

\begin{eqnarray*}
1-2 \epsilon_2 -1 + \epsilon_1 + \epsilon_2 + 3 \epsilon_2 + 1 - \epsilon_1 - \epsilon_2 = 1+ \epsilon_2 > 1.
\end{eqnarray*}
Therefore, there is a solution $\tilde{\epsilon} \in (0,1)$ solving $1- 2 \epsilon_2 - \frac{1 - \epsilon_1 - \epsilon_2}{1-\tilde{\epsilon}} + 3 \epsilon_2 + 1- \epsilon_1 - \epsilon_2 =1$ and satisfying $1-2 \epsilon_2 - \frac{1-\epsilon_1 - \epsilon_2}{1-\tilde{\epsilon}}>0$. So, it suffices to find $0 < \gamma <1$. Note letting $\gamma \rightarrow 0$  in \eqref{Eq:2nd}  gives
\begin{eqnarray*}
\frac{1}{2} \left[ 1- 2\epsilon_2 - \frac{1-\epsilon_1 - \epsilon_2}{1-\tilde{\epsilon}} \right] + 2  \epsilon_2 + \frac{1-\epsilon_1 - \epsilon_2}{2} = \frac{1}{2}   + \frac{ \epsilon_2}{2} > 1/2. 
\end{eqnarray*}
Moreover, letting $\gamma \rightarrow 1$ in \eqref{Eq:2nd} gives

\begin{eqnarray*}
\frac{1}{2} \left[ 1- 2 \epsilon_2 - \frac{1-\epsilon_1 - \epsilon_2}{1-\tilde{\epsilon}} \right] + \epsilon_2 + \frac{1-\epsilon_1 - \epsilon_2}{2}  = \frac{1}{2}  - \frac{\epsilon_2}{2} <1/2
\end{eqnarray*}
Therefore, we may find the desired $0 < \gamma<1$.

For $A_2= (\frac{1}{2}, \frac{1}{2}, 1, -1)$, we need to show that for each $\alpha_3<1$ in some small, fixed neighborhood of  $1$ the restricted type estimate $\left( \frac{1}{2}, \frac{1}{2}, \alpha_3 , - \alpha_3 \right)$ holds. To this end, let $\theta_1= 1/2 - \epsilon_1, \theta_2 = 1/2 - \epsilon_2, \theta_3 = \epsilon_1 + \epsilon_2$. We need to choose $0 < \epsilon_1, \epsilon_2 \ll 1$ and $0< \gamma, \tilde{\epsilon}<1$  such that

\begin{align}
2 \epsilon_2 >& \frac{\epsilon_1 + \epsilon_2}{1-\tilde{\epsilon}} \label{Ineq:d}  \\
\frac{1}{2} \left[ 1- 2\left(\frac{1}{2} - \epsilon_2\right) - \frac{\epsilon_1 + \epsilon_2}{1-\tilde{\epsilon}} \right] + \left(\frac{1}{2} - \epsilon_2\right) (2-\gamma) + \frac{\epsilon_1 + \epsilon_2}{2} =& \alpha_3 \label{Eq:4th} \\
\frac{1}{2} \left[ 1- 2\left(\frac{1}{2} - \epsilon_2\right) - \frac{\epsilon_1 + \epsilon_2}{1-\tilde{\epsilon}} \right] + \left(\frac{1}{2}- \epsilon_2\right) (1+\gamma) + \frac{\epsilon_1 + \epsilon_2}{2} =& \frac{1}{2}. \label{Eq:5th} 
\end{align}
Take $\epsilon_2=\frac{3}{4} \cdot (1 - \alpha_3)$ and $\epsilon_1 = \frac{1}{4} \cdot ( 1 - \alpha_3)$. 
Any $0< \tilde{\epsilon} <1/3$ will satisfy \eqref{Ineq:d}. 
Moreover, from adding \eqref{Eq:4th} and \eqref{Eq:5th} it follows that
\begin{eqnarray*}
- \frac{\epsilon_1 + \epsilon_2}{1-\tilde{\epsilon}} + 3/2  + \epsilon_1= \alpha_3 + \frac{1}{2}. 
\end{eqnarray*}
Therefore, we take $\tilde{\epsilon}=1/5$. 
Lastly, we need to check that there is a solution $0 < \gamma<1$ to \eqref{Eq:4th}. 
Note by our choice of $\tilde{\epsilon}$ that \eqref{Eq:4th} can be rewritten as
\begin{align*}
\frac{1}{2} \left[ 2 \epsilon_2 + \alpha_3 - 1 - \epsilon_1 \right] + \left(\frac{1}{2} - \epsilon_2\right)(2-\gamma) + \frac{ 1- \alpha_3}{2}=\alpha_3.
\end{align*}
So, it suffices to show by letting $\gamma=0$ and $\gamma =1$ respectively that

\begin{align*}
\frac{1}{2} \left[ 2 \epsilon_2 + \alpha_3 - 1 - \epsilon_1 \right] +1-2 \epsilon_2  + \frac{ 1 - \alpha_3}{2} &>\alpha_3\\ 
 \frac{1}{2} \left[ 2 \epsilon_2 + \alpha_3 - 1 - \epsilon_1 \right] +\frac{1}{2}- \epsilon_2  + \frac{ 1 - \alpha_3}{2}  &< \alpha_3. 
\end{align*}
This is the same as requiring

\begin{eqnarray*}
\alpha_3 <1~\text{and}~ \frac{1}{2} - \frac{1}{8} \left[ 1- \alpha_3\right] < \alpha_3.
\end{eqnarray*}
Both inequalities are satisfied for all $\frac{3}{7} < \alpha_3 <1$.  Estimates near $A_3= (\frac{1}{2}, 1, \frac{1}{2}, -1)$ follow from a nearly identical argument, so the details are omitted. 
\subsection{Generalized Restricted Type Estimates near $A_4$ and $A_5$}
Recall $A_4=\left (-\frac{3}{2}, \frac{1}{2}, 1, 1 \right)$ and $A_5 = \left( - \frac{3}{2}, 1, \frac{1}{2}, 1 \right)$.  By rescaling, we may assume $|E_1|=1$. Set

\begin{eqnarray*}
\tilde{\Omega} = \left\{ M1_{E_2} \gtrsim |E_2| \right\} \bigcup \left\{ M1_{E_3} \gtrsim |E_3| \right\}   \bigcup \left\{ M1_{E_4} \gtrsim |E_4| \right\}.
\end{eqnarray*}
Let $\mathbb{Q}^{\tilde{d}} := \left\{ \vec{Q} \in \mathbb{Q} : 1 + \frac{ dist(I_{\vec{Q}}, \tilde{\Omega}^c) }{|I_{\vec{Q}}|} \simeq 2^{\tilde{d}} \right\}$ and define as before

\begin{align*}
\Omega_1^{0} &=\left\{ M \left[ \int_0^1  BHT^{\alpha, \mathbb{Q}^0}(f_2, f_3) d\alpha\right] \gtrsim  |E_2|^{1/2} |E_3|^{1/2} \right\} \\ \Omega_2^{\tilde{d}} &=  \left\{ M \left(  \left[ \int_0^1  \sum_{\vec{Q} \in \mathbb{Q}^{\tilde{d}}} \frac{ |\langle f_2, \Phi^\alpha_{Q_1, 2} \rangle \langle f_3, \Phi^\alpha_{Q_2, 3} \rangle|}{|I_{\vec{Q}}|} \tilde{1}_{I_{\vec{Q}}} d \alpha\right]^2 \right) \gtrsim 2^{2\tilde{d}} |E_2| |E_3| \right\}.
\end{align*}
Lastly, construct
\begin{eqnarray*}
\Omega= \tilde{\Omega} \bigcup \Omega_1^0 \bigcup_{\tilde{d} \geq 1} \Omega_2^{\tilde{d}}. 
\end{eqnarray*}
Then for large enough implicit constants, $|\Omega| \leq 1/2$ and $\tilde{E}_1 := E_1 \cap \Omega^c$ is a major subset of $E_1$ since $|E_1|=1$. The rest of the proof of Theorem \ref{DT} near $\{A_4, A_5\}$ proceeds exactly as before. For any $0\leq \theta_1, \theta_2, \theta_3 \leq 1$ satisfying $\theta_1 + \theta_2+\theta_3=1$ and $0 < \gamma, \tilde{\epsilon} <1$ one has 

\begin{eqnarray*}
&& \left| \Lambda_2^{\mathbb{P}, \mathbb{Q}}(f_1 1_{\Omega^\prime}, f_2, f_3, f_4)  \right| \\ &\lesssim_{\tilde{\epsilon}}&   \sum_{\tilde{d} , d \geq 0} \sum_{n_1 \geq N_1(d)}  \sum_{\mathfrak{d} \geq N_2(d, \tilde{d})}2^{-\tilde{N} d}2^{2 \theta_2 \tilde{d}} 2^{-n_1(1- 2 \theta_1)}  2^{-\mathfrak{d}\left[1- 2 \theta_2 - \frac{\theta_3}{1-\tilde{\epsilon}} \right]}  |E_2|^{\theta_2(1+\gamma) + \frac{\theta_3}{2} } |E_3|^{ \theta_2(2- \gamma)+\frac{\theta_3}{2}} |E_4|. 
\end{eqnarray*}
For estimates near $A_4= (-\frac{3}{2}, \frac{1}{2}, 1, 1)$, let $\alpha_2 = 1/2, |1-\alpha_3| \ll 1$, and $\alpha_4 =1$. Set $\theta_1 =\frac{1}{2}-\epsilon_1, \theta_2= \frac{1}{2} -\epsilon_2, $ and $\theta_3 = \epsilon_1 + \epsilon_2$, where we are required to choose $0 < \epsilon_1, \epsilon_2 \ll 1$ and $0 < \gamma , \tilde{\epsilon} <1$ to ensure

\begin{align*}
2 \epsilon_2 >& \frac{ \epsilon_1 + \epsilon_2}{1-\tilde{\epsilon}}   \\
\frac{1}{2} \left[ 1-2 \left(\frac{1}{2} - \epsilon_2\right) - \frac{ \epsilon_1 + \epsilon_2}{1-\tilde{\epsilon}} \right] + \left(\frac{1}{2} - \epsilon_2 \right)( 2- \gamma) + \frac{\epsilon_1 + \epsilon_2}{2} =& \alpha_3 \\ 
\frac{1}{2} \left[ 1-2 \left(\frac{1}{2} - \epsilon_2\right) - \frac{ \epsilon_1 + \epsilon_2}{1-\tilde{\epsilon}} \right] + \left(\frac{1}{2} - \epsilon_2\right )( 1+ \gamma) + \frac{\epsilon_1 + \epsilon_2}{2} =& 1/2 . 
\end{align*}
This system was solved in our analysis of $A_2$, so the details are omitted here.  For estimates near $A_5 = \left( - \frac{3}{2}, 1, \frac{1}{2}, 1 \right)$, set $\theta \simeq 1,  \theta _2 \simeq \frac{1}{2}, \theta_3 \simeq \frac{1}{2}, \theta _3 \simeq 0$ and proceed as in the argument for $A_4=(-\frac{3}{2}, \frac{1}{2}, 1,1)$.

\subsection{Generalized Restricted Type Estimates near $A_6, A_7, A_8, A_9$}
Recall that 

\begin{eqnarray*}
C^{1,1,-2} : (f_1, f_2, f_3) \mapsto \int_{\xi_1 < \xi_2 < -\xi_3/2} \hat{f}_1(\xi_1) \hat{f}_2(\xi_2) \hat{f}_3(\xi_3) e^{2 \pi i x (\xi_1 + \xi_2 + \xi_3)} d\xi_1 d \xi_2 d \xi_3
\end{eqnarray*}
satisfies the identity \eqref{Est:Iden}, namely 

\begin{align*}
 C^{1,1,-2}(f_1, f_2, f_3)(x)  =& C^{1,1}(f_1, f_2)(x) \cdot  f_3(x) \\-& H^+(C^{1,1}(f_1 ,f_2)\cdot  f_3)(x) \\ -&H^-(f_1 \cdot C^{-2,1}(f_3, f_2))(x).
\end{align*}
Therefore, the $3$-adjoint denoted by $C^{1,1,-2}_{*, 3}$ defined by the usual property 

\begin{eqnarray*}
\int_\mathbb{R}  C^{1,1,-2} (f_1, f_2, f_4)(x) f_3(x) dx= \int_\mathbb{R} C^{1,1,-2}_{*, 3} (f_1, f_2, f_3) (x) f_4(x) dx\end{eqnarray*}
on $\mathcal{S}^4(\mathbb{R})$ is writable as

\begin{align*}
C^{1,1,-2}_{*,3}(f_1, f_2, f_3)(x)=&C^{1,1}(f_1, f_2)(x) \cdot f_3(x) \\-&  C^{1,1} (f_1, f_2) (x) \cdot H^{-} (f_3) (x)\\-&  C^{1,1}( f_1 \cdot H^+(f_3), f_2)(x).
\end{align*}
Indeed, it suffices to check the last term, which we claim is the $3$-adjoint of the map $(f_1, f_2, f_3) \mapsto - H^{-}(f_1 \cdot C^{-2,1}(f_3, f_2))(x)$. Indeed, first observe 

\begin{align*}
 &\int_\mathbb{R} H^{-}\left[ f_1 \cdot C^{-2,1}(f_4, f_2)\right](x)  f_3(x) dx \\ =& \int_\mathbb{R} f_1 (x) C^{-2,1}(f_4, f_2)(x) H^+ [f_3](x) dx \\ =& \int_\mathbb{R} C^{-2,1}(f_4, f_2)(x)   f_1(x) H^+ [f_3] (x) dx .  
 \end{align*}
Next, note that for any $F,G,H \in \mathcal{S}(\mathbb{R})$

\begin{align*}
\int_\mathbb{R} C^{-2, 1} (F,G)(x) H(x) dx =&   \int_{\xi_1 + \xi_2 + \xi_3 = 0} 1_{\left\{-\frac{\xi_1}{2} < \xi_2 \right \}} \hat{F}(\xi_1) \hat{G}(\xi_2) \hat{H}(\xi_3) d \vec{\xi} \\ =& \int_{\xi_1 + \xi_2 + \xi_3 = 0} 1_{\left\{\frac{\xi_2 + \xi_3}{2} < \xi_2 \right\}} \hat{F}(\xi_1) \hat{G}(\xi_2) \hat{H}(\xi_3) d \vec{\xi} \\ =& \int_{\xi_1 + \xi_2 + \xi_3 = 0} 1_{\left\{ \xi_3 < \xi_2 \right \}} \hat{F}(\xi_1) \hat{G}(\xi_2) \hat{H}(\xi_3) d \vec{\xi} \\ =& \int_\mathbb{R} C^{1,1} (H,G)(x) F(x) dx. 
\end{align*}
Setting $F=f_4, G= f_2$ and $H=f_1 \cdot H^+[f_3]$  then yields

\begin{align*}
\int_\mathbb{R}  C^{-2,1}(f_4, f_2)(x)   f_1(x) H^+ [f_3](x) dx = \int_\mathbb{R} C^{1,1}(f_1 H^+[f_3], f_2)(x)f_4(x) dx . 
 \end{align*}
Using the $BHT$ and Hilbert transform estimates, we may observe that $C^{1,1,-2}_{*, 3}$ maps into $L^r(\mathbb{R})$ for all $r \in (\frac{2}{3}, \infty)$. Therefore, we should not expect the generic adjoint models to map below $L^{\frac{2}{3}}(\mathbb{R})$. We now proceed to prove the generalized restricted type estimates near the points $A_6, A_7, A_8, A_9$, where the adjoint index is restricted to map into the above range. By symmetry, it will suffice to prove the estimate only near the points $A_8 =\left( 0, 1, -\frac{1}{2}, \frac{1}{2} \right)$ and $A_9 = \left( \frac{1}{2}, 1 - \frac{1}{2}, 0 \right)$, for which $3$ is the adjoint index.  Indeed, estimates near $A_6, A_7$ are obtained from estimates near $A_8$ and $A_9$ by interchanging the roles of $f_2, f_3$. 

The adjoint situation is more complicated in the semi-degenerate case than in the fully non-degenerate one because one cannot simply flip the frequency inclusions to reduce to the situation where the exceptional set is associated with functions in the second and third index. This is because the composition

\begin{eqnarray*}
\sum_{\vec{P} \in \mathbb{P}} \frac{ \langle f_1, \Phi_{P_1,1} \rangle \langle f_4, \Phi_{P_4,4} \rangle \left \langle \int_0^1 BHT^{\alpha, \mathbb{Q}} (f_2, f_3)d \alpha , \Phi_{P_2,0}\right \rangle}{|I_{\vec{P}}|^{1/2}}
\end{eqnarray*}
satisfies no generalized restricted type estimates as $\mathbb{P}$ is not a rank-1 collection of tri-tiles. Moreover, if one tries to repeat the arguments for $A_1, A_2, A_3, A_4, A_5$ estimates,  one cannot enlarge the exceptional set $\Omega$ to obtain good control over the averages of the $BHT^{\mathbb{Q}}$-type operators on time-intervals of $\mathbb{P}$-tiles that may be much farther from $\Omega^c$ than the time-intervals of the $\mathbb{Q}$-tiles.  The way around this obstruction is to decompose our collection of degenerate tri-tiles $\mathbb{P}$. 

By scaling invariance, set $|E_3|=1$ and use $f_j$ to denote $f_j^\prime$ for $j=1,2,3,4$ as described in Definition \ref{Def:RestrType}. Let $\Omega = \left\{ M1_{E_2} \gtrsim |E_2| \right\}$ for a large enough implicit constant so that $\widetilde{E_3} := E_3 \cap \Omega^c$ satisfies $|\widetilde{E_3}| \geq |E_3|/2$. As usual, set

\begin{eqnarray*}
\mathbb{Q}^d  = \left\{ \vec{Q} \in \mathbb{Q}: 1+ \frac{ dist( I_{\vec{Q}}, \Omega^c)}{ |I_{\vec{Q}}|} \simeq 2^d \right\}.
\end{eqnarray*}
For each $d \geq 0$, we decompose $\mathbb{P} = \bigcup_{k \geq 0} \mathbb{P}^d_k$ by setting

\begin{eqnarray*}
\mathbb{P}^d_0 = \left\{ \vec{P} \in \mathbb{P} :   I_{\vec{P}} \cap \left\{ M\left[  \sum_{\vec{Q} \in \mathbb{Q}^d}  \frac{ \langle f_2, \Phi_{Q_1,2} \rangle \langle f_3, \Phi_{Q_2,3} \rangle \Phi_{Q_3,5}}{|I_{\vec{Q}}|^{1/2}} \right] \geq 1 \right\}^c \not = \emptyset \right\}.
\end{eqnarray*}
and then inductively constructing for all $k \geq 1$

\begin{eqnarray*}
\mathbb{P}^d_k = \left\{ \vec{P} \in \mathbb{P} \cap \left( \bigcup_{\tilde{k} \leq k-1} \mathbb{P}^d_{\tilde{k}} \right)^c~ :~   I_{\vec{P}} \cap \left\{ M\left[  \sum_{\vec{Q} \in \mathbb{Q}^d} \frac{ \langle f_2, \Phi_{Q_1,2} \rangle \langle f_3, \Phi_{Q_2,3} \rangle \Phi_{Q_3,5}}{|I_{\vec{Q}}|^{1/2}}\right]  \geq 2^k \right\}^c \not = \emptyset \right\}.
\end{eqnarray*}
We assign $\mathbb{I}_k^d$ to be the collection of maximal dyadic intervals in the set $\{ I_{\vec{P}} : \vec{P} \in \mathbb{P}_k^d \}$ and record a straightforward yet essential size estimate. 
\begin{lemma}\label{L:Size-2}
Recall that for any $\tilde{\mathbb{P}} \subset \mathbb{P}$

\begin{align*}
 & Size_0^{d} (f_2,f_3, \tilde{\mathbb{P}}) \\ :=&   \sup_{T \subset \tilde{\mathbb{P}}} \frac{1}{|I_T|^{1/2}} \left( \sum_{\vec{P} \in T}\left | \left \langle  BHT^{\mathbb{Q}^d} (f_2, f_3) , \Phi_{P_2, 0} \right \rangle \right|^2 \right. \\  +& \left. \left | \left \langle \int_0^1 \sum_{\vec{Q} \in \mathbb{Q}^d: \omega_{Q_3} \supset \supset \omega_{P_2}} \frac{1}{|I_{\vec{Q}}|^{1/2}} \langle f_2, \Phi^\alpha_{Q_1, 2} \rangle \langle f_3, \Phi^\alpha_{Q_2,3} \rangle \Phi^\alpha_{Q_3,5} d\alpha, \Phi_{P_2,0} \right \rangle  \right|^2 \right. \\ +& \left.\left | \left \langle \int_0^1 \sum_{\vec{Q} \in \mathbb{Q}^d: |\omega_{Q_3}| \simeq  |\omega_{P_2}|} \frac{1}{|I_{\vec{Q}}|^{1/2}} \langle f_2, \Phi^\alpha_{Q_1, 2} \rangle \langle f_3, \Phi^\alpha_{Q_2,3} \rangle \Phi^\alpha_{Q_3,5} d\alpha, \Phi_{P_2,0} \right \rangle  \right|^2  \right)^{1/2}
\end{align*}
where the supremum is over all $1$-trees $T \subset \tilde{\mathbb{P}}$. 
Then for all $k, d \geq 0$

\begin{eqnarray*}
Size_0^d(f_2, f_3, \mathbb{P}^d_k) \lesssim 2^k. 
\end{eqnarray*}
\end{lemma}
\begin{proof}
By triangle inequality, it suffices to bound the sum
\begin{align*}
& \left[  \frac{1}{|I_T|} \sum_{\vec{P} \in T \cap \mathbb{P}^d_k } \left| \left \langle BHT^{\mathbb{Q}^d}(f_2, f_3) ,\Phi_{P_2,0} \right \rangle \right|^2 \right]  ^{1/2} \\+& \left[  \frac{1}{|I_T|} \sum_{\vec{P} \in T \cap \mathbb{P}^d_k } \left| \left \langle \sum_{\vec{Q} \in \mathbb{Q}^d: \omega_{Q_3} \supset \supset \omega_{P_2}}\frac{ \langle f_2, \Phi_{Q_1,2} \rangle \langle f_3, \Phi_{Q_2,3} \rangle \Phi_{Q_3,5}}{|I_{\vec{Q}}|^{1/2}}, \Phi_{P_2,0} \right \rangle \right|^2 \right]^{1/2} \\ +& \left[ \frac{1}{|I_T|} \sum_{\vec{P} \in T \cap \mathbb{P}^d_k} \left | \left \langle \int_0^1 \sum_{\vec{Q} \in \mathbb{Q}^d: |\omega_{Q_3}| \simeq  |\omega_{P_2}|} \frac{1}{|I_{\vec{Q}}|^{1/2}} \langle f_2, \Phi^\alpha_{Q_1, 2} \rangle \langle f_3, \Phi^\alpha_{Q_2,3} \rangle \Phi^\alpha_{Q_3,5} d\alpha, \Phi_{P_2,0} \right \rangle  \right|^2  \right]^{1/2} \\ =:& ~I + II+III. 
\end{align*} 
For term I, use Lemma \ref{L:John-Nirenberg} to observe

\begin{eqnarray*}
I \lesssim  \sup_{\vec{P} \in T \cap \mathbb{P}_k} \frac{1}{|I_{\vec{P}}|} \int_\mathbb{R} \left| \sum_{\vec{Q} \in \mathbb{Q}^d } \frac{1}{|I_{\vec{Q}}|^{1/2}} \langle f_2, \Phi_{Q_1,2} \rangle \langle f_3, \Phi_{Q_2,3} \rangle \Phi_{Q_3,5} \right| \chi^M_{I_{\vec{P}}} dx.
\end{eqnarray*}
Because $\vec{P} \in T\cap \mathbb{P}_k$, $I_{\vec{P}} \cap \left\{ M\left[  \sum_{\vec{Q} \in \mathbb{Q}^d} \frac{1}{|I_{\vec{Q}}|^{1/2}} \langle f_2, \Phi_{Q_1,2} \rangle \langle f_3, \Phi_{Q_2,3} \rangle \Phi_{Q_3,5}\right] \geq  2^k \right\}^c \not = \emptyset$ and 

\begin{eqnarray*}
\frac{1}{|I_{\vec{P}}|} \int_\mathbb{R} \left| \sum_{\vec{Q} \in \mathbb{Q}^d } \frac{ \langle f_2, \Phi_{Q_1,2} \rangle \langle f_3, \Phi_{Q_2,3} \rangle \Phi_{Q_3,5}}{|I_{\vec{Q}}|^{1/2}} \right| \chi^M_{I_{\vec{P}}} dx \lesssim 2^k .
\end{eqnarray*}
The Biest size estimate in Lemma \ref{L:BiestSize} handles term II. Indeed, 
\begin{align*}
&\left[  \frac{1}{|I_T|} \sum_{\vec{P} \in T \cap \mathbb{P}^d_k } \left| \left \langle \sum_{\vec{Q} \in \mathbb{Q}^d: \omega_{Q_3} \supset \supset  \omega_{P_2}} \frac{ \langle f_2, \Phi_{Q_1,2} \rangle \langle f_3, \Phi_{Q_2,3} \rangle \Phi_{Q_3,5}}{|I_{\vec{Q}}|^{1/2}}, \Phi_{P_2,0} \right \rangle \right|^2 \right]^{1/2} \\ \lesssim & \left[ \sup_{\vec{P} \in T} \frac{1}{|I_{\vec{P}}|} \int_{E_2} \chi^M_{I_{\vec{P}}} dx\right]^\theta \left[ \sup_{\vec{P} \in T} \frac{1}{|I_{\vec{P}}|} \int_{E_3} \chi^M_{I_{\vec{P}}} dx \right] ^{1-\theta}  \\ \lesssim &~ 1.
\end{align*}

It therefore remains to bound term $III$: 
\begin{align}\label{Est:l2-trivial}
III^2 &\lesssim  \frac{1}{|I_T|} \sum_{l \in \mathbb{Z}} \frac{1}{1+l^M} \sum_{\vec{Q}\in \mathbb{Q}^{\tilde{d}}: I_{\vec{Q}} \subset I_T + l |I_T|}  \frac{ |\langle f_2, \Phi_{Q_1,2} \rangle|^2 |\langle f_3, \Phi_{Q_2, 3} \rangle|^2}{|I_{\vec{Q}}|} \\&\lesssim  \sum_{l \in \mathbb{Z}} \frac{1}{1+l^M} \frac{  ||1_{E_2} \chi^M_{I_{T}+l |I_T|} ||_{L^2} \cdot ||1_{E_3} \chi^M_{I_{T}+l |I_T|}||_{L^2}}{|I_T|} \nonumber \\ &\lesssim  1 \nonumber. 
\end{align}

\end{proof}
We must therefore contend with exponential growth in the tree sizes in our tile collections $\mathbb{P}^d_k$. What makes this growth acceptable is the simple observation that for every $ d, k \geq 0$
\begin{eqnarray*}
\bigcup_{\vec{P} \in \mathbb{P}^d_k} I_{\vec{P}} = \bigcup_{I \in \mathbb{I}^d_k}\subset \left\{  M\left[  \sum_{\vec{Q} \in \mathbb{Q}^d}\frac{ \langle f_2, \Phi_{Q_1,2} \rangle \langle f_3, \Phi_{Q_2,3} \rangle \Phi_{Q_3,5}}{|I_{\vec{Q}}|^{1/2}}\right] \gtrsim  2^k  \right\}.
\end{eqnarray*}
Therefore, using $|E_3|=1$, the $L^1 \rightarrow L^{1, \infty}$ bounds for the Hardy-Littlewood maximal function, and the localized $BHT$ estimate in Lemma \ref{Est:BHTloc}, we note

\begin{align*}
\left| \bigcup_{\vec{P} \in \mathbb{P}^d_k} I_{\vec{P}}\right|  = \sum_{I \in \mathbb{I}^d_k} |I| \lesssim  2^{-N d} 2^{-k}   |E_2|^{1/2}. 
\end{align*}
Not surprisingly, it is the smallness of the support of the intervals in $\mathbb{P}^d_k$ that will allow us to control the largeness of the tree sizes. We now need to recall another standard tile decomposition: 

\begin{lemma}
Let $Size_4(f_4, \mathbb{P}_{n_4,4}) := \sup_{\vec{P} \in \mathbb{P}_{n_4, 4}}\frac{1}{|I_{\vec{P}}|} \int_{E_4} \chi^M_{I_{\vec{P}}} dx$. There there is a decomposition $\mathbb{P} = \bigcup_{n_4 \geq 0} \mathbb{P}_{n_4,4}$ into disjoint subcollections with the property that if $\mathbb{I}_{n_4,4} $ is the collection of maximal dyadic intervals in $\left\{ I_{\vec{P}} : \vec{P} \in \mathbb{P}_{n_4, 4} \right\}$ then 

\begin{align*}
Size_4(f_4, \mathbb{P}_{n_4,4})  \lesssim &2^{-n_4} \\
\sum_{I \in  \mathbb{I}_{n_4, 4}} |I| \lesssim &2^{n_4} |E_4|. 
\end{align*}

\end{lemma}
\begin{proof}
Let $\mathbb{P}_{0, 4} = \left\{ \vec{P} \in \mathbb{P} : I_{\vec{P}} \subset \left\{ M 1_{E_4} \gtrsim 1 \right\} \right\}$ and iteratively construct 

\begin{align*}
\mathbb{P}_{n_4, 4} = \left\{ \vec{P} \in \mathbb{P} \cap \left[ \bigcup_{0 \leq m < n_4} \mathbb{P}_{m, 4} \right]^c : I_{\vec{P}} \subset  \left\{ M 1_{E_4} \geq 2^{-n_4} \right\}  \right\}. 
\end{align*}
Using Lemma \ref{L:John-Nirenberg}, it is easy to check the desired properties. 

\end{proof}

The next step is to perform a size-energy stopping time decomposition in $\mathbb{P}_{n_4,4}$ and $\mathbb{P}^d_k$.

\begin{lemma}\label{TDL*}
Fix $d,  k, n_4 \geq 0$. Let $\mathbb{P}^d_{k, n_4} = \mathbb{P}_{n_4, 4} \cap \mathbb{P}^d_k$ and  

\begin{eqnarray*}
\mathbb{I}_{k, n_4}^d = \left\{ I : I = J \cap K~for~some~ J \in \mathbb{I}_{n_4,4}, K \in \mathbb{I}_k^d \right\}.
\end{eqnarray*}
 Then there are two decompositions of $\mathbb{P}^d_{k,n_4}$, namely $\bigcup_{n_1 \geq 0}  \mathbb{P}^{d,1}_{k,n_4,n_1}$ and $ \bigcup_{ \mathfrak{d} \gtrsim -k} \mathbb{P}^{d,2}_{k,n_4, \mathfrak{d}}$ such that 
  
  \begin{eqnarray*}
  Size_1(f_1, \mathbb{P}^{d,1}_{k, n_4, n_1}) \lesssim 2^{-n_1}~and~Size_0^d(f_2, f_3, \mathbb{P}^{d,2}_{k,n_4, \mathfrak{d}}) \lesssim 2^{-\mathfrak{d}}.
  \end{eqnarray*}
   Moreover, both $\mathbb{P}^{d,1}_{k,n_4, n_1}$ and $\mathbb{P}^{d,2}_{k,n_4,\mathfrak{d}}$ can be written as a union of trees, i.e. 

\begin{align}\label{Def:Decomp-2}
\mathbb{P}^{d,1}_{k, n_4 ,n_1} =& \bigcup_{T \in \mathcal{T}^d_{k, n_4,n_1}} \bigcup_{\vec{P} \in T} \vec{P} \\ 
\mathbb{P}^{d,2}_{k, n_4, \mathfrak{d}} =& \bigcup_{T \in \mathcal{T}^{d,2}_{k,n_4, \mathfrak{d}}} \bigcup_{\vec{P} \in T} \vec{P}
\end{align}
where

\begin{eqnarray*}
 \sum_{T \in \mathcal{T}_{k, n_4, n_1}^{d,1}} |I_T| &\lesssim& 2^{2n_1} \sum_{T \in \mathcal{T}_{k, n_4,n_1, *}^d} \sum_{\vec{P} \in T} \left| \langle f_1, \Phi_{P_1, 1} \rangle \right|^2  \\
 \sum_{T \in \mathcal{T}_{k, \mathfrak{d}}^{d,2}} |I_T| &\lesssim & 2^{2 \mathfrak{d}}\sum_{T \in \mathcal{T}_{k,\mathfrak{d},*,I}^{d,2}}  \sum_{\vec{P} \in T}\left| \left \langle  BHT^{\mathbb{Q}^d} (f_2, f_3), \Phi_{P_2,0} \right\rangle \right| ^2\\&+& 2^{2 \mathfrak{d}}\sum_{T \in \mathcal{T}_{k,\mathfrak{d},*,II}^{d,2}}  \sum_{\vec{P} \in T}\left| \left \langle  \sum_{\vec{Q} \in \mathbb{Q}^d: \omega_{Q_3} \supset \supset \omega_{P_2}} \frac{1}{|I_{\vec{Q}}|^{1/2}} \langle f_2, \Phi_{Q_1, 2} \rangle \langle f_3, \Phi_{Q_2, 3} \rangle \Phi_{Q_3, 5}, \Phi_{P_2,0} \right \rangle  \right| ^2 \\ &+&  2^{2 \mathfrak{d}}\sum_{T \in \mathcal{T}_{k,\mathfrak{d},*,II}^{d,2}}  \sum_{\vec{P} \in T}\left| \left \langle  \sum_{\vec{Q} \in \mathbb{Q}^d: |\omega_{Q_3}| \simeq  |\omega_{P_2}|} \frac{1}{|I_{\vec{Q}}|^{1/2}} \langle f_2, \Phi_{Q_1, 2} \rangle \langle f_3, \Phi_{Q_2, 3} \rangle \Phi_{Q_3, 5}, \Phi_{P_2,0} \right \rangle  \right| ^2
\end{eqnarray*}
where $\mathcal{T}^{d,1}_{k, n_1, n_4 ,*} \subset  \mathcal{T}^{d,1}_{k, n_1, n_4} $ and $\mathcal{T}^{d,2}_{k, n_4, \mathfrak{d}, *} := \mathcal{T}^{d,2}_{k, n_4, \mathfrak{d}, * , I} \bigcup \mathcal{T}^{d,2}_{k,n_4, \mathfrak{d}, *, II} \subset  \mathcal{T}_{k, \mathfrak{d}}^{d,2} $, each tree in $\mathcal{T}^{d,1}_{k, n_1,*}$ is a $2$-tree and each tree in $ \mathcal{T}^{d,2}_{k, \mathfrak{d},*} $ is a $1$-tree, and the collections $\mathcal{T}^{d,1}_{k, n_1, n_4,*} , \mathcal{T}^{d,2}_{k, n_4,\mathfrak{d},*,I}, \mathcal{T}^{d,2}_{k, n_4,\mathfrak{d}, *, II}$ can each be written as the union of two chains of strongly disjoint trees as in Definition \ref{Def:StronglyDisjoint}. In particular, 

\begin{eqnarray*}
\mathcal{T}^{d,1}_{k,n_4, n_1,*} &=& \mathcal{T}^{d,2}_{k, n_4, n_1,*, +} \bigcup \mathcal{T}^{d,1}_{k, n_4, n_1,*, -} \\ 
\mathcal{T}_{k, n_4, \mathfrak{d},*}^{d,2}&=& \mathcal{T}_{k, n_4,  \mathfrak{d},*, I, +}^{d,2} \bigcup \mathcal{T}_{k, n_4 \mathfrak{d},*, I, -}^{d,2} \bigcup \mathcal{T}_{k, n_4,\mathfrak{d},*, II, +}^{d,2} \bigcup \mathcal{T}_{k, n_4, \mathfrak{d},*, II, -}^{d,2}
\end{eqnarray*}
where each collection on the right side of the above display is strongly disjoint. 
Lastly, 

\begin{align*}
& \left[ \bigcup_ {T \in \mathcal{T}^{d,1}_{k, n_4, n_1,*}} I_T \right] \cup  \left[ \bigcup_ {T \in \mathcal{T}^{d,2}_{k, n_4, \mathfrak{d},*} }I_T \right]\subset \bigcup_{I \in \mathbb{I}^d_{k,n_4}} I \\  
& \sum_{I \in \mathbb{I}^d_{k, n_4}} |I| \lesssim \min \{ 2^{-N d} 2^k  |E_2|^{1/2}, 2^{n_4} |E_4| \}.
\end{align*}
\end{lemma}
\begin{proof}
Apply the argument localized to each dyadic interval $I \in \mathbb{I}_{k,n_4}^d$ from Lemmas \ref{TDL} and \ref{L:John-Nirenberg}. 
\end{proof}

\vspace{5mm}
\subsubsection{Energy Savings}
We shall now show 

\begin{lemma}
\begin{align}\label{Est:Intermediate}
\sum_{T \in \mathcal{T}_{k, n_4, \mathfrak{d}, *}^{d,2}} |I_T| \lesssim_{\epsilon, \tilde{\epsilon}}  \min \left\{2^{-N d/2} 2^{2 \mathfrak{d}} 2^{-k/2} |E_2|^{1/2}  , 2^{2 \mathfrak{d}} 2^{d} |E_2|^{2-\epsilon} ,  2^{\frac{ \mathfrak{d}}{1-\tilde{\epsilon}}} |E_2|^{1/2} \right\}.
\end{align}
\end{lemma}
\begin{proof}
First note

\begin{align*}
\sum_{T  \in \mathcal{T}_{k, \mathfrak{d}, *}^{d}}  |I_T|  \lesssim&   2^{2 \mathfrak{d}} \sum_{T \in \mathcal{T}_{k, \mathfrak{d}, *,I}^d} \sum_{\vec{P} \in T} \left| \left \langle  BHT^{\mathbb{Q}^d}(f_2, f_3), \Phi_{P_2,0} \right \rangle \right|^2 \\ +& 2^{2 \mathfrak{d}} \sum_{T \in \mathcal{T}_{k, \mathfrak{d}, *, II}^{d}} \sum_{\vec{P} \in T} \left| \left \langle   \sum_{\vec{Q} \in \mathbb{Q}^d: \omega_{Q_3} \supset \supset \omega_{P_2}}  \frac{1}{|I_{\vec{Q}}|^{1/2}} \langle f_2 , \Phi_{Q_1,2} \rangle \langle f_3, \Phi_{Q_2,3} \rangle \Phi_{Q_3,5}, \Phi_{P_2,0} \right \rangle \right|^2  \\+& 2^{2 \mathfrak{d}} \sum_{T \in \mathcal{T}_{k, \mathfrak{d}, *,III}^d} \left| \left \langle  \sum_{\vec{Q} \in \mathbb{Q}^d: |\omega_{Q_3}| \simeq  |\omega_{P_2}|} \frac{1}{|I_{\vec{Q}}|^{1/2}} \langle f_2, \Phi_{Q_1, 2} \rangle \langle f_3, \Phi_{Q_2, 3} \rangle \Phi_{Q_3, 5}, \Phi_{P_2,0} \right \rangle  \right| ^2 \\ 
:=& I + II +III. 
\end{align*}
Using the BHT energy estimate in Lemma \ref{Est:BHTEnergy} on $I$ and the Biest energy estimate in Lemma \ref{Est:BiestEnergy} on $II$ yields 

\begin{align*}
I \lesssim& 2^{2 \mathfrak{d}}\sum_{I \in \mathbb{I} ^d_k }  ||1_{E_2} \chi^M_I ||_4^2 || 1_{E_3} \chi^M_I ||_4^2\\ 
II \lesssim& 2^{2 \mathfrak{d}} \sum_{I \in \mathbb{I}^d_k} || BHT^{\mathbb{Q}^d}(f_2, f_3)\chi^M_I ||_2^2 . 
\end{align*}
For $III$, we majorize according to 
\begin{align*}
III \lesssim& 2^{2 \mathfrak{d}} \sum_{I \in \mathbb{I}^d_k} \sum_{\vec{Q} \in \mathbb{Q}^d: I_{\vec{Q}} \subset I} \frac{\langle f_2, \Phi_{Q_1,2}\rangle|^2 \langle f_3, \Phi_{Q_2, 3}\rangle|^2}{|I_{\vec{Q}}|} \lesssim 2^{2 \mathfrak{d}} \sum_{I \in \mathbb{I}^d_k} ||1_{E_2} \chi^M_I ||_4^2 || 1_{E_3} \chi^M_I ||_4^2. 
\end{align*}
Since  $\sum_{I \in \mathbb{I}^d_{k, n_4} } |I| \lesssim 2^{-N d} 2^{-k}  |E_2|^{1/2}$

\begin{align}\label{Est:BiestEnergy+}
\sum_{T \in \mathcal{T}_{k, \mathfrak{d}, *}^{\tilde{d}}} |I_T| \lesssim & 2^{2 \mathfrak{d}}\sum_{I \in \mathbb{I} ^d_{k, n_4} }  ||1_{E_2} \chi^M_I ||_4^2 || 1_{E_3} \chi^M_I ||_4^2 + 2^{2 \mathfrak{d}} \sum_{I \in \mathbb{I}^d_{k, n_4}} || BHT^{\mathbb{Q}^d}(f_2, f_3)\chi^M_I ||_2^2 \\ \lesssim& 2^{-N d/2} 2^{2 \mathfrak{d}} 2^{-k/2}  |E_2|^{1/2} \nonumber.
\end{align}
Combining Lemmas \ref{Est:BHTloc}, \ref{Est:BHTEnergy} and \ref{Est:BiestEnergy} yields for every $0<\epsilon \ll 1$

\begin{align}\label{Est:l2-redux}
\sum_{T \in \mathcal{T}_{k,n_4, \mathfrak{d}, *}^{d,2}} |I_T|  \lesssim_\epsilon 2^{2 \mathfrak{d}} 2^{d} |E_2|^{2-\epsilon}. 
\end{align}
Moreover, an application of Theorem \ref{l1T} yields for every $0 < \tilde{\epsilon} \ll 1$
\begin{align}\label{Est:l1-redux}
\sum_{T \in \mathcal{T}_{k, n_4,\mathfrak{d}, *}^{d,2}} |I_T| \lesssim_{\tilde{\epsilon}} 2^{\frac{ \mathfrak{d}}{1-\tilde{\epsilon}}} |E_2|^{1/2}. 
\end{align}
Therefore, from estimates \eqref{Est:BiestEnergy+}, \eqref{Est:l2-redux}, and \eqref{Est:l1-redux}, we obtain estimate \eqref{Est:Intermediate}. 
\end{proof}
\begin{lemma}
For all $n_1, n_4, d, k \geq 0$, 
\begin{eqnarray}\label{Est:Intermediate1}
\sum_{T \in \mathcal{T}^{d,1}_{k,n_4, n_1, *}} |I_T| \lesssim  2^{2n_1} \min \left\{ |E_1|,  2^{n_4} |E_4| \right\}.
\end{eqnarray}
\end{lemma}
\begin{proof}
By Lemma \ref{Est:BHTEnergy},

\begin{align*}
\sum_{T \in \mathcal{T}^{d,1}_{k, n_4,n_1, *}}  |I_T| \lesssim 2^{2n_1}||f_1||_{L^2}^2  \lesssim 2^{2n_1} |E_1|.
\end{align*}
Moreover, 

\begin{align*}
\sum_{T \in \mathcal{T}^{d,1}_{k, n_4, n_1}}|I_T| \lesssim&  2^{2n_1} \sum_{I \in \mathbb{I}^d_{k, n_4}} \left| \left| f_1 \chi^M_I \right| \right|_2^2 \\ \lesssim& 2^{2n_1} \sum_{I \in \mathbb{I}^d_{k, n_4}} |I| \\ \lesssim& 2^{2n_1} 2^{n_4} |E_4|. 
\end{align*}
Consequently, estimate \eqref{Est:Intermediate1} holds.
\end{proof}

\subsubsection{Synthesis}
For $d, k, n_1, n_4 \geq 0$ and $\mathfrak{d} \gtrsim  -k$, let $\mathbb{P}^{d}_{k, n_4, n_1, \mathfrak{d}} = \mathbb{P}^{d,1}_{k, n_4, n_1} \cap \mathbb{P}^{d,2}_{k, n_4,\mathfrak{d}}$. Then
the decomposition \eqref{Def:Decomp-2} yields

 \begin{align*}
 &\sum_{ \vec{P} \in  \mathbb{P}^{d}_{k, n_4, n_1, \mathfrak{d}}}\frac{1}{|I_{\vec{P}}|^{1/2}} \langle f_1, \Phi_{P_1,1} \rangle \langle f_4,   \Phi_{P_4,4} \rangle\left \langle  \int_0^1 BHT^{\alpha, \mathbb{Q}^d}_{\omega_{P_2}}(f_2, f_3) d \alpha , \Phi_{P_2,0} \right\rangle \\  =& \sum_{T \in \mathcal{T}_{k,n_4,n_1}^{d,1}} \sum_{ \vec{P} \in T \cap \mathbb{P}^d_{k,n_4,n_1, \mathfrak{d}}} \frac{1}{|I_{\vec{P}}|^{1/2}} \langle f_1, \Phi_{P_1,1} \rangle \langle f_4,   \Phi_{P_4,4} \rangle\left \langle  \int_0^1 BHT^{\alpha, \mathbb{Q}^d}_{\omega_{P_2}}(f_2, f_3) d \alpha , \Phi_{P_2,0} \right\rangle \\ =&  \sum_{T \in \mathcal{T}_{k, n_4, \mathfrak{d}}^{d,2}} \sum_{ \vec{P} \in T \cap \mathbb{P}^d_{k, n_4, n_1, \mathfrak{d}}} \frac{1}{|I_{\vec{P}}|^{1/2}} \langle f_1, \Phi_{P_1,1} \rangle \langle f_4,   \Phi_{P_4,4} \rangle\left \langle  \int_0^1 BHT^{\alpha, \mathbb{Q}^d}_{\omega_{P_2}}(f_2, f_3) d \alpha , \Phi_{P_2,0} \right\rangle. 
\end{align*}
Each element in $ \mathcal{T}_{k,n_4,n_1}^{d,1} \bigcup \mathcal{T}_{k,n_4,\mathfrak{d}}^{d,2}$  is a $1$- or $2$-tree.  Using the tree estimates in Section 4.2.2 and estimate \eqref{Est:Intermediate1} yields

\begin{align*}
& \left| \sum_{T \in \mathcal{T}_{k,n_4,n_1}^{d,1} } \sum_{ \vec{P} \in T \cap \mathbb{P}^d_{k, n_4, n_1,\mathfrak{d}} } \frac{1}{|I_{\vec{P}}|^{1/2}} \langle f_1, \Phi_{P_1,1} \rangle \langle f_4,   \Phi_{P_4,4} \rangle\left \langle  \int_0^1 BHT^{\alpha, \mathbb{Q}^{\tilde{d}}}_{\omega_{P_2}}(f_2, f_3) d \alpha , \Phi_{P_2,0} \right\rangle \right|\\& \lesssim  2^{-n_1} 2^{-n_4} 2^{-\mathfrak{d}} \sum_{T \in \mathcal{T}_{k, n_4, n_1}^{d,1} }  |I_T| \\ &\lesssim  2^{-n_1} 2^{-n_4} 2^{-\mathfrak{d}} \min \left\{ 2^{2n_1} |E_1|, 2^{2n_1} 2^{n_4} |E_4| \right\}. 
\end{align*}
Similarly, the tree estimates in Section 4.2.2 and estimate \eqref{Est:Intermediate} yield
\begin{align*}
& \left| \sum_{T \in \mathcal{T}_{k, n_4,\mathfrak{d}}^{d,2}} \sum_{ \vec{P} \in T \cap \mathbb{P}^d_{k, n_4, n_1, \mathfrak{d}}} \frac{1}{|I_{\vec{P}}|^{1/2}} \langle f_1, \Phi_{P_1,1} \rangle \langle f_4,   \Phi_{P_4,4} \rangle\left \langle  \int_0^1 BHT^{\alpha, \mathbb{Q}^{\tilde{d}}}_{\omega_{P_2}}(f_2, f_3) d \alpha , \Phi_{P_2,0} \right\rangle \right| \\&\lesssim  2^{-n_1} 2^{-n_4} 2^{-\mathfrak{d}} \left[ \sum_{T \in \mathcal{T}_{k, n_4, \mathfrak{d}}^{d,2}}  |I_T| \right] \\ &\lesssim   2^{-n_1}2^{-n_4}  2^{-\mathfrak{d}} \min\left\{ 2^{-N d/2} 2^{2 \mathfrak{d}} 2^{-k/2} |E_2|^{1/2}  , 2^{2 \mathfrak{d}} 2^{d} |E_2|^{2-\epsilon} ,  2^{\frac{ \mathfrak{d}}{1-\tilde{\epsilon}}} |E_2|^{1/2}  \right\} . 
\end{align*}
Hence,
\begin{align*}
& \left|\Lambda_2^{\mathbb{P}, \mathbb{Q}}(f_1, f_2, f_3, f_4) \right| \\ \lesssim & \sum_{d, k, n_1, n_4  \geq 0}  ~\sum_{\mathfrak{d} \gtrsim -k} 2^{-n_1} 2^{-n_4} 2^{-\mathfrak{d}} \times \\ & \min\left\{  2^{2n_1} |E_1|, 2^{2n_1} 2^{n_4} |E_4| ,  2^{-N d/2} 2^{2 \mathfrak{d}} 2^{-k/2} |E_2| ^{1/2} , 2^{2 \mathfrak{d}} 2^{2d} |E_2|^{2-\epsilon} ,  2^{\frac{ \mathfrak{d}}{1-\tilde{\epsilon}}} |E_2| ^{ 1/2} \right\}. 
\end{align*}
We now show generalized restricted type estimates for a sequence of 4-tuples approaching $A_8 = (0, 1, -\frac{1}{2}, \frac{1}{2} )$. Using $0 \leq \theta_1, \theta_2, \theta_3, \theta_4, \theta_5 \leq 1$ to denote the weighting assigned to each component in the above minimum, set $\theta_1 = 0,  \theta_2 = \frac{1}{2} - \epsilon_2, \theta_3 = \epsilon_3, \theta_4 = \frac{1}{2} - \epsilon_4, \theta_5 = \epsilon_2 +\epsilon_4 - \epsilon_3$ 
for some $0< \epsilon_2, \epsilon_3, \epsilon_4 \ll 1$ to be determined. For summability over $k$ and $\mathfrak{d}$, we require

\begin{align*}
\epsilon_3 < & \epsilon_2 + \epsilon_4   \\ 
2 \epsilon_3 + 2 \left(\frac{1}{2} - \epsilon_4\right) + \frac{ \epsilon_2 + \epsilon_4 - \epsilon_3}{1- \tilde{\epsilon}} <&1 \\ 
1- 2 \epsilon_3 - 2 \left( \frac{1}{2} - \epsilon_4 \right) - \frac{ \epsilon_2 + \epsilon_4 - \epsilon_3}{ 1-\tilde{\epsilon}} <&\frac{\epsilon_3}{2}. 
\end{align*}
For fixed $0< \epsilon_4 \ll 1$, let $\epsilon_2 = \frac{\epsilon_4}{2}$ and $\epsilon_3 = \frac{ 3 \epsilon_4}{8}$ so that the first inequality holds. Moreover, there exists a choice of $0< \tilde{\epsilon}\ll 1$ for which the last two inequalities in the above display are satisfied. Indeed, letting $\tilde{\epsilon} \rightarrow 0$, the system becomes

\begin{eqnarray*}
0 < \epsilon_4 - \epsilon_2 - \epsilon_3 < \frac{\epsilon_3}{2}< \frac{\epsilon_2 + \epsilon_4}{2}. 
\end{eqnarray*}
These inequalities are satisfied by our choice of $\epsilon_2, \epsilon_3$ as a function of $\epsilon_4$. Letting $\epsilon_4 \rightarrow 0$ gives us restricted weak type estimates for a sequence of tuples approaching $A_8$. Proving restricted estimates for a sequence approaching $A_9$ is very similar to the argument for $A_8$, so the details are omitted. It is clear that one should set $\theta_1 = \frac{1}{2} - \epsilon_1, \theta_2 = 0, \theta_3 = \epsilon_3, \theta_4 = \frac{1}{2} - \epsilon_4 , 
\theta_5 = \epsilon_1+ \epsilon_4 - \epsilon_3. 
$
for special choices of $0 < \epsilon_1, \epsilon_3 , \epsilon_4 \ll 1$. 
 This ends the proof of Theorem \ref{DT} from which Theorem \ref{MT} follows. 
\end{proof}

\vspace{10mm}
\section{$C^{1,1,1-2}$ Estimates}

Our goal in this section is to prove

\MTV*

\subsection{Reduction to the $\Lambda_3$-Model}
The discretized and localized version of Theorem \ref{MT*} is 

\begin{theorem}\label{DT*}
Every $5$-form of type $\Lambda_3$ as described in Definition \ref{Def:Lambda-3} is generalized restricted type $\vec{\beta}$ for all admissible tuples $\vec{\beta}$ sufficiently close to the extremal points in $\mathbb{B}$.
If $\vec{\beta}$ has a bad index $j$, the restricted type estimate is uniform in the sense that the major subset $E_j^\prime$ can be chosen uniformly in the parameters

\begin{eqnarray*}
\mathbb{P}, \mathbb{Q}, \mathbb{R}, \left\{ \Phi_{P_k, j} \right\}, \left\{ \Phi^\alpha_{Q_k, j} \right\}, \left\{ \Phi_{R_k, j}\right\}. 
\end{eqnarray*}
\end{theorem} 
Before showing Theorem \ref{DT*}, we prove the following
 \begin{prop}\label{MP*}
 Theorem \ref{DT*} implies Theorem \ref{MT*}. 
 \end{prop}
 \begin{proof}
We proceed as in the proof of $C^{1,1,-2}$ by decomposing $\left\{ \xi_1 < \xi_2 < \xi_3< - \frac{\xi_4}{2} \right\}$ into the following regions, which are viewed as subsets of $\left\{ \xi_1 < \xi_2 < \xi_3 < - \frac{\xi_4}{2}\right\}$: 

\begin{align*}
\mathcal{S}_1 =& \left\{ |\xi_1- \xi_2| \gg |\xi_2 - \xi_3| \gg  \left|\xi_3 + \frac{\xi_4}{2}\right| \right\} 
, \mathcal{S}_2 = \left\{| \xi_1 - \xi_2| \gg \left |\xi_3 + \frac{\xi_4}{2}\right| \gg |\xi_2- \xi_3| \right\}  \\ 
\mathcal{S}_3 =& \left\{ |\xi_1 - \xi_2|\gg |\xi_2 - \xi_3| \simeq \left| \xi_3 + \frac{\xi_4}{2} \right| \right\} 
,\mathcal{S}_4 = \left\{ |\xi_1 - \xi_2| \simeq |\xi_2 -\xi_3| \gg \left|\xi_3 + \frac{\xi_4}{2}\right| \right\} \\ 
\mathcal{S}_5 =& \left\{ |\xi _1 - \xi_2| \simeq \left|\xi_3 + \frac{\xi_4}{2}\right|  \gg |\xi_2 - \xi_3| \right\}
,\mathcal{S}_6 = \left\{ |\xi_1 - \xi_2 | \simeq |\xi_2 - \xi_3| \simeq \left|\xi_3 + \frac{\xi_4}{2} \right| \right\}\\
\mathcal{S}_7 =& \left\{ \left| \xi_3 + \frac{\xi_4}{2} \right| \gg |\xi_2 - \xi_3| \gg |\xi_1 - \xi_2| \right\} 
,\mathcal{S}_8 = \left\{ \left| \xi_2 + \frac{\xi_4}{2} \right| \gg |\xi_1 - \xi_2| \gg |\xi_2 - \xi_3| \right\} \\ 
\mathcal{S}_9 =& \left\{ \left| \xi_3 + \frac{\xi_4}{2} \right| \simeq |\xi_2 - \xi_3 | \gg |\xi_1 - \xi_2| \right\} 
,\mathcal{S}_{10} = \left\{ |\xi_3 + \frac{\xi_4}{2}| \gg |\xi_2 - \xi_3| \simeq |\xi_1 - \xi_2| \right\} \\
\mathcal{S}_{11} =& \left\{ |\xi_2 - \xi_3| \gg |\xi_1 - \xi_2 | \right\} \cap \left\{ |\xi_2 - \xi_3| \gg\left |\xi_3 + \frac{\xi_4}{2} \right|  \right\}. \\ 
\end{align*}
By the methods used in \cite{MR3329857}, estimates for $C^{1,1,1,-2}$ follow from estimates for multipliers with symbols  adapted to each of the above $11$ regions of $\mathbb{R}^4$ and which have the nested structure of a paracomposition. To say a symbol $m$ is adapted to a region $\mathcal{R}_j$ means that it is supported in $\mathcal{R}_j$ and satisfies the derivative estimate

\begin{eqnarray*}
| \partial^{\vec{\alpha}} m(\vec{\xi}) | \leq A \frac{1}{dist(\vec{\xi}, \Sigma)^{|\vec{\alpha}|}} \qquad \forall \vec{\xi} \in \mathbb{R}^4
\end{eqnarray*}
for sufficiently many multi-indices $\vec{\alpha} \in \mathbb{Z}_{\geq 0}^4$ where $\Sigma := \{ \xi_1 = \xi_2 = \xi_3 = - \xi_4/2\}$. 
A nested structure means here that the models 
It is tedious but straightforward to verify using the results in \cite{MR3329857} that the estimates appearing in Theorem \ref{MT*} hold for multipliers with symbols adapted to regions $\{ \mathcal{S}_j\}_{j=3}^{11}$. Therefore, our problem reduces to producing estimates for multipliers with symbols adapted to regions $\mathcal{S}_1$ and $\mathcal{S}_2$. Moreover, a closer look at the argument in \cite{MR3329857} reveals that for $\mathcal{S}_1$ and $\mathcal{S}_2$ it suffices to prove the estimates in Theorem \ref{MT*} for all symbols adapted to $\mathcal{S}_1$ and $\mathcal{S}_2$ and which are identically equal to one on sets $\mathcal{S}_1^*$ and $\mathcal{S}_2^*$, which have the same shape as $\mathcal{S}_1$ and $\mathcal{S}_2$ respectively and are given by 
\begin{eqnarray*}
\mathcal{S}^*_1 &=& \left\{ |\xi_1- \xi_2| \gg \gg |\xi_2 - \xi_3| \gg \gg \left|\xi_3 + \frac{\xi_4}{2}\right| \right\} \\ \mathcal{S}^*_2 &=& \left\{| \xi_1 - \xi_2| \gg \gg \left |\xi_3 + \frac{\xi_4}{2}\right| \gg \gg |\xi_2- \xi_3| \right\}
\end{eqnarray*}
where $\gg \gg$ signifies an implicit constant much greater than that appearing in $\gg$. 

Furthermore, it suffices by symmetry and the interpolation results from Chapter 3 of \cite{MR2199086} to produce generalized restricted type estimates for the discretized model of generic symbols adapted to $\mathcal{S}_1,$ say. To this end, we first recall the equality (7.19) in Section 7.2 of \cite{MR3052499} valid for any $N \in \mathbb{R}$: 

\begin{eqnarray}\label{Identity}
1_{(-\infty, N]} (\xi) =  \lim_{M \rightarrow \infty} \frac{C}{2M} \int_{-M}^M \int_0^1 \sum_{\omega \in \mathcal{D}_{k, \eta}} \hat{\Phi}_{\omega_l}(\xi) 1_{\omega_r}(N) ~d k d \eta
\end{eqnarray}
where $\mathcal{D}_{k, \eta} = 2^k D_0^1+\eta$ and $D_0^1$ is the collection of standard dyadic cubes, for any interval $\omega = [c_\omega - |\omega|/2, c_\omega+ |\omega|/2]$, $\omega_l=[c_\omega-|\omega|/2, c_\omega]$ and $\omega_r=[c_\omega, c_{\omega} + |\omega|/2]$, and  $\hat{\Phi}_{\omega}(\xi) = \hat{\Phi}(\frac{\xi - c_{\omega}}{|\omega|})$ for some $\hat{\Phi} \in C^\infty[-1/4, 1/4]$ and $\hat{\Phi} \equiv 1$ on $[-3/16, 3/16]$. 
We now construct $\Phi_{\mathcal{S}_1}:\mathbb{R}^5 \rightarrow \mathbb{R}$ given by
\begin{eqnarray*}
&& \Phi_{\mathcal{S}_1}(\xi_1, \xi_2, \xi_3, \xi_4, \xi_5) \\ &=&  \sum_{\vec{k} ,\vec{l} \in \mathbb{Z}^3} \sum_{\vec{\sigma} ,\vec{\gamma} \in \{0, \frac{1}{3}, \frac{2}{3}\}^3} 
 \sum_{\omega_{\vec{P}} \in \mathcal{P}^{\vec{\gamma}}}~ \sum_{\omega_{\vec{Q}} \in \mathcal{Q}^{\vec{\sigma}}: |\omega_{\vec{Q}}| \ll |\omega_{\vec{P}}|}c_{\vec{k}} d_{\vec{l}}  \\ && \lim_{M \rightarrow \infty} \frac{C}{2M} \int_{-M}^M \int_0^1 \sum_{\omega \in \mathcal{D}_{k, \eta}: |\omega| \gg |\omega_{\vec{P}}|} \hat{\Phi}_{\omega_l}(\xi_1) 1_{\omega_r}\left(\frac{c_{\omega_{P_2}}-c_{\omega_{P_3}}}{2}\right)\hat{\Phi}_{-\tilde{\omega}_l}(\xi_5) \hat{\Phi}_{\omega_0}(-\xi_1-\xi_5) ~d k d \eta  \\ &\times&  \hat{\eta}^{l_1, \gamma_1}_{\omega_{P_1}, 6}(\xi_1 + \xi_5) \hat{\eta}^{l_2, \gamma_2}_{\omega_{P_2}, 2}(\xi_2) \hat{\eta}^{l_3, \gamma_3}_{\omega_{P_3},7}(\xi_3 + \xi_4) \hat{\eta}^{k_3, \sigma_3}_{-\omega_{Q_3}, 1}(\xi_1 + \xi_2 + \xi_5) \hat{\eta}^{k_1, \sigma_1}_{\omega_{Q_1}, 3}(\xi_3) \hat{\eta}^{k_2, \sigma_2}_{\omega_{Q_2}, 4} (\xi_4), 
\end{eqnarray*}
where each $\omega_{\vec{P}}=(\omega_{P_1}, \omega_{P_2}, \omega_{P_3})$ is adapted to $\{\xi_1+ \xi_2=0 ; \xi_3=0\}$, each $\omega_{\vec{Q}} = (\omega_{Q_1}, \omega_{Q_2}, \omega_{Q_3})$ is adapted to $\left\{ \xi_1 = - \xi_2/2 =\xi_3 \right\}$, $\tilde{\omega}$ is the interval with the same center as $\omega$ and  1.5 times the length and $\omega_0$ has the same length as $\tilde{\omega}$ centered at the origin, i.e.  
$\tilde{\omega}= [c_{\omega}-\frac{3}{4}|\omega|,c_{\omega}+\frac{3}{4}|\omega|]$ and $\omega_0 = [-\frac{3}{4} |\omega|, \frac{3}{4}|\omega|]$. The bump functions with dyadic shifts $\vec{\sigma}$ and $\vec{\gamma}$ and oscillation parameters $\vec{k}$ and $\vec{l}$ are taken directly from the decompositions given in \eqref{Def:Splitting1} and \eqref{Def:Splitting2} from \S{4}. We next introduce the collection of frequency rectangles $\mathcal{R}_{k, \eta} = \left\{ (\tilde{\omega}_l, - \tilde{\omega}_l, \omega_0): \omega \in \mathcal{D}_{k, \eta} \right\}$. For $\vec{\omega}:=(\omega_{R_1}, \omega_{R_2}, \omega_{R_3}) = (\omega_l, - \tilde{\omega}_l , \omega_0) \in \mathcal{R}_{k, \eta}$, let $\widetilde{\omega_{R_1}} = \omega_r$. 
Set $\phi_{\mathcal{S}_1}(\xi_1, \xi_2, \xi_3, \xi_4) = \Phi_{\mathcal{S}_1}(\xi_1, \xi_2, \xi_3, \xi_4, -\xi_1-\xi_2 -\xi_3-\xi_4)$ and observe that $\phi_{\mathcal{S}_1} \equiv 1$ on $\mathcal{S}_1^*$ for an appropriate choice of implicit constants. We may now dualize:

\begin{eqnarray*}
&& \int_\mathbb{R} T_{\phi_{\mathcal{S}_1}} (f_1, f_2, f_3, f_4)(x) f_5(x) dx \\
&=&  \lim_{M \rightarrow \infty} \frac{1}{2M} \int_{-M} ^M \int_0^1 \sum_{\vec{k} ,\vec{l} \in \mathbb{Z}^3} \sum_{\vec{\sigma} ,\vec{\gamma} \in \{0, \frac{1}{3}, \frac{2}{3}\}^3} 
 \sum_{\omega_{\vec{P}} \in \mathcal{P}^{\vec{\gamma}}}~ \sum_{\omega_{\vec{Q}} \in \mathcal{Q}^{\vec{\sigma}}: |\omega_{\vec{Q}}| \ll |\omega_{\vec{P}}|} ~ \sum_{\omega_{\vec{R}} \in \mathcal{R}^{k, \eta}: \widetilde{\omega_{R_1}} \ni (c_{\omega_{P_2}}-c_{\omega_{P_3}})/2} c_{\vec{k}} d_{\vec{l}} \\ &\times& \int_\mathbb{R}  \left[  f_1*  \Phi_{\omega_{R_1}}f_5* \Phi_{\omega_{R_2}} \right] * \Phi_{\omega_{R_3}}  \left[f_2*\eta^{l_2, \gamma_2} _{\omega_{P_2}, 2}  \left[ f_3* \eta^{k_1, \sigma_1}_{\omega_{Q_1}, 3} f_4 * \eta^{k_2, \sigma_2}_{\omega_{Q_2}, 4}  \right] *\eta^{l_3, \gamma_3}_{\omega_{P_3}, 7} \right]*\eta_{\omega_{P_1},6}^{l_1, \gamma_1}  dx dk d \eta.   \end{eqnarray*}
The last line of the above display can then be discretized using the procedure of Section 6.1 of \cite{MR3052499} to yield a limit of averages and sums over various rapidly decaying terms of generic forms of type $\Lambda_3$:

\begin{eqnarray*}
\Lambda_3^{\mathbb{P}, \mathbb{Q}, \mathbb{R}}(\vec{f})=&&  \sum_{\vec{P} \in \mathbb{P}} \frac{1}{|I_{\vec{P}}|^{1/2}} \left \langle \sum_{\vec{R} \in \mathbb{R}: \widetilde{\omega_{R_1}} \ni (c_{\omega_{P_2}}-c_{\omega_{P_3}})/2} \frac{1}{|I_{\vec{R}}|^{1/2}}  \langle f_1, \Phi_{R_1, 1} \rangle \langle f_5, \Phi_{R_2, 5}\rangle \Phi^{n-l}_{R_3, 0} , \Phi^{lac}_{P_1,6}\right  \rangle \\ && \hspace{20mm} \times \langle f_2, \Phi_{P_2, 2} \rangle \left\langle  \int_0^1  BHT_{\omega_{P_3}}^{\alpha, \mathbb{Q}} (f_3, f_4) d \alpha,  \Phi_{P_3, 7}\right \rangle . 
\end{eqnarray*}
where $\mathbb{P}$  is a collection of tri-tiles for which $\omega_{\vec{P}}=(\omega_{P_1}, \omega_{P_2}, \omega_{P_4})$ is adapted to $\{\xi_1 =-\xi_2; \xi_3=0\}$, $\mathbb{Q}$ is a rank-1 collection of tri-tiles for which $\omega_{\vec{Q}}=(\omega_{Q_1}, \omega_{Q_2}, \omega_{Q_3})$ is adapted to $\{ \xi_1=-\xi_2/2= -\xi_3 \}$, and $\mathbb{R}$ is a collection of generalized tri-tiles  for which $\omega_{\vec{R}} =(\omega_1, \omega_2 , \omega_3) =  (\omega_1 , - \omega_1, - |\omega_1|/2, |\omega_1|/2)$. Generalized tiles and tri-tiles are given at the end of Definition \ref{TT}. Each $\Phi_{P_k,j}$ is a wave packet on the tile $P_k$, and each $\Phi_{R_k,j}$ is a wave packet on the tile $R_k$ according to Definition \ref{WP}. Therefore, the proof of Theorem \ref{MT*} reduces to obtaining generalized restricted type estimates for all $\Lambda_3$-models. 

\end{proof}



The remainder of \S{5} is dedicated to the proof of Theorem \ref{DT*}. 
\subsection{Generalized Restricted Type Estimates near $B_1, B_2, B_3$}
\subsubsection{Tile Decomposition}
Fix tri-tile collections $\mathbb{P}, \mathbb{Q},$ and $\mathbb{R}$ once and for all. For convenience, we shall subsequently use $f_j$ to denote $f_j^\prime$ for $j=1,2,3,4,5$ as described in Definition \ref{Def:RestrType}. Furthermore, we assume that $|E_5|=1$ by rescaling and that the collections $\mathbb{P}, \mathbb{Q}, \mathbb{R}$ are sparse. For each $\tilde{d} \geq 0$ set $\mathbb{Q}^{\tilde{d}} := \left\{ \vec{Q} \in \mathbb{Q} : 1 + \frac{ dist(I_{\vec{Q}}, \tilde{\Omega}^c) }{|I_{\vec{Q}}|} \simeq 2^{\tilde{d}} \right\}$ and set 

\begin{eqnarray}
\tilde{\Omega} &=& \left\{ M1_{E_1} \gtrsim |E_1| \right\} \bigcup \left\{ M1_{E_2} \gtrsim |E_2| \right\} \bigcup \left\{ M1_{E_3} \gtrsim |E_3| \right\} \bigcup \left\{ M1_{E_4} \gtrsim |E_4| \right\} \\
\Omega_1^0&=&\left\{ M \left[ \int_0^1  BHT^{\alpha, \mathbb{Q}^{0}}(f_3, f_4) d\alpha\right] \gtrsim  |E_3|^{1/2} |E_4|^{1/2} \right\} \\ 
\Omega_1^{\tilde{d}} &= & \left\{ M \left(  \left[ \int_0^1  \sum_{\vec{Q} \in \mathbb{Q}^{\tilde{d}}} \frac{ |\langle f_3, \Phi^\alpha_{Q_1, 3} \rangle \langle f_4, \Phi^\alpha_{Q_2,4 } \rangle|}{|I_{\vec{Q}}|} \chi^M_{I_{\vec{Q}}} d \alpha\right]^2 \right) \gtrsim 2^{2\tilde{d}} |E_3| |E_4| \right\}.
\end{eqnarray}
Lastly, construct 
\begin{eqnarray*}
\Omega= \tilde{\Omega} \bigcup \Omega_1^0 \bigcup_{\tilde{d} \geq  1} \Omega_1^{\tilde{d}}. 
\end{eqnarray*}
Then for large enough implicit constants, $|\Omega| \leq 1/2$ and $\tilde{E}_5 := E_5 \cap \Omega^c$ is a major subset of $E_5$ since $|E_5|=1$. Now let $\mathbb{P}^d := \left\{ \vec{P} \in \mathbb{P}: 1+ \frac{dist(I_{\vec{P}}, \Omega^c)}{|I_{\vec{P}}|} \simeq 2^d \right\}$.  Assuming $|f_1|\leq 1_{E_1}, |f_2| \leq 1_{E_2}, |f_3| \leq 1_{E_3}, |f_4|\leq 1_{E_4}, |f_5|\leq 1_{E_5 \cap \Omega^c}$, recall that our task in this section is to obtain the estimate $|\Lambda_3^{\mathbb{P}, \mathbb{Q}, \mathbb{R}}(f_1, f_2, f_3, f_4, f_5)| \lesssim |E_1|^{\alpha_1} |E_2|^{\alpha_2} |E_3|^{\alpha_3} |E_4|^{\alpha_4}$ for $(\alpha_1, \alpha_2, \alpha_3, \alpha_4)$ in a small neighborhood near an extremal point $\vec{\beta}  \in \{B_1, B_2, B_3\}$. 

For any subcollection of tri-tiles $\tilde{\mathbb{P}} \subset \mathbb{P}$, let 

\begin{align*}
Size_2\left(f_2, \tilde{ \mathbb{P}}\right) :=& \sup_{T \subset \tilde{\mathbb{P}}} \frac{1}{|I_T|^{1/2}} \left( \sum_{\vec{P} \in T} |\langle f_2, \Phi_{P_2,2} \rangle|^2 \right)^{1/2}
\end{align*}
and 
\begin{align*}
& Size^{\tilde{d}}_7(f_3, f_4, \tilde{\mathbb{P}}) \\ :=& \sup_{T \subset \tilde{\mathbb{P}}} \frac{1}{|I_T|^{1/2}} \left( \sum_{\vec{P} \in T}\left | \left \langle \int_0^1 BHT^{\alpha, \mathbb{Q}^{\tilde{d}}} (f_3, f_4) d \alpha, \Phi_{P_3, 7} \right \rangle \right|^2 \right. \\  +& \left. \left | \left \langle \int_0^1 \sum_{\vec{Q} \in \mathbb{Q}^{\tilde{d}}: \omega_{Q_3} \supset \supset \omega_{P_2}} \frac{1}{|I_{\vec{Q}}|^{1/2}} \langle f_2, \Phi^\alpha_{Q_1, 2} \rangle \langle f_3, \Phi^\alpha_{Q_2,3} \rangle \Phi^\alpha_{Q_3,5} d\alpha, \Phi_{P_3,7} \right \rangle  \right|^2 \right. \\ +& \left.\left | \left \langle \int_0^1 \sum_{\vec{Q} \in \mathbb{Q}^{\tilde{d}}: |\omega_{Q_3}| \simeq  |\omega_{P_2}|} \frac{1}{|I_{\vec{Q}}|^{1/2}} \langle f_2, \Phi^\alpha_{Q_1, 2} \rangle \langle f_3, \Phi^\alpha_{Q_2,3} \rangle \Phi^\alpha_{Q_3,5} d\alpha, \Phi_{P_3,7} \right \rangle  \right|^2  \right)^{1/2}
\end{align*}
where the supremum arising in the definition of the $2$-size is over all $3$-trees and the supremum arising in the definition of the $7$-size is over all $2$-trees. As before, both sizes generate decompositions of $\mathbb{P}^d$ for each $\tilde{d} \geq 0$, namely $\bigcup_{n_2 \geq N_2(d)}  \mathbb{P}^d_{n_2, 2}$ and $ \bigcup_{ \mathfrak{d} \geq N_3(d, \tilde{d})} \mathbb{P}^{d, \tilde{d}}_{\mathfrak{d}, 3}$ such that $Size_2(f_2, \mathbb{P}_{n_2, 2}^d) \lesssim 2^{-n_2}$ and $Size_7^{\tilde{d}}(f_3, f_4,\mathbb{P}_{\mathfrak{d},3}^{d, \tilde{d}}) \lesssim 2^{-\mathfrak{d}}$. Moreover, $\mathbb{P}^d_{n_2, 2}$ and $\mathbb{P}_{\mathfrak{d}, 3}^{d, \tilde{d}}$ can each be written as a union of trees, i.e. 

\begin{align*}
\mathbb{P}^d_{n_2, 2} =& \bigcup_{T \in \mathcal{T}_{n_2, 2}^d} \bigcup_{\vec{P} \in T} \vec{P} \\ 
\mathbb{P}^{d, \tilde{d}}_{\mathfrak{d}, 3} =& \bigcup_{T \in \mathcal{T}^{d, \tilde{d}}_{\mathfrak{d}, 3}} \bigcup_{\vec{P} \in T} \vec{P}, 
\end{align*}
such that $\sum_{T \in \mathcal{T}_{n_2, 2}^d} |I_T| \lesssim 2^{2n_2} \sum_{T \in \mathcal{T}_{n_2, 2,*}^d} \sum_{\vec{P} \in T} \left| \langle f_1, \Phi_{P_2, 2} \rangle \right|^2$ and  

\begin{align*}
& \sum_{T \in \mathcal{T}_{\mathfrak{d}, 3}^{d, \tilde{d}}} |I_T| \\  \lesssim &2^{2 \mathfrak{d}}\sum_{T \in \mathcal{T}_{\mathfrak{d}, 3,*}^{d, \tilde{d}}}  \sum_{\vec{P} \in T}\left| \left \langle \int_0^1 BHT^{\alpha , \mathbb{Q}^{\tilde{d}}} (f_3, f_4) d \alpha, \Phi_{P_3,7} \right\rangle \right| ^2\\ +& 2^{2 \mathfrak{d}}\sum_{T \in \mathcal{T}_{\mathfrak{d}, 3,*}^{d, \tilde{d}}}  \sum_{\vec{P} \in T}\left | \left \langle \int_0^1 \sum_{\vec{Q} \in \mathbb{Q}^{\tilde{d}}: \omega_{Q_3} \supset \supset \omega_{P_2}} \frac{1}{|I_{\vec{Q}}|^{1/2}} \langle f_2, \Phi^\alpha_{Q_1, 2} \rangle \langle f_3, \Phi^\alpha_{Q_2,3} \rangle \Phi^\alpha_{Q_3,5} d\alpha, \Phi_{P_3,7} \right \rangle  \right|^2  \\+&
2^{2 \mathfrak{d}}\sum_{T \in \mathcal{T}_{\mathfrak{d}, 3,*}^{d, \tilde{d}}}  \sum_{\vec{P} \in T} \left | \left \langle \int_0^1 \sum_{\vec{Q} \in \mathbb{Q}^{\tilde{d}}: |\omega_{Q_3}| \simeq  |\omega_{P_2}|} \frac{1}{|I_{\vec{Q}}|^{1/2}} \langle f_2, \Phi^\alpha_{Q_1, 2} \rangle \langle f_3, \Phi^\alpha_{Q_2,3} \rangle \Phi^\alpha_{Q_3,5} d\alpha, \Phi_{P_3,7} \right \rangle  \right|^2 
\end{align*}
where $\mathcal{T}_{n_2, 2,*}^d \subset  \mathcal{T}_{n_2, 2}^d $, $\mathcal{T}_{\mathfrak{d}, 3,*}^{d, \tilde{d}} \subset  \mathcal{T}_{\mathfrak{d}, 3}^{d, \tilde{d}} $, each tree in $\mathcal{T}_{n_2, 2,*}^d$ is a $3$-tree and each tree in $ \mathcal{T}_{\mathfrak{d}, 3,*}^{d, \tilde{d}}$ is a $2$-tree, and the collections $\mathcal{T}_{n_2, 2,*}^d , \mathcal{T}_{\mathfrak{d}, 3,*}^{d, \tilde{d}}$ can each be written as the union of a strongly $2$-disjoint and $3$-disjoint subcollections respectively. We denote this last property by

\begin{align*}
\mathcal{T}_{n_1, 1,*}^d =& \mathcal{T}_{n_1, 1,*, +}^d \bigcup \mathcal{T}_{n_1, 1,*, -}^d \\ 
\mathcal{T}_{\mathfrak{d}, 2,*}^{d, \tilde{d}}=& \mathcal{T}_{\mathfrak{d}, 2,*, +}^{d, \tilde{d}} \bigcup \mathcal{T}_{\mathfrak{d}, 2,*, -}^{d, \tilde{d}}.
\end{align*}
Similar to before, we construct $\mathbb{P}^{d, \tilde{d},*}_{n_2, \mathfrak{d}} = \mathbb{P}^{d, \tilde{d}}_{n_2, 2} \cap \mathbb{P}^{d, \tilde{d}}_{\mathfrak{d}, 3}$
and record the following tree estimates. 

\subsubsection{Tree Estimates}
If $T\subset \mathbb{P}^{d, \tilde{d},*}_{n_2, \mathfrak{d}}$ is a $3$-tree, we may use Lemma \ref{L:Tree-4} to conclude that for any $0 < \theta <1$

\begin{eqnarray*}
&&\left| \sum_{\vec{P} \in T} \frac{1}{|I_{\vec{P}}|^{1/2}}  \left \langle \sum_{\vec{R} \in \mathbb{R}: \widetilde{\omega_{R_1}} \ni (c_{\omega_{P_2}}-c_{\omega_{P_3}})/2)} \frac{ \langle f_1, \Phi_{R_1, 1} \rangle \langle f_5, \Phi_{R_2, 5}\rangle}{|I_{\vec{R}}|^{1/2}} \Phi^{n-l}_{R_3, 0} , \Phi^{lac}_{P_1,6}\right  \rangle \right.  \\  && \hspace{20mm} \times \left.  \langle f_2, \Phi_{P_2, 2} \rangle \left\langle  \int_0^1  BHT_{\omega_{P_3}}^{\alpha, \mathbb{Q}} (f_3, f_4) d \alpha,  \Phi_{P_3, 7}\right \rangle \right|\\& \leq&\frac{ \left( \sum_{\vec{P} \in T} \left|  \left \langle \sum_{\vec{R} \in \mathbb{R}: \widetilde{\omega_{R_1}} \ni ( c_{\omega_{P_2}}-c_{\omega_{P_3}})/2} \frac{ \langle f_1, \Phi_{R_1, 1} \rangle \langle f_5, \Phi_{R_2, 5}\rangle}{|I_{\vec{R}}|^{1/2}} \Phi^{n-l}_{R_3, 0} , \Phi^{lac}_{P_1,6}\right  \rangle \right|^2  \right)^{1/2}}{|I_T|^{1/2}}  \\  && \times \frac{  \left( \sum_{\vec{P} \in T} |\langle f_2, \Phi_{P_2,2} \rangle|^2 \right)^{1/2}}{|I_T|^{1/2}}  \sup_{\vec{P} \in T}\left[  \frac{\left| \left \langle \int_0^1 BHT^{\alpha, \mathbb{Q}^{\tilde{d}}}_{\omega_{P_3}}(f_3, f_4) d \alpha, \tilde{\Phi}^{\infty}_{P_3,7} \right\rangle \right|}{|I_{\vec{Q}}|} \right] |I_T|\\ &\lesssim_\theta &2^{-N d(1-\theta) } |E_1|^{\theta} 2^{-n_2}2^{-\frak{d}} |I_T|. 
\end{eqnarray*}
If $T \subset \mathbb{P}^{d, \tilde{d},*}_{n_2, \mathfrak{d}} $ is a $2$-tree, then for any $0 < \theta <1$

\begin{eqnarray*}
&&\left| \sum_{\vec{P} \in T}\frac{1}{|I_{\vec{P}}|^{1/2}}  \left \langle \sum_{\vec{R} \in \mathbb{R}: \widetilde{\omega_{R_1}} \ni  (c_{\omega_{P_2}}-c_{\omega_{P_3}})/2} \frac{ \langle f_1, \Phi_{R_1, 1} \rangle \langle f_5, \Phi_{R_2, 5}\rangle}{|I_{\vec{R}}|^{1/2}} \Phi^{n-l}_{R_3, 0} , \Phi^{lac}_{P_1,6}\right  \rangle \right. \\ && \left. \hspace{20mm}  \times \langle f_2, \Phi_{P_2, 2} \rangle \left\langle  \int_0^1  BHT_{\omega_{P_3}}^{\alpha, \mathbb{Q}} (f_3, f_4) d \alpha,  \Phi_{P_3, 7}\right \rangle \right|\\ &\lesssim&   \left[ \sup_{\vec{P} \in T} \frac{  \left|  \left \langle \sum_{\vec{R} \in \mathbb{R}: \widetilde{\omega_{R_1}} \ni (c_{\omega_{P_2}}-c_{\omega_{P_3}})/2} \frac{ \langle f_1, \Phi_{R_1, 1} \rangle \langle f_5, \Phi_{R_2, 5}\rangle}{|I_{\vec{R}}|^{1/2}} \Phi^{n-l}_{R_3, 0} , \Phi^{lac}_{P_1,6}\right  \rangle \right|}{|I_{\vec{P}}|^{1/2}} \right]    \\ &\times& \left( \sum_{\vec{P} \in T} \frac{|\langle f_4, \Phi_{\vec{P},4}^{lac} \rangle|^2}{|I_T|} \right)^{1/2} \left( \sum_{\vec{P} \in T} \frac{ \left| \left \langle \int_0^1 BHT^{\alpha, \mathbb{Q}^{\tilde{d}}}_{\omega_{P_2}}(f_3, f_4) d \alpha, \Phi_{P_3,7} \right\rangle\right|^2}{|I_T|}\right)^{1/2} |I_T| \\ &\lesssim_\theta& 2^{-N d(1-\theta)} |E_1|^\theta 2^{-n_2}2^{-\frak{d}} |I_T|.
\end{eqnarray*}
\subsubsection{Size Restrictions}

\begin{lemma}\label{L:Tree-4}
Fix $d, \tilde{d}, n_2, \mathfrak{d} \geq 0$ and let $T \subset \mathbb{P}^{d, \tilde{d},*}_{n_2, \mathfrak{d}}$ be a $2$- or $3$-tree. 
Then for some $M \gg 1$ and every $0 < \theta <1$

\begin{eqnarray*}
&& \left( \frac{1}{|I_T|} \sum_{\vec{P} \in T} \left|  \left \langle \sum_{\vec{R} \in \mathbb{R}: \widetilde{\omega_{R_1}} \ni (c_{\omega_{P_2}}-c_{\omega_{P_3}})/2} \frac{ \langle f_1, \Phi_{R_1, 1} \rangle \langle f_5, \Phi_{R_2, 5}\rangle}{|I_{\vec{R}}|^{1/2}} \Phi^{n-l}_{R_3, 0} , \Phi^{lac}_{P_1,6}\right  \rangle \right|^2 \right)^{1/2} \\& \lesssim_\theta & \left[ \sup_{\vec{P} \in T}  \frac{1}{|I_{\vec{P}}|} \int_{E_1} \chi^M_{I_{\vec{P}}} dx \right]^\theta \left[ \sup_{\vec{P} \in T}  \frac{1}{|I_{\vec{P}}|} \int_{E_5} \chi^M_{I_{\vec{P}}} dx \right]^{1-\theta}. 
\end{eqnarray*}
\end{lemma}
\begin{proof}
It suffices to assume that $T$ contains its top $P_T$ and moreover, that $E_1, E_5 \subset 5 I_T$. This reduction follows from the proof of the standard Biest size estimate in \cite{MR2127985}.  We may now use the John-Nirenberg inequality in (3.15) in Section 3.3 of \cite{MR3052499} to reduce to showing for any $2$- or $3$-tree $T$ that

\begin{eqnarray}\label{Est:John-Nirenberg1}
&& \left| \left| \left( \sum_{\vec{P} \in T }  \left| \left \langle  \sum_{\vec{R} \in \mathbb{R}: \widetilde{\omega_{R_1}} \ni (c_{\omega_{P_2}}-c_{\omega_{P_3}})/2} \frac{ \langle f_1, \Phi_{R_1, 1} \rangle \langle f_5, \Phi_{R_2, 5}\rangle}{|I_{\vec{R}}|^{1/2}} \Phi^{n-l}_{R_3, 0} , \Phi^{lac}_{P_1,6}\right  \rangle \right|^2\frac{1_{I_{\vec{P}}}}{|I_{\vec{P}}|} \right)^{1/2} \right| \right|_{L^{1,\infty}(I_T)} 
\end{eqnarray}
is bounded above by $C_\theta \left( \int_{E_1} \chi^M_{I_T} dx\right)^\theta \left( \int_{E_5} \chi^M_{I_T} dx \right)^{1-\theta}.$ Assume that $T$ is a $2$-tree without loss of generality. Then we split $T$ into $T^+ \bigcup T^-$, where 

\begin{eqnarray*}
T^+&=& \left\{ \vec{P} \in T: P_3 \lesssim^+ P_{T,3} \right\} \\
T^- &=& \left\{ \vec{P} \in T: P_3 \lesssim^- P_{T,3} \right\}
\end{eqnarray*}
and prove \eqref{Est:John-Nirenberg1} for $T^+$ and $T^-$. In fact, we prove \eqref{Est:John-Nirenberg1} only for $T^+$ again without loss of generality. Let $\mathbb{R}(T^+) = \left\{ \vec{R} \in \mathbb{R} : \exists \vec{P} \in T^+~s.t.~ \widetilde{\omega_{R_1}} \ni (c_{\omega_{P_2}}-c_{\omega_{P_3}})/2 , |\omega_{\vec{P}}| << |\omega_{\vec{R}}|\right\}$. For each $\vec{R} \in \mathbb{R}$, there exists an $L^\infty$-normalized bump function $\hat{\Phi}_{\vec{R}, T^+}$ centered at $0$ such that $\langle \Phi^{n-l}_{R_3,0} * \Phi_{\vec{R}, T^+} , \Phi_{P_1, 6} \rangle =\langle \Phi^{n-l}_{R_3,0} , \Phi_{P_1, 6} \rangle$ for all pairings $(\vec{P}, \vec{R})$ for which $|\omega_{\vec{R}}| >>> | \omega_{\vec{P}}|$ and $\widetilde{\omega_{R_1}} \not \ni  (c_{\omega_{P_2}}-c_{\omega_{P_3}})/2$.  The inner product $\langle \Phi^{n-l}_{R_3,0} * \Phi_{\vec{R}, T^+} , \Phi_{P_1, 6} \rangle$ vanishes for all others pairings $(\vec{P}, \vec{R}).$  Moreover, note that $\langle \Phi^{n-l}_{R_3,0} , \Phi_{P_1,6} \rangle = 0$ whenever it is not the case that $|\omega_{\vec{P}}| <<  |\omega_{\vec{R}}|$. From these properties, it is now straightforward to bound \eqref{Est:John-Nirenberg1}  by a constant times 

\begin{align*}
& \left| \left| \left( \sum_{\vec{P} \in T^+ }  \left| \left \langle  \sum_{\vec{R} \in \mathbb{R}(T^+)} \frac{ \langle f_1, \Phi_{R_1, 1} \rangle \langle f_5, \Phi_{R_2, 5}\rangle}{|I_{\vec{R}}|^{1/2}} \Phi^{n-l}_{R_3, 0} , \Phi^{lac}_{P_1,6}\right  \rangle \right|^2\frac{1_{I_{\vec{P}}}}{|I_{\vec{P}}|} \right)^{1/2} \right| \right|_{L^{1,\infty}(I_{T^+})}  \\ +& \left| \left| \left( \sum_{\vec{P} \in T^+ }  \left| \left \langle  \sum_{\vec{R} \in \mathbb{R}(T^+)} \frac{ \langle f_1, \Phi_{R_1, 1} \rangle \langle f_5, \Phi_{R_2, 5}\rangle}{|I_{\vec{R}}|^{1/2}} \Phi^{n-l}_{R_3, 0} * \Phi_{\vec{R},T^+} , \Phi^{lac}_{P_1,6}\right  \rangle \right|^2\frac{1_{I_{\vec{P}}}}{|I_{\vec{P}}|} \right)^{1/2} \right| \right|_{L^{1,\infty}(I_{T^+})} \\ +& \left| \left| \left( \sum_{\vec{P} \in T^+ }  \left| \left \langle  \sum_{\substack{ \vec{R} \in \mathbb{R}(T^+): |\omega_{\vec{R}}| \simeq |\omega_{\vec{P}}|\\ \widetilde{\omega_{R_1}} \not \ni  (c_{\omega_{P_2}}-c_{\omega_{P_3}})/2} } \frac{ \langle f_1, \Phi_{R_1, 1} \rangle \langle f_5, \Phi_{R_2, 5}\rangle}{|I_{\vec{R}}|^{1/2}} \Phi^{n-l}_{R_3, 0} , \Phi^{lac}_{P_1,6}\right  \rangle \right|^2\frac{1_{I_{\vec{P}}}}{|I_{\vec{P}}|} \right)^{1/2} \right| \right|_{L^{1,\infty}(I_{T^+})}\\ =:& ~I + II + III. 
\end{align*}
We first bound the contributions of $I$ and $II$. Linearizing the $\ell^2(T^+)$-sum in terms $I$ and $II$ yields Calder\'on-Zygmund operators, which maps $L^1(\mathbb{R}) \rightarrow L^{1,\infty}(\mathbb{R})$. As each $\Phi_{\vec{R},T^+}$ is an $L^1$-normalized bump function, these two summands in the above display may be bounded by a constant times

\begin{eqnarray}\label{Est:Intermediate2}
\left| \left|    \sum_{\vec{R} \in \mathbb{R}(T^+)} \frac{ \langle f_1, \Phi_{R_1, 1} \rangle \langle f_5, \Phi_{R_2, 5}\rangle}{|I_{\vec{R}}|^{1/2}} \Phi^{n-l}_{R_3, 0} \right| \right|_{L^1(\mathbb{R})}. 
\end{eqnarray}
Furthermore, as the wave packets $\Phi_{R_1,1}$ and $\Phi_{R_2,5}$ have finer time localization than the scale $|I_T|$, it follows that for all $0 < \theta <1$ 

\begin{align*}
\left| \left| \left(\sum_{\vec{R} \in \mathbb{R}(T^+)}  |\langle f_1, \Phi_{R_1,1} \rangle|^2 \frac{1_{I_{\vec{R}}}}{|I_{\vec{R}}|} \right)^{1/2} \right| \right|_{L^{\frac{1}{\theta}} (\mathbb{R})} \lesssim& \left(\int_{E_1} \chi_{I_T} ^M dx \right)^\theta \\ 
\left| \left| \left(\sum_{\vec{R} \in \mathbb{R}(T^+)}  |\langle f_5, \Phi_{R_5,2} \rangle|^2 \frac{1_{I_{\vec{R}}}}{|I_{\vec{R}}|} \right)^{1/2} \right| \right|_{L^{\frac{1}{1-\theta}} (\mathbb{R})} \lesssim& \left(\int_{E_5} \chi_{I_T} ^M dx \right)^{1-\theta}. 
\end{align*}
An application of Cauchy-Schwarz on the $\mathbb{R}(T^+)$-sum followed by H\"older's inequality allows us to dominate $I+II$ by $C \left( \int_E \chi^M_{I_T} dx\right)^\theta \left( \int_{E_4} \chi^M_{I_T} dx \right)^{1-\theta}.$ 

To bound term $III$, it suffices to observe

\begin{align*}
III \lesssim &  \left| \left| \left(  \sum_{\vec{R} \in \mathbb{R}(T^+)}  \frac{ |\langle f_1, \Phi_{R_1, 1} \rangle|^2 | \langle f_5, \Phi_{R_2, 5}\rangle|^2}{|I_{\vec{R}}|} \frac{1_{I_{\vec{R}}}}{|I_{\vec{R}}|} \right)^{1/2} \right| \right|_{L^1(\mathbb{R})} \\ \lesssim& \left| \left| \left(  \sum_{\vec{R} \in \mathbb{R}(T^+)}   |\langle f_1, \Phi_{R_1, 1} \rangle|^2 \frac{1_{I_{\vec{R}}}}{|I_{\vec{R}}|} \right)^{1/2} \left(  \sum_{\vec{R} \in \mathbb{R}(T^+)}   |\langle f_5, \Phi_{R_5, 2} \rangle|^2 \frac{1_{I_{\vec{R}}}}{|I_{\vec{R}}|} \right)^{1/2}\right| \right|_{L^1(\mathbb{R})} .
\end{align*}
Another application of H\"older's inequality and the square function estimates allow us to dominate \eqref{Est:Intermediate2} by $C_\theta \left( \int_{E_1} \chi^M_{I_T} dx\right)^\theta \left( \int_{E_5} \chi^M_{I_T} dx \right)^{1-\theta}$ as desired.

\end{proof}
An immediate consequence of Lemmas \ref{L:John-Nirenberg} and \ref{L:Tree-4} is 
\begin{lemma}
Fix $d, \tilde{d}, n_2, \mathfrak{d}$ such that $\mathbb{P}^{d, \tilde{d},*}_{n_2, \mathfrak{d}}$ is nonempty. Then 

\begin{align*}
2^{-n_2} \lesssim& 2^d |E_2| \\ 
2^{-\mathfrak{d}} \lesssim& 2^{-N(\tilde{d}-d) } |E_3|^{1/2} |E_4|^{1/2}. 
\end{align*}
As a consequence, we have that in the decomposition 

\begin{align*}
\mathbb{P} \times \mathbb{Q}= \bigcup_{d, \tilde{d} \geq 0} \bigcup _{n_2 \geq N_2(d)} \bigcup_{\mathfrak{d} \geq N_3(d, \tilde{d})} \mathbb{P}^{d, \tilde{d},*}_{n_2, \mathfrak{d}} \times \mathbb{Q}^{\tilde{d}} 
\end{align*}
$2^{-N_2(d)} \lesssim 2^d|E_2|$ and $2^{-N_3(d, \tilde{d})} \lesssim 2^{-N(\tilde{d}- d)} |E_3|^{1/2} |E_4|^{1/2}.$

\end{lemma}
\begin{proof}
The argument is identical to the proof of Lemma \ref{L:BiestSize} and therefore omitted. 
\end{proof}

\subsubsection{Synthesis}
As the $\ell^2$ and $\ell^1$ energy estimates for $\Lambda_3^{\mathbb{P}, \mathbb{Q}, \mathbb{R}}(f_1, f_2, f_3, f_4, f_5)$ are essentially  identical to the $\Lambda_2^{\mathbb{P}, \mathbb{Q}}(f_1, f_2, f_3, f_4)$ case, it suffices to assemble all the pieces. Using Theorem \ref{DT} as a guide, the reader may check that for all $0 \leq \theta_1 , \theta_2, \theta_3 \leq 1$ 

\begin{align*}
& \left|\Lambda_3^{\mathbb{P}, \mathbb{Q}, \mathbb{R}} (f_1, f_2, f_3, f_4, f_5 ) \right| \\  \lesssim &
\sum_{d, \tilde{d} \geq 0} \sum_{n_2 \geq N_2(d)} \sum_{\mathfrak{d} \geq N_3(d, \tilde{d})} 2^{-Nd (1-\theta)} 2^{d \theta} |E_1|^{\theta} 2^{-n_2} 2^{-\mathfrak{d}}  \\& \hspace{20mm} \times \min \left\{ 2^{2n_2} |E_2|, 2^{2 \mathfrak{d}} 2^{2 \tilde{d}} |E_3|^{1+\gamma} |E_4|^{2-\gamma}, 2^{\frac{\mathfrak{d}}{1-\tilde{\epsilon}}} |E_3|^{1/2} |E_4|^{1/2} \right\} \\ \lesssim& \sum_{d, \tilde{d} \geq 0} \sum_{n_2 \geq N_2(d)} \sum_{\mathfrak{d} \geq N_3(d, \tilde{d})} 2^{-Nd (1-\theta)} 2^{d \theta}  2^{-n_2} 2^{-\mathfrak{d}}\\   &  \hspace{20mm} \times 2^{2 n_2 \theta_2} 2^{2 \mathfrak{d} \theta_2} 2^{2 \tilde{d} \theta_2} |E_1|^{\theta}|E_3|^{(1+\tilde{\theta}) \theta_2} |E_4|^{(2-\tilde{\theta}) \theta_2} 2^{\frac{ \mathfrak{d} \theta_3}{1-\tilde{\epsilon}}} |E_3|^{\theta_3/2} |E_4|^{\theta_3/2} \\ \leq & \sum_{d, \tilde{d} \geq 0} \sum_{n_2 \geq N_2(d)} \sum_{\mathfrak{d} \geq N_3(d, \tilde{d})} 2^{-Nd (1-\theta)/2}2^{-n_2(1-2 \theta_1)} 2^{-\mathfrak{d}\left[ 1-2 \theta_2-  \frac{\theta_3}{1-\tilde{\epsilon}}\right]}   \\ &  \hspace{20mm} \times |E_1|^\theta |E_3|^{(1+\gamma) \theta_2+\frac{\theta_3}{2}} |E_4|^{(2-\gamma) \theta_2+\frac{\theta_3}{2}} . 
\end{align*} 
For appropriate choices of $0< \epsilon_1, \epsilon_2 , \tilde{\epsilon} , \gamma \ll 1$ and $\theta \simeq 1$, produce estimates for a sequence of tuples approaching $B_1 = \left(1, 1, \frac{1}{2}, \frac{1}{2}, - 1 \right)$ by setting  $ \theta_1 =\epsilon_1, \theta_2 =\epsilon_2, \theta_3 =1-\epsilon_1- \epsilon_2$. For $B_2 = \left( 1, \frac{1}{2} , \frac{1}{2} , 1 , -1 \right)$, set $ \theta_1 =1/2-\epsilon_1, \theta_2 =\frac{1}{2}-\epsilon_2, \theta_3 =\epsilon_1 + \epsilon_2$. Lastly, for $B_3= \left(1,  \frac{1}{2}, 1, \frac{1}{2} - 1\right)$, set $\theta =1-\gamma, \tilde{\theta}= 1, \theta_1 = \frac{1}{2}-\epsilon_1, \theta_2 = \frac{1}{2}-\epsilon_2, \theta_3 =\epsilon_1 + \epsilon_2$. 

\subsection{Generalized Restricted Type Estimates near $B_4, B_5, B_6, B_7, B_8, B_9, B_{10}$}
The model sum $\Lambda_3^{\mathbb{P}, \mathbb{R}, \mathbb{Q}}$ is symmetric in positions $1$ and $5$, and so generalized restricted estimates near $B_1,B_2, B_3$ ensures generalized restricted estimates near $B_4, B_5, B_6$. Generalized restricted type estimates near $B_7, B_8, B_9, B_{10}$ follow from the fact that for every $0\leq \theta_1, \theta_2, \theta_3, \theta_4, \theta_5 \leq 1$ such that $\theta_1 + \theta_2 + \theta_3 + \theta_4 + \theta_5=1$, 

\begin{align*}
& \left|\Lambda_3^{\mathbb{P}, \mathbb{Q}, \mathbb{R}} (f_1,f_21_{\Omega^c} , f_3, f_4, f_5) \right| \\& \sum_{d, \tilde{d} \geq 0} \sum_{n_2 \geq \bar{N}_2(d)} \sum_{\mathfrak{d} \geq N_3(d, \tilde{d})}  2^d |E_1|^{\theta} |E_5|^{1-\theta} 2^{-n_2} 2^{-\mathfrak{d}}  \\& \hspace{20mm} \times \min \left\{ 2^{2n_2} |E_2|, 2^{2 \mathfrak{d}} 2^{2 \tilde{d}} |E_3|^{1+\gamma} |E_4|^{2-\gamma}, 2^{\frac{\mathfrak{d}}{1-\tilde{\epsilon}}} |E_3|^{1/2} |E_4|^{1/2} \right\}\\ & \lesssim  \sum_{d, \tilde{d} \geq 0} \sum_{n_2 \geq N_2(d)} \sum_{\mathfrak{d} \geq N_3(d, \tilde{d})}2^d 2^{2 \tilde{d} \theta_2}  2^{-n_2(1-2 \theta_1)} 2^{-\mathfrak{d}\left[ 1-2 \theta_2-  \frac{\theta_3}{1-\tilde{\epsilon}}\right]} \\ &  \hspace{20mm} \times   |E_1|^{1-\theta} |E_5|^{\theta}  |E_3|^{(1+\gamma) \theta_2+\frac{\theta_3}{2}} |E_4|^{(2-\gamma) \theta_2+ \frac{\theta_3}{2}}.  
\end{align*}
For appropriate choices of $0< \epsilon_1, \epsilon_2 , \tilde{\epsilon} \ll 1$ and $ 0 < \theta, \gamma <1$ , we may produce estimates for a sequence of tuples approaching $B_7 = \left(0  , -\frac{3}{2}, \frac{1}{2}, 1, 1  \right)$ by setting $\theta \simeq 1, \gamma \simeq 0,  \theta_1 =1/2-\epsilon_1, \theta_2 =1/2 - \epsilon_2, \theta_3 = \epsilon_1 + \epsilon_2$. For $B_8= \left( 1, -\frac{3}{2}, \frac{1}{2}, 1, 0 \right),$ set  $ \theta \simeq 0, \gamma \simeq 0, \theta_1 =1/2-\epsilon_1, \theta_2=1/2 - \epsilon_2, \theta_3= \epsilon_1 + \epsilon_2$. For $B_9= \left( 0, - \frac{3}{2}, 1, \frac{1}{2}, 1 \right)$, set $\theta \simeq 1, \gamma \simeq 1, \theta_1=1/2 - \epsilon_1, \theta_2 =1/2 -\epsilon_2, \theta_3 = \epsilon_1 +\epsilon_2$. For $ B_{10} = \left( 1 , -\frac{3}{2}, 1, \frac{1}{2}, 0 \right)$, set  $\theta \simeq 0, \gamma \simeq 1, \theta_1 =1/2 - \epsilon_1, \theta_2=1/2 - \epsilon_2, \theta_3 =\epsilon_1 + \epsilon_2$. 

\subsection{Generalized Restricted Type Estimates near $B_{11}, B_{12}, B_{13}, B_{14}, B_{15}, B_{16}$}
By modifying the adjoint tile decomposition for $\Lambda_2^{\mathbb{P},\mathbb{Q}}$, it is now straightforward to observe 

\begin{align*}
& \left|\Lambda_3^{\mathbb{P}, \mathbb{Q}, \mathbb{R}}(f_1, f_2, f_3 1_{\Omega^c}, f_4,f_5) \right| \\ &\lesssim \sum_{d,  n_1, n_2, n_5,k  \geq 0}\sum_{\mathfrak{d} \gtrsim -k} 2^{-n_1(1-\theta)} 2^{-n_5 \theta} 2^{-n_2} 2^{-n_4} 2^{-\mathfrak{d}} \\ & \hspace{20mm} \times  \min\left\{  2^{2n_2} |E_2|, 2^{2n_2} 2^{n_1(1-\theta)} 2^{n_5 \theta (1-\theta_2)}  |E_1|^{1-\theta} |E_5|^\theta ,   \right. \\ &\left. \hspace{30mm} 2^{2 \mathfrak{d}} 2^{-k/2} 2^{-N d} |E_4| ^{1/2}, 2^{2 \mathfrak{d}} 2^{2d} |E_4|^{2-\epsilon} ,  2^{\frac{ \mathfrak{d}}{1-\tilde{\epsilon}}} |E_4| ^{ 1/2} \right\} \\  &\lesssim \sum_{d,  n_1, n_2, n_5,k  \geq 0}\sum_{\mathfrak{d} \gtrsim -k} 2^{-Nd \theta_3} 2^{2d \theta_4} 2^{-n_1(1-\theta)(1-\theta_2)} 2^{-n_5\theta (1-\theta_2)} 2^{-n_2(1-2(\theta_1 + \theta_2))}  2^{-n_4}  \\ &  \hspace{20mm} \times 2^{-\mathfrak{d}(1-2 (\theta_3+\theta_4) -\frac{\theta_5}{1-\tilde{\epsilon}} )} 2^{-k \theta_3/2}  |E_1|^{(1-\theta) \theta_2} |E_5|^{\theta \theta_2} |E_2|^{\theta_1}  |E_4|^{\theta_3/2 + (2-\epsilon)\theta_4 +\theta_5/2} 
\end{align*}
 Using $0 \leq \theta_1, \theta_2, \theta_3, \theta_4, \theta_5 \leq 1$ to denote the weightings assigned to each term in the above minimum, we may deduce suitable weak type estimates in a neighborhood of $B_{11} = \left( 0 , \frac{1}{2} , -\frac{1}{2} ,1,0\right)$  by taking $\tilde{\epsilon} \simeq 0,\epsilon \simeq 0,  \theta=\frac{1}{2},  \theta_1 =\epsilon_1, \theta_2 = \frac{1}{2}-\epsilon_2, \theta_3 =\epsilon_3 , \theta_4 =\frac{1}{2}-\epsilon_4, \theta_5 =\epsilon_2 + \epsilon_4-\epsilon_1-\epsilon_3$. For estimates near $B_{12} = \left( \frac{1}{2}, 0, -\frac{1}{2}, 1, 0 \right) $, take $\tilde{\epsilon} \simeq 0,\epsilon \simeq 0,\theta \simeq 0 , \theta_1 =\epsilon_1, \theta_2 =\frac{1}{2}-\epsilon_2 , \theta_3 =\epsilon_3 , \theta_4 =\frac{1}{2}-\epsilon_4, \theta_5 =\epsilon_2+\epsilon_4-\epsilon_1-\epsilon_3$. For estimates near $B_{13} = \left( 0, 0, -\frac{1}{2}, 1, \frac{1}{2} \right)$, take $\tilde{\epsilon} \simeq 0,\epsilon \simeq 0, \theta \simeq 1, \theta_1 =\epsilon_1, \theta_2 =\frac{1}{2}-\epsilon_2 , \theta_3 =\epsilon_3 , \theta_4 =\frac{1}{2}-\epsilon_4, \theta_5 =\epsilon_2+\epsilon_4-\epsilon_1-\epsilon_3$. Some care has to be taken to ensure summability over $k \geq 0$. That this is possible is again straightforward, and so details are left to the reader. Generalized restricted type estimates near $B_{14} = \left( 0, \frac{1}{2}, 1, -\frac{1}{2}, 0 \right) , B_{15} = \left( \frac{1}{2}, 0, 1, -\frac{1}{2}, 0 \right), B_{16} = \left( 0,0, 1, - \frac{1}{2}, \frac{1}{2} \right)$ are obtained by symmetry with $B_{11}, B_{12}, B_{13}$. This ends the proof of Theorem \ref{DT*} from which Theorem \ref{MT*} follows.

\nocite{*}
\bibliography{MER}{}
\bibliographystyle{plain}

\Addresses
\end{document}